\documentclass[11pt]{amsart}   
\numberwithin{equation}{section}
\title[]{Existence of Hermitian-Yang-Mills metrics under conifold transitions}
\author{Ming-Tao Chuan} 

\address{Department of Mathematics\\ Harvard University\\ Cambridge,
MA 02138\\ USA} \email{mtchuan@math.harvard.edu}

\newtheorem{theorem}{Theorem}[section]
\newtheorem{definition}[theorem]{Definition}
\newtheorem{lemma}[theorem]{Lemma}

\newtheorem{proposition}[theorem]{Proposition}
\newtheorem{corollary}[theorem]{Corollary}

\usepackage{setspace, hyperref}

     \def\qed{\qedmark\medbreak}%
\def\qedmark{{\enspace\vrule height 6pt width 5pt depth 1.5pt}}%
\begin{document}
\bibliographystyle{plain}
\maketitle

\begin{abstract}
  We first study the degeneration of a sequence of Hermitian-Yang-Mills metrics with respect to a sequence of balanced metrics on a Calabi-Yau threefold $\hat{X}$ that degenerates to the balanced metric constructed by Fu, Li and Yau \cite{FLY} on the complement of finitely many (-1,-1)-curves in $\hat{X}$. Then under some assumptions we show the existence of Hermitian-Yang-Mills metrics on bundles over a family of threefolds $X_t$ with trivial canonical bundles obtained by performing conifold transitions on $\hat{X}$.
\end{abstract}

\section{Introduction}

This paper is about the existence problem for Hermitian-Yang-Mills metrics on holomorphic vector bundles with respect to balanced metrics, when conifold transitions are performed on the base Calabi-Yau threefolds.

The construction of canonical geometric structures on manifolds and vector bundles has always been a very important problem in differential geometry, especially in K$\ddot{\text{a}}$hler geometry. A class of manifolds which are the main focus in this direction is the K$\ddot{\text{a}}$hler Calabi-Yau manifolds\footnote{In this paper, by a Calabi-Yau manifold we mean a complex manifold with trivial canonical bundle which may or may not be K$\ddot{\text{a}}$hler, and what is usually called a Calabi-Yau manifold will now be a K$\ddot{\text{a}}$hler Calabi-Yau manifold.}, i.e., K$\ddot{\text{a}}$hler manifolds with trivial canonical bundles.  The Calabi conjecture which was solved by Yau \cite{Y} in 1976 states that in every K$\ddot{\text{a}}$hler class of a K$\ddot{\text{a}}$hler Calabi-Yau manifold there is a unique representative which is Ricci-flat. 

After the solution of the Calabi conjecture, K$\ddot{\text{a}}$hler Calabi-Yau manifolds have undergone rapid developments, and the moduli spaces of K$\ddot{\text{a}}$hler Calabi-Yau threefolds gradually became one of the most important area of study. In the work of Todorov \cite{To} and Tian \cite{Tian1} the smoothness of the moduli spaces of K$\ddot{\text{a}}$hler Calabi-Yau manifolds in general dimensions was proved. In the complex two dimensional case, the moduli space of K3 surfaces is known to be a 20-dimensional complex smooth irreducible analytic space, with the algebraic K3 surfaces occupying a 19-dimensional reducible analytic subvariety with countable irreducible components \cite{Ko} \cite{Ty} \cite{M}. The global properties of the moduli spaces of K$\ddot{\text{a}}$hler Calabi-Yau threefolds remain much less understood. 

However, there was the proposal by Miles Reid \cite{Reid} which states that the moduli spaces of all Calabi-Yau threefolds can be connected by means of taking birational transformations and smoothings on the Calabi-Yau threefolds. This idea, later dubbed as ``Reid's Fantasy'', was checked for a huge number of examples in \cite{C1}\cite{CGGK}. The processes just mentioned are called geometric transitions in general, and the main focus in this paper is the most studied example, namely the conifold transition, which was first considered by Clemens \cite{Cl} in 1982 and later caught the attention of the physicists starting the late 1980's. It is described as follows. Let $\hat{X}$ be a smooth Calabi-Yau threefold containing a collection of mutually disjoint (-1,-1)-curves $C_1,...,C_l$, i.e., rational curves $C_i \cong \mathbb{P}^1$ with normal bundles in $\hat{X}$ isomorphic to $\mathcal{O}_{\mathbb{P}^1}(-1)\oplus \mathcal{O}_{\mathbb{P}^1}(-1)$. One can contract the $C_i$'s to obtain a space $X_0$ with ordinary double points, and then under certain conditions given by Friedman, $X_0$ can be smoothed and one obtains a family of threefolds $X_t$ with trivial canonical bundles. 

Even when $\hat{X}$ is K$\ddot{\text{a}}$hler, the manifolds $X_t$ may be non-K$\ddot{\text{a}}$hler, and it was proved in \cite{FLY} that they nevertheless admit balanced metrics, which we denote by $\tilde{\omega}_t$. In general, a Hermitian metric $\omega$ on a complex $n$-dimensional manifold is balanced if $d(\omega^{n-1})=0$. K$\ddot{\text{a}}$hler metrics are obviously balanced metrics, but, unlike the K$\ddot{\text{a}}$hler case, the existence of balanced metrics is preserved under birational transformations \cite{AB}. Moreover, if the manifold satisfies the $\partial\bar{\partial}$-lemma, then the aforementioned existence is also preserved under small deformations \cite{Wu}. What \cite{FLY} shows is that it is also preserved under conifold transitions provided $\hat{X}$ is K$\ddot{\text{a}}$hler Calabi-Yau.

In this paper we would like to push further the above result on the preservation of geometric structures after conifold transitions. Consider a pair $(\hat{X},\mathcal{E})$ where $\hat{X}$ is a K$\ddot{\text{a}}$hler Calabi-Yau threefold with a K$\ddot{\text{a}}$hler metric $\omega$, and $\mathcal{E}$ is a holomorphic vector bundle endowed with a Hermitian-Yang-Mills metric with respect to $\omega$. Denote the contraction of exceptional rational curves mentioned above by $\pi:\hat{X} \rightarrow X_0$. From the point of view of metric geometry, such a contraction can be seen as a degeneration of Hermitian metrics on $\hat{X}$ to a metric which is singular along the exceptional curves. In fact, following the methods in \cite{FLY}, one can construct a family of smooth balanced metrics $\{\hat{\omega}_a\}_{a>0}$ on $\hat{X}$ such that $\hat{\omega}_a^2$ and $\omega^2$ differ by $\partial\bar{\partial}$-exact forms and, as $a\rightarrow 0$, $\hat{\omega}_a$ converges to a metric $\hat{\omega}_0$ which is singular along the exceptional curves. The metric $\hat{\omega}_0$ can also be viewed as a smooth metric on $X_{0,sm}$, the smooth part of $X_0$.

We have the following result which is the first main theorem.
\begin{theorem}
 \label{th:main1}
    Let $\mathcal{E}$ be an irreducible holomorphic vector bundle over a K$\ddot{\text{a}}$hler Calabi-Yau threefold $(X,\omega)$ such that $c_1(\mathcal{E})=0$ and $\mathcal{E}$ is trivial on a neighborhood of the exceptional rational curves $C_i$. Suppose $\mathcal{E}$ is endowed with a HYM metric w.r.t. $\omega$. 
    
    Then there exists a HYM metric $H_0$ on $\mathcal{E}|_{X_{0,sm}}$ with respect to $\hat{\omega}_0$, and there is a decreasing sequence $\{a_i\}_{i=1}^\infty$ converging to 0, such that there is a sequence $\{H_{a_i}\}_{i=1}^\infty$ of Hermitian metrics on $\mathcal{E}$ converging weakly in the $L^p_2$-sense, for all $p$, to $H_0$ on each compactly embedded open subset of $X_{0,sm}$, where each $H_{a_i}$ is HYM with respect to $\hat{\omega}_{a_i}$. 
\end{theorem}

Suppose that one can smooth the singular space $X_0$ to $X_t$, and that the bundle $\pi_*\mathcal{E}$ fits in a family of holomorphic bundles $\mathcal{E}_t$ over $X_t$, i.e., the pair $(X_0, \pi_*\mathcal{E})$ can be smoothed to $(X_t,\mathcal{E}_t)$. We ask the question of whether a Hermitian-Yang-Mills metric with respect to the balanced metric $\tilde{\omega}_t$ exists on the bundle $\mathcal{E}_t$.
Note that the condition that $\mathcal{E}$ is trivial in a neighborhood of the exceptional rational curves $C_i$ implies that the bundles $\mathcal{E}_t$ would be trivial in a neighborhood of the vanishing cycles. Also note that $c_1(\mathcal{E}_t)=0$ for any $t\neq 0$.
 
We now state the second main theorem of this paper.
\begin{theorem}
 \label{th:main2}
Let $(\hat{X},\omega)$ be a smooth K$\ddot{\text{a}}$hler Calabi-Yau threefold and $\pi:\hat{X} \rightarrow X_0$ be a contraction of mutually disjoint (-1,-1)-curves. Let $\mathcal{E}$ be an irreducible  holomorphic vector bundle over $\hat{X}$ with $c_1(\mathcal{E})=0$ that is trivial in a neighborhood of the exceptional curves of $\pi$, and admits a Hermitian-Yang-Mills metric with respect to $\omega$. Suppose that the pair $(X_0,\pi_*\mathcal{E})$ can be smoothed to a family of pairs $(X_t,\mathcal{E}_t)$ where $X_t$ is a smooth Calabi-Yau threefold and $\mathcal{E}_t$ is a holomorphic vector bundle on $X_t$. 

Then for $t\neq 0$ sufficiently small, $\mathcal{E}_t$ admits a smooth Hermitian-Yang-Mills metric with respect to the balanced metric $\tilde{\omega}_t$ constructed in \cite{FLY}.
\end{theorem}

For irreducible holomorphic vector bundles over a K$\ddot{\text{a}}$hler manifold, the existence of Hermitian-Yang-Mills metrics corresponds to the slope stability of the bundles. For proofs of this correspondence, see \cite{Don1}\cite{Don2}\cite{UY}. On a complex manifold endowed with a balanced metric, or more generally a Gauduchon metric, i.e., a Hermitian metric $\omega$ satisfying $\partial\bar{\partial}(\omega^{n-1})=0$, one can still define the slopes of bundles and hence the notion of slope stability. Under this setting, Li and Yau \cite{LY1} proved the same correspondence. 

Another motivation for considering stable vector bundles over non-K$\ddot{\text{a}}$hler manifolds comes from physics. K$\ddot{\text{a}}$hler Calabi-Yau manifolds have always played a central role in the study of Supersymmetric String Theory, a theory that holds the highest promise so far concerning the unification of the fundamental forces of the physical world. Among the many models in Supersymmetric String Theory, the Heterotic String models \cite {G}\cite{W} require not only a manifold with trivial canonical bundle but a stable holomorphic vector bundle over it as well. Besides using the K$\ddot{\text{a}}$hler Calabi-Yau threefolds as the internal spaces, Strominger also suggested to use a model allowing nontrivial torsions in the metric. In \cite{St}, he proposed the following system of equations for a pair $(\omega,H)$ consisting of a Hermitian metric $\omega$ on a Calabi-Yau threefold $X$ and a Hermitian metric $H$ on a vector bundle $\mathcal{E} \rightarrow X$ with $c_1(\mathcal{E})=0$: 
\begin{equation}
 \label{s1}
		F_H\wedge \omega^2=0;\,\,\,F_H^{0,2}=F_H^{2,0}=0;
\end{equation}
\begin{equation}
 \label{s2}
		\sqrt{-1}\partial\bar{\partial}\omega=\frac{\alpha}{4}(\text{tr}(R_{\omega}\wedge R_{\omega})-\text{tr}(F_H\wedge F_H));
\end{equation}		
\begin{equation}
 \label{s3}
				d^*\omega =\sqrt{-1}(\bar{\partial}-\partial)\ln \Arrowvert \Omega \Arrowvert_{\omega};
\end{equation}
where $R_{\omega}$ is the full curvature of $\omega$ and $F_H$ is the Hermitian curvature of $H$. The equations (\ref{s1}) is simply the Hermitian-Yang-Mills equations for $H$. Equation (\ref{s2}) is named the Anomaly Cancellation equation derived from physics. In \cite{LY2} it was shown that equation (\ref{s3}) is equivalent to another equation showing that $\omega$ is conformally balanced:
\begin{equation*}
	d(\Arrowvert \Omega \Arrowvert_{\omega}\omega^2)=0.
\end{equation*}
It is mentioned in \cite{FLY} that this system should be viewed as a generalization of Calabi Conjecture for the case of non-K$\ddot{\text{a}}$hler Calabi-Yau manifolds. 

The system, though written down in 1986, was first shown to have non-K$\ddot{\text{a}}$hler solutions only in 2004 by Li and Yau \cite{LY2} using perturbation from a K$\ddot{\text{a}}$hler solution. The first solutions to exist on manifolds which are never K$\ddot{\text{a}}$hler are constructed by Fu and Yau \cite{FY}. The class of threefolds they consider are the $T^2$-bundles over K3 surfaces constructed by Goldstein and Prokushkin \cite{GP}. Some non-compact examples have also been constructed by Fu, Tseng and Yau \cite{FTY} on $T^2$-bundles over the Eguchi-Hanson space. More solutions are found in a recent preprint \cite{AG} using the perturbation method developed in \cite{LY2}.

The present paper can also be viewed as a step following \cite{FLY} in the investigation of the relation between the solutions to Strominger's system on $\hat{X}$ and those on $X_0$ and $X_t$.

This paper is organized as follows:

Section 2 sets up the conventions and contains more background information of conifold transitions and Hermitian-Yang-Mills metrics over vector bundles. Moreover, the construction of balanced metrics in \cite{FLY} is described in more details necessary for later discussions.

In Section 3 the uniform coordinate systems on $X_0$ and on $X_t$ are introduced, which are needed to show a uniform control of the constants in the Sobolev inequalities and elliptic regularity theorems.

In Section 4 Theorem \ref{th:main1} is proved, and several boundedness results of the HYM metric $H_0$ in that theorem are discussed. 

In Section 5 a family of approximate Hermitian metrics $H_t$ on $\mathcal{E}_t$ are constructed, and some estimates on their mean curvatures are established.

Section 6 describes the contraction mapping setup for the HYM equation on the bundle $\mathcal{E}_t$. Theorem \ref{th:main2} is proved here.

Section 7 deals with a proposition left to be proved from Section 6.\\[0.3cm]

\noindent \textbf{Acknowledgements} The author would like to thank his thesis advisor Professor S.-T. Yau for constant supports and valuable comments. The author is also grateful to Professor C. Taubes and Professor J. Li for helpful discussions, and to Professor J.-X. Fu for useful comments during the preparation of this work.

\section{Backgrounds}

\subsection{Conifold transitions}

Let $\hat{X}$ be a K$\ddot{\text{a}}$hler Calabi-Yau threefold with a K$\ddot{\text{a}}$hler metric denoted by $\omega$. Let $\bigcup C_i$ be a collection of $(-1,-1)$-curves in $\hat{X}$, and let $X_0$ be the threefold obtained by contracting $\bigcup C_i$, so $\hat{X}$ is a small resolution of $X_0$. $X_0$ has ordinary double points, which are the images of the curves $C_i$ under the contraction. There is a condition given by Friedman which relates the smoothability of the singular space $X_0$ to the classes $[C_i]$ of the exceptional curves in $\hat{X}$:
\begin{theorem} \cite{Fri1}\cite{Fri2}
		If there are nonzero numbers $\lambda_i$ such that the class 
\begin{equation}
 \label{eq:i-1}
		\sum_i \lambda_i[C_i]=0
\end{equation}
in $H^2(\hat{X},\Omega_{\hat{X}}^2)$ then a smoothing of $X_0$ exists, i.e., there is a 4-dimensional complex manifold $\mathcal{X}$ and a holomorphic projection $\mathcal{X} \rightarrow \Delta$ to the disk $\Delta$ in $\mathbb{C}$ such that the general fibers are smooth and the central fiber is $X_0$.
\end{theorem}

The above theorem is also considered in \cite{Tian2} from a more differential geometric point of view, and in \cite{Chan} the condition (\ref{eq:i-1}) is discussed in the obstructed case of the desingularization of K$\ddot{\text{a}}$hler Calabi-Yau 3-folds with conical singularities. \\[0.2cm]

The local geometry of the total space $\mathcal{X}$ near an ODP of $X_0$ is described in the following. For some $\epsilon>0$ and for
\begin{equation*}
       \tilde{U}=\{(z,t) \in \mathbb{C}^4 \times \Delta_\epsilon | \Arrowvert z \Arrowvert<2,\,\,z_1^2+z_2^2+z_3^2+z_4^2=t \}
\end{equation*}
there is a holomorphic map $\Xi:\tilde{U} \rightarrow \mathcal{X}$ respecting the projections to $\Delta$ and $\Delta_\epsilon$ so that $\tilde{U}$ is biholomorphic to its image. We will denote
\begin{equation*}
    Q_t:=\{z_1^2+z_2^2+z_3^2+z_4^2=t\} \subset \mathbb{C}^4.
\end{equation*}
From the above description, a neighborhood of 0 in $Q_0$ models a neighborhood of an ODP in $X_0$. For $t \neq 0$, $Q_t$ is called a deformed conifold. Throughout this paper we will denote by $\mathbf{r}_t$ the restriction of $\Arrowvert z \Arrowvert$ to $Q_t \subset \mathbb{C}^4$, and use the same notation for their pullbacks under $\Xi^{-1}$. 

For each ODP $p_i$ of $X_0$, we have the biholomorphism $\Xi_i:\tilde{U}_i \rightarrow \mathcal{X}$ as above. Without loss of generality we may assume that the images of the $\Xi_i$'s are disjoint. For a given $t \in \Delta$, define $V_{t,i}(c)$ to be the image under $\Xi_i$ of $\{(z,t) \in \mathbb{C}^4 \times \Delta_\epsilon | \mathbf{r}_t(z)< c,\,\,z_1^2+z_2^2+z_3^2+z_4^2=t \}$, and define $V_t(c)=\bigcup_i V_{t,i}(c)$. Define $V_{t,i}(R_1,R_2)=V_{t,i}(R_2)\backslash V_{t,i}(R_1)$ for any $0<R_1<R_2$ and $V_t(R_1,R_2)=\bigcup_i V_{t,i}(R_1,R_2)$. Define $U_i(c):=\pi^{-1}(V_{0,i}(c))\subset \hat{X}$ where $\pi$ is the small resolution $\pi:\hat{X} \rightarrow X_0$, and $U(c)=\bigcup_i U_{i}(c)$. Finally, define $X_t[c]=X_t \backslash V_t(c)$.

For each $t \neq0$, it can be easily checked that $\mathbf{r}_t\geq |t|^\frac{1}{2}$ on $Q_t$ and the subset $\{\mathbf{r}_t=|t|^\frac{1}{2}\}\subset Q_t$ is isomorphic to a copy of $S^3$, which is usually called the vanishing sphere. Each subset $V_t(c)$ is thus an open neighborhood of the vanishing spheres.\\[0.2cm]

\noindent \textbf{Remark} In the rest of the paper we will always regard $V_{t,i}(c)$ not only as a subset of $X_t$, but also as a subset of $Q_t$ via the map $\Xi_i$ and the projection map from the set $\{(z,t) \in \mathbb{C}^4 \times \Delta_\epsilon | \mathbf{r}_t(z)< c,\,\,z_1^2+z_2^2+z_3^2+z_4^2=t \}$ to $\mathbb{C}^4$. 

We also use the same notation $\mathbf{r}_t$ to denote a fixed smooth extension of $\mathbf{r}_t$ from $V_t(1)$ to $X_t$ so that $\mathbf{r}_t<3$.\\[0.2cm]

The following description of $Q_0$ and $Q_1$ will be useful in our discussion. Denote $\Sigma=SO(4)/SO(2)$. Then there are diffeomorphisms 
\begin{equation}
 \label{p1}
    \phi_0: \Sigma \times (0,\infty) \rightarrow Q_{0,sm}\,\,\,\,\text{such that}\,\,\,
    \phi_0(A\cdot SO(2),\mathbf{r}_0)=A
              \begin{pmatrix}                
           			\frac{1}{\sqrt{2}}\mathbf{r}_0 \\
  						\frac{i}{\sqrt{2}}\mathbf{r}_0 \\
          						0 \\
          						0
				\end{pmatrix},
\end{equation}
and
\begin{equation}
 \label{p2}
    \phi_1: \Sigma \times (1,\infty) \rightarrow Q_1\backslash \{\mathbf{r}_1=1\}\,\,\,\,\text{such that}\,\,\,
    \phi_1(A\cdot SO(2),\mathbf{r}_1)=A
              \begin{pmatrix}                
           			\cosh (\frac{1}{2}\cosh^{-1}(\mathbf{r}_1^2)) \\
  						i\sinh (\frac{1}{2}\cosh^{-1}(\mathbf{r}_1^2)) \\
          						0 \\
          						0
				\end{pmatrix}.
\end{equation}
Here $Q_{0,sm}$ is the smooth part of $Q_0$, and the variables $\mathbf{r}_0$ and $\mathbf{r}_1$ are indeed the distances of the image points to the origin. 

We can see in particular from (\ref{p1}) that $\phi_0$ describes $Q_0$ as a cone over $\Sigma$. It is not hard to see that $\Sigma \cong S^2 \times S^3$. However, the radial variable for the Ricci-flat K$\ddot{\text{a}}$hler cone metric $g_{co,0}$ on $Q_0$ is not $\mathbf{r}_0$, but $\rho_0=\mathbf{r}_0^\frac{2}{3}$. In fact, $g_{co,0}$ can be expressed as
\begin{equation}
 \label{cone}
    g_{co,0}=(d\mathbf{r}_0^\frac{2}{3})^2+\mathbf{r}_0^\frac{4}{3} g_\Sigma
\end{equation}
where $g_\Sigma$ is an $SO(4)$-invariant Sasaki-Einstein metric on $\Sigma$. The K$\ddot{\text{a}}$hler form of $g_{co,0}$ is given by $\omega_{co,0}=\sqrt{-1}\partial \bar{\partial} f_0(\mathbf{r}_0^2)$ where $f_0(s)=\frac{3}{2}s^\frac{2}{3}$. In this paper we will not use the variable $\rho_0$.

In this paper, given a Hermitian metric $g$, the notation $\nabla_g$ will always refer to the Chern connection of $g$.\\[0.2cm]

\subsection{The Candelas-de la Ossa metrics}

Candelas and de la Ossa \cite{CO} constructed a 1-parameter family of Ricci-flat K$\ddot{\text{a}}$hler metrics $\{ g_{co,a} | a>0\}$ on the small resolution $\hat{Q}$ of $Q_0$. The space $\hat{Q}$ is named the resolved conifold, and the parameter $a$ measures the size of the exceptional curve $C$ in $\hat{Q}$. Identifying $Q_{0,sm}$ with $\hat{Q}\backslash C$ biholomorphically via the resolution map, the family $\{ g_{co,a} | a>0\}$ converges smoothly, as $a$ goes to 0, to the cone metric $g_{co,0}$ on each compactly embedded open subset of $Q_{0,sm}$, i.e., each open subset of $Q_{0,sm}$ whose closure in $Q_0$ is contained in $Q_{0,sm}$. The K$\ddot{\text{a}}$hler forms of the metrics $g_{co,a}$ will be denoted by $\omega_{co,a}$.

They also construct a Ricci-flat K$\ddot{\text{a}}$hler metric $g_{co,t}$ on $Q_t$ for each $0\neq t \in \Delta$. Explicitly, the K$\ddot{\text{a}}$hler form of $g_{co,t}$ is given by $\omega_{co,t}=\sqrt{-1}\partial \bar{\partial} f_t(\mathbf{r}_t^2)$ where 
\begin{equation}
 \label{potentials}
    f_t(s)=2^{-\frac{1}{3}}|t|^\frac{2}{3} \int_0^{\cosh^{-1}(\frac{s}{|t|})} (\sinh(2\tau)-2\tau)^\frac{1}{3} d\tau,
\end{equation}
and it satisfies
\begin{equation}
 \label{n}
		\omega_{co,t}^3=\sqrt{-1}\frac{1}{2}\Omega_t\wedge \Omega_t
\end{equation}
where $\Omega_t$ is the holomorphic (3,0)-form on $Q_t$ such that, on $\{ z_1 \neq 0 \}$,
\begin{equation*}
     \Omega_t=\frac{1}{z_1}dz_2 \wedge dz_3 \wedge dz_4|_{Q_t}.
\end{equation*}

In this paper, the metrics $g_{co,a}$ with subscripts $a$ will always denote the Candelas-de la Ossa metrics on the resolved conifold $\hat{Q}$, and the metrics $g_{co,t}$ with subscripts $t$ will always denote the Candelas-de la Ossa metrics on the deformed conifolds $Q_t$.\\[0.2cm]

In the following we discuss the asymptotic behavior of the CO-metrics $g_{co,t}$. Consider the smooth map 
\begin{equation*}
     \Phi:\Sigma \times (1,\infty) \rightarrow \Sigma \times (0,\infty)
\end{equation*}
defined by
\begin{equation*}
     \Phi (A\cdot SO(2),\mathbf{r}_1)= ( A\cdot SO(2), \mathbf{r}_0(\mathbf{r}_1))	
\end{equation*}
where
\begin{equation*}
		\mathbf{r}_0(\mathbf{r}_1) =\left(\frac{1}{2}(\sinh(2\cosh^{-1}(\mathbf{r}_1^2))-2\cosh^{-1}(\mathbf{r}_1^2)) \right)^\frac{1}{4}.
\end{equation*}

Note that 
\begin{equation}
 \label{eq:6-2}
          \mathbf{r}_1=\left(\cosh \left ( f^{-1}\left( 2\mathbf{r}_0^4 \right) \right)\right)^\frac{1}{2}
\end{equation}
where $f(s)=\sinh(2s)-2s$. 

Define $x_1= \phi_0\circ\Phi\circ\phi_1^{-1}$, which is a diffeomorphism from $Q_1\backslash \{\mathbf{r}_1=1\}$ to $Q_{0,sm}$. Then $\mathbf{r}_0(x_1(x))=\mathbf{r}_0(\mathbf{r}_1(x))$ for $x \in Q_1\backslash \{\mathbf{r}_1=1\}$. Define $\Upsilon_1=x_1^{-1}$. It is shown in \cite{Chan} that the following hold for some constants $D_{1,k}$, $D_{2,k}$, and $D_{3,k}$ as $\mathbf{r}_0 \rightarrow \infty$:

\begin{equation}
 \label{eq:3-4}
           \Upsilon_1^*\omega_{co,1} = \omega_{co,0},
\end{equation}
\begin{equation}
 \label{eq:3-4-1}
        |\nabla_{g_{co,0}}^k(\Upsilon_1^*\Omega_1-\Omega_0)|_{g_{co,0}} \leq  D_{1,k}\mathbf{r}_0^{\frac{2}{3}(-3-k)},
\end{equation}
\begin{equation}
 \label{eq:6-1}
         |\nabla_{g_{co,0}}^k(\Upsilon_1^*g_{co,1}-g_{co,0})|_{g_{co,0}} \leq D_{2,k}\mathbf{r}_0^{\frac{2}{3}(-3-k)},
\end{equation}
and 
\begin{equation}
 \label{eq:3-5}
         |\nabla_{g_{co,0}}^k(\Upsilon_1^*J_1-J_0)|_{g_{co,0}} \leq D_{3,k}\mathbf{r}_0^{\frac{2}{3}(-3-k)}
\end{equation}
where $J_t$ is the complex structure on $Q_t$.

Let $\psi_t:Q_1 \rightarrow Q_t$ be a map such that $\psi_t^*(z_i)=t^{\frac{1}{2}}z_i$. Here $t^\frac{1}{2}$ can be either of the two square roots of $t$. We then have 
\begin{equation}
 \label{eq:s}
   \begin{split}
                               \psi_t^*\mathbf{r}_t=|t|^\frac{1}{2}\mathbf{r}_1, \, &\,\psi_t^*\mathbf{r}_0=|t|^\frac{1}{2}\mathbf{r}_0\\
								\psi_t^*\Omega_t=t\Omega_1, \,&\,\psi_t^*\Omega_0=t\Omega_0, \\
     \psi_t^*\omega_{co,t}=|t|^\frac{2}{3}\omega_{co,1}, \,&\,\psi_t^*\omega_{co,0}=|t|^\frac{2}{3}\omega_{co,0},\,\, \\
     \psi_t^* g_{co,t}=|t|^\frac{2}{3}g_{co,1}, \,\,& \text{and}\, \,\psi_t^* g_{co,0}=|t|^\frac{2}{3} g_{co,0}.\,\, \\
   \end{split}
\end{equation}
The equality $\psi_t^* \omega_{co,t}=|t|^\frac{2}{3}\omega_{co,1}$ follows from the explicit formulas of the K$\ddot{\text{a}}$hler potentials (\ref{potentials}) and the fact that the map $\psi_t$ is biholomorphic. With this understood, $\psi_t^* g_{co,t}=|t|^\frac{2}{3}g_{co,1}$ then follows easily. The rest are trivial. Note that $\nabla_{\psi_t^* g_{co,0}}=\nabla_{|t|^\frac{2}{3}g_{co,0}}=\nabla_{g_{co,0}}$ and $\nabla_{\psi_t^* g_{co,t}}=\nabla_{|t|^\frac{2}{3}g_{co,1}}=\nabla_{g_{co,1}}$ for $t\neq0$.

Let $x_t=\psi_t\circ x_1 \circ \psi_t^{-1}$, which is understood as a diffeomorphism from $Q_t\backslash \{\mathbf{r}_t=|t|^\frac{1}{2}\}$ to $Q_{0,sm}$. Note that $x_t$ is independent of the choice of $t^\frac{1}{2}$, and so $\{x_t\}_t$ form a smooth family. Define $\Upsilon_t=x_t^{-1}$.

\begin{lemma}
 \label{lm:4}
 		We have
		\begin{equation*}
				x_t^*\omega_{co,0}=\omega_{co,t},
		\end{equation*}
		and for the same constants $D_{1,k}$, $D_{2,k}$ and $D_{3,k}$ as in (\ref{eq:3-4-1}), (\ref{eq:6-1}) and (\ref{eq:3-5}), we have, as $\mathbf{r}_0 \rightarrow \infty$,
        \begin{equation*}
            \begin{split}
                 & |\nabla_{g_{co,0}}^k(\Upsilon_t^*\Omega_t-\Omega_0)|_{g_{co,0}} \leq D_{1,k} |t|\mathbf{r}_0^{\frac{2}{3}(-3-k)}, \\
                 & |\nabla_{g_{co,0}}^k(\Upsilon_t^*g_{co,t}-g_{co,0})|_{g_{co,0}} \leq D_{2,k}|t|\mathbf{r}_0^{\frac{2}{3}(-3-k)}, \,\,\text{and}\\
                 & |\nabla_{g_{co,0}}^k(\Upsilon_t^*J_t-J_0)|_{g_{co,0}} \leq D_{3,k}|t|\mathbf{r}_0^{\frac{2}{3}(-3-k)}.
            \end{split}  
        \end{equation*}
\end{lemma}
\begin{proof}
The first equation follows easily. From the rescaling properties (\ref{eq:s}) we have, for $w\in X_0$,
\begin{equation*}
    \begin{split}
        & |\nabla_{g_{co,0}}^k(\Upsilon_t^*\Omega_t-\Omega_0)|_{g_{co,0}}(w)= |\nabla_{g_{co,0}}^k((\psi_t^{-1})^*\Upsilon_1^*\psi_t^*\Omega_t-\Omega_0)|_{g_{co,0}}(w) \\
     = & |\nabla_{g_{co,0}}^k(t (\psi_t^{-1})^*\Upsilon_1^*\Omega_1-\Omega_0)|_{g_{co,0}}(w)
         =|\nabla_{\psi_t^*g_{co,0}}^k(t \Upsilon_1^*\Omega_1-\psi_t^*\Omega_0)|_{\psi_t^*g_{co,0}}(\psi_t^{-1}(w)) \\
     = & |\nabla_{g_{co,0}}^k( t \Upsilon_1^*\Omega_1- t \Omega_0)|_{|t|^\frac{2}{3}g_{co,0}}(\psi_t^{-1}(w))
         =|t| |\nabla_{g_{co,0}}^k( \Upsilon_1^*\Omega_1- \Omega_0 )|_{|t|^\frac{2}{3}g_{co,0}}(\psi_t^{-1}(w)) \\
     = & |t| |t|^{-\frac{1}{3}(3+k)} |\nabla_{g_{co,0}}^k( \Upsilon_1^*\Omega_1- \Omega_0)|_{g_{co,0}}(\psi_t^{-1}(w)) \\
  \leq & |t| |t|^{-\frac{1}{3}(3+k)} D_{1,k}\mathbf{r}_0(\psi_t^{-1}(w))^{\frac{2}{3}(-3-k)}
         =|t|^{-\frac{1}{3}k} D_{1,k}|t|^{\frac{1}{3}(3+k)}\mathbf{r}_0(w)^{\frac{2}{3}(-3-k)}\\
       =&D_{1,k}|t|\mathbf{r}_0(w)^{\frac{2}{3}(-3-k)}.
    \end{split}
\end{equation*}
The other two estimates can be carried out in a similar manner. \qed
\end{proof}

Using the explicit formula (\ref{eq:6-2}), the following lemma is elementary, and the proof is omitted: 
\begin{lemma}
 \label{lm:2}
As $x\in Q_1\backslash \{\mathbf{r}_1=1\}$ goes to infinity, $\mathbf{r}_1(x)\mathbf{r}_0(x_1(x))^{-1}$ goes to 1. In particular, there is a constant $A>0$ such that 
    \begin{equation*}
        \frac{1}{A}<\mathbf{r}_1(x)\mathbf{r}_0(x_1(x))^{-1}<A
    \end{equation*}
for any $x\in Q_1$ such that $1 \ll \mathbf{r}_1(x)$. As a result, by the rescaling relation (\ref{eq:s}), for the same constant $A$ we have
    \begin{equation*}
        \frac{1}{A}<\mathbf{r}_t(z)\mathbf{r}_0(x_t(z))^{-1}<A
    \end{equation*}
for any $z\in Q_t$ such that $|t|^\frac{1}{2}\ll \mathbf{r}_t(z)$.
\end{lemma}

Lemma \ref{lm:4} and Lemma \ref{lm:2} imply
\begin{corollary}
 \label{asym}
   There exists a constant $D_0>0$ such that for any $z\in Q_t$ with $|t|^\frac{1}{2}\ll \mathbf{r}_t(z)$,
   \begin{equation*}
			\begin{split}
                  |\nabla_{g_{co,0}}^k(\Upsilon_t^*J_t-J_0)|_{\Upsilon_t^*g_{co,t}}(x_t(z))  \leq D_0|t|\mathbf{r}_t(z)^{\frac{2}{3}(-3-k)}
            \end{split}
	\end{equation*}
		for $k=0,1$.
\end{corollary}

\subsection{The balanced metrics constructed by Fu-Li-Yau}

Using Mayer-Vietoris sequence, the change in the second Betti numbers before and after a conifold transition is given in the following proposition:
\begin{proposition} \cite{Reid}
Let $k$ be the maximal number of homologically independent exceptional rational curves in $\hat{X}$. Then the second Betti numbers of $\hat{X}$ and $X_t$ satisfy the equations
\begin{equation*}
   \begin{split}
		& b_2(X_t)=b_2(\hat{X})-k. \\
	\end{split}
\end{equation*}
\end{proposition}

From this proposition one sees that the second Betti number drops after each transition, and when it becomes 0, the resulting threefold is never K$\ddot{\text{a}}$hler. Because of this, when considering Reid's conjecture, a class of threefolds strictly containing the K$\ddot{\text{a}}$hler Calabi-Yau ones have to be taken into account. A particular question of interest would be finding out suitable geometric structures that are possessed by every member in this class of threefolds. One achievement in this direction is the work of \cite{FLY} in which the following theorem is proved:

\begin{theorem} 
Let $\hat{X}$ be a K$\ddot{\text{a}}$hler Calabi-Yau threefold. Then after a conifold transition, for sufficiently small $t$, $X_t$ admits a balanced metric. 
\end{theorem}
In the following we review the results in \cite{FLY} in more detail.

First, a balanced metric $\hat{\omega}_0$ on $X_{0,sm}$ is constructed by replacing the original metric $\omega$ near the ODPs with the CO-cone metric $\omega_{co,0}$. One of the main feature of this construction is that $\omega^2$ and $\hat{\omega}_0^2$ differ by a $\partial \bar{\partial}$-exact form. It is not hard to see that their construction can be used the construct a family of balanced metrics $\{ \hat{\omega}_{co,a} | a>0\}$ on $\hat{X}$ converging smoothly, as $a$ goes to 0, to the metric $\hat{\omega}_0$ on compactly embedded open subsets of $\hat{X}\backslash \bigcup C_i \cong X_{0,sm}$, such that $\omega^2$ and all $\hat{\omega}_{co,a}^2$ differ by $\partial \bar{\partial}$-exact forms.

The main achievement in \cite{FLY} is the construction of balanced metrics $\tilde{\omega}_t$ on $X_t$. Fix a smooth family of diffeomorphisms $x_t: X_t[\frac{1}{2}] \rightarrow X_0[\frac{1}{2}]$ that such $x_0=id$.  Let $\varrho(s)$ be a decreasing cut-off function such that $\varrho(s)=1$ when $s \leq \frac{5}{8}$ and $\varrho(s)=0$ when $s \geq \frac{7}{8}$. Define a cut off function $\varrho_0$ on $X_0$ such that $\varrho_0|_{X_0[1]}=0$, $\varrho_0|_{V_0(\frac{1}{2})}=1$, and $\varrho_0|_{V_0(\frac{1}{2},1)}=\varrho(\mathbf{r}_0)$. Also define a cut off function $\varrho_t$ on $X_t$ such that $\varrho_t|_{X_t[\frac{1}{2}]}=x_t^*\varrho_0$ and $\varrho_t|_{V_t(\frac{1}{2})}=1$. Denote $\hat{\Omega}_0=\hat{\omega}_0^2=i\partial \bar{\partial}(f_0 \partial \bar{\partial}f_0)$, and let 
\begin{equation*}
     \Phi_t=f_t^*(\hat{\Omega}_0-i\partial \bar{\partial}(\varrho_0\cdot f_0(\mathbf{r}_0^2) \partial \bar{\partial} f_0(\mathbf{r}_0^2)))+i\partial \bar{\partial}(\varrho_t\cdot f_t(\mathbf{r}_t^2) \partial \bar{\partial} f_t(\mathbf{r}_t^2)).
\end{equation*}
We can decompose the 4-form $\Phi_t=\Phi_t^{3,1}+\Phi_t^{2,2}+\Phi_t^{1,3}$. It is proved in \cite{FLY} that for $t\neq 0$ sufficiently small the (2,2) part $\Phi_t^{2,2}$ is positive and over $V_t(\frac{1}{2})$ it coincides with $\omega_{co,t}^2$. Let $\omega_t$ be the positive (1,1)-form on $X_t$ such that $\omega_t^2=\Phi_t^{2,2}$. Neither $\omega_t$ nor $\omega_t^2$ is closed in general. The balanced metric $\tilde{\omega}_t$ constructed in \cite{FLY} satisfies the condition $\tilde{\omega}_t^2=\Phi_t^{2,2}+\theta_t+\bar{\theta}_t$ where $\theta_t$ is a (2,2)-form satisfying the condition that, for any $\kappa>-\frac{4}{3}$,
\begin{equation}
 \label{eq:0-1}
          \lim_{t \rightarrow 0} (|t|^\kappa \sup_{X_t}|\theta_t|_{g_t}^2)=0
\end{equation}
where $g_t$ is the Hermitian metric associated to $\omega_t$. The proof of this limit makes use of the expression 
\begin{equation}
 \label{eq:0-2}
        \theta_t=\partial\bar{\partial}^*\partial^*\gamma_t
\end{equation}
for a unique (2,3)-form $\gamma_t$ satisfying the equation $E_t(\gamma_t)=-\partial \Phi^{1,3}_t$ and $\gamma_t\perp \ker{E_t}$ where
\begin{equation*}
    E_t=\partial\bar{\partial}\bar{\partial}^*\partial^*+\partial^*\bar{\partial}\bar{\partial}^*\partial+\partial^*\partial
\end{equation*}
and the $\ast$-operators are with respect to the metric $g_t$. It was proved in \cite{FLY} that $\partial\gamma_t=0$. Moreover, the $(2,3)$-form $\partial \Phi^{1,3}_t$ is supported on $X_t[1]$, so there is a constant $C>0$ independent of $t$ such that 
\begin{equation}
 \label{eq:1}
    |\partial \Phi^{1,3}_t|_{C^k}<C|t|.
\end{equation}

We will denote $|\cdot|_t$ the norm w.r.t. $\tilde{g}_t$, $|\cdot|_{co,t}$ the norm w.r.t. $g_{co,t}$, and $|\cdot|$ the norm w.r.t. $g_t$. We will denote $dV_t$ the volume w.r.t. $\tilde{g}_t$, $dV_{co,t}$ the volume w.r.t. $g_{co,t}$, and $dV$ the volume w.r.t. $g_t$.

Because of (\ref{eq:0-1}) we have the following lemma concerning a uniformity property between the metrics $g_t$ and $g_{co,t}$. 
\begin{lemma}
 \label{lm:0-1}
    There exists a constant $\tilde{C}>1$ such that for any small $t\neq 0$, over the region $V_t(1)$ we have
    \begin{equation*}
          \tilde{C}^{-1}\tilde{g}_t \leq g_{co,t} \leq \tilde{C}\tilde{g}_t.
    \end{equation*} 
    Consequently, we have constants $\tilde{C}_1>1$ and $\tilde{C}_2>1$ such that for any $t\neq 0$ small enough,
    \begin{equation*}
          \tilde{C}_1^{-1}dV_t \leq dV_{co,t} \leq \tilde{C}_1 dV_t
    \end{equation*} 
    and
    \begin{equation*}
          \tilde{C}_2^{-1}|\cdot|_t \leq |\cdot|_{co,t} \leq \tilde{C}_2 |\cdot|_t.
    \end{equation*}     
\end{lemma}

Now we introduce our conventions on (negative) Laplacians. Let $\omega$ be a Hermitian metric on $X$ and $n=\dim_{\mathbb{C}}X$. For any (1,1)-form $\varphi$ on $X$, define $\Lambda_{\omega}\varphi:=\frac{\varphi\wedge\omega^{n-1}}{\omega^n}$. For a smooth function $f$ on $X$, define $\Delta_{\omega}f=\sqrt{-1}\Lambda_{\omega} \partial\bar{\partial}f$. In local coordinates, if $\omega=\frac{\sqrt{-1}}{2}g_{i\bar{j}}dz_i\wedge d\bar{z}_j$ and $\varphi=\varphi_{i\bar{j}}dz_i\wedge d\bar{z}_j$, then $\sqrt{-1}\Lambda_{\omega}\varphi=2g^{i\bar{j}}\varphi_{i\bar{j}}$. We denote $\tilde{\Delta}_t:=\Delta_{\tilde{\omega}_t}$ and $\hat{\Delta}_a:=\Delta_{\hat{\omega}_a}$. \\[0.2cm]

\subsection{Hermitian-Yang-Mills equation}

Let $H$ be a Hermitian metric on a holomorphic vector bundle $\mathcal{E}$ over a complex manifold $X$ endowed with a balanced metric $g$. Let $\nabla_A=\partial_A+\bar{\partial}_A$ be an $H$-unitary connection on $\mathcal{E}$. We denote by $\langle \cdot,\cdot \rangle_{H,g}$ the pointwise pairing induced by $H$ and $g$ between the $\mathcal{E}$-valued forms or the $\text{End}(\mathcal{E})$-valued forms. The following proposition is will be used in later calculations.
\begin{proposition} \cite{LT}
		For $h_1,h_2 \in \Gamma(\text{End}(\mathcal{E}))$, we have
		\begin{equation*}
				\int_X \langle \partial_Ah_1,\partial_Ah_2 \rangle_{H,g}\,dV_g=\sqrt{-1}\int_X \langle \Lambda_g\bar{\partial}_A\partial_Ah_1,h_2 \rangle_{H,g}\,dV_g
		\end{equation*}
				and
		\begin{equation*}
				\int_X \langle \bar{\partial}_Ah_1,\bar{\partial}_Ah_2 \rangle_{H,g}\,dV_g=-\sqrt{-1}\int_X \langle \Lambda_g\partial_A\bar{\partial}_Ah_1,h_2 \rangle_{H,g}\,dV_g.\\[0.3cm]
		\end{equation*}
\end{proposition}

In a local holomorphic frame of $\mathcal{E}$, the curvature of a connection $\nabla_A$ is given by
\begin{equation*}
		F_A:=dA-A \wedge A, 
\end{equation*}
which is an End$(\mathcal{E})$-valued 2-form. Given a Hermitian metric $H$ over a bundle $\mathcal{E}$, the curvature for the Chern connection can then be locally computed to be 
\begin{equation*}
		F_H=\bar{\partial}(\partial HH^{-1}).
\end{equation*}
Taking the trace of the curvature 2-form with respect to a Hermitian metric $\omega$, we obtain the mean curvature $\sqrt{-1}\Lambda_{\omega}F_H$ of $H$. It is not hard to see that $\sqrt{-1}\Lambda_{\omega}F_H$ is $H$-symmetric.

\begin{definition}
		A Hermitian metric $H$ on $\mathcal{E}$ satisfies the Hermitian-Yang-Mills equation with respect to $\omega$ if
		\begin{equation*}
					\sqrt{-1}\Lambda_{\omega}F_H=\lambda I
		\end{equation*}
		for some constant $\lambda$. Here $I$ denotes the identity endomorphism of $\mathcal{E}$.
\end{definition}
Next we introduce slope stability. For a given Hermitian metric $H$ on $\mathcal{E}$, the first Chern form of $\mathcal{E}$ with respect to $H$ is defined to be 
\begin{equation*}
		c_1(\mathcal{E},H)=\frac{\sqrt{-1}}{2\pi}\text{tr}F_H.
\end{equation*}
It is independent of $H$ up to a $\partial\bar{\partial}$-exact form, and is a representative of the topological first Chern class $c_1(E)\in H^2(X,\mathbb{C})$. 

The $\omega$-degree of $\mathcal{E}$ with respect to a Hermitian metric $\omega$ is defined to be
\begin{equation*}
		\text{deg}_{\omega}(\mathcal{E}):=\int_{X} c_1(\mathcal{E},H)\wedge \omega^{n-1}
\end{equation*}
where $n=\dim_{\mathbb{C}}X$. This is not well-defined for a general $\omega$. It is, however, well-defined for a Gauduchon metric $\omega$ since $\partial\bar{\partial}(\omega^{n-1})=0$ and $c_1(\mathcal{E},H)$ is independent of $H$ up to $\partial\bar{\partial}$-exact forms. In particular, the degree with respect to a balanced metric is well-defined. Note that the $\omega$-degree is a topological invariant, i.e., depends only on $c_1(\mathcal{E})$, if $\omega$ is balanced. We restrict ourselves from now on to the case when $\omega$ is Gauduchon.

For an arbitrary coherent sheaves $\mathcal{F}$ of $\mathcal{O}_X$-modules of rank $s>0$, we define deg$_{\omega}(\mathcal{F}):=\text{deg}_{\omega}(\det\mathcal{F})$ where $\det \mathcal{F}:=(\Lambda^s \mathcal{F})^{**}$ is the determinant line bundle of $\mathcal{F}$. We define the $\omega$-slope of $\mathcal{F}$ to be $\mu_{\omega}(\mathcal{F}):=\frac{ \text{deg}_{\omega}(\mathcal{F})}{s}$. 

\begin{definition}
   A holomorphic vector bundle $\mathcal{E}$ is said to be $\omega$-(semi)stable if $\mu_{\omega}(\mathcal{F})<(\leq)\mu_{\omega}(\mathcal{E})$ for every coherent subsheaf $\mathcal{F}\hookrightarrow\mathcal{E}$ with $0<\text{rank}\,\mathcal{F}<\text{rank}\,\mathcal{E}$. 
   
   A holomorphic vector bundle $\mathcal{E}$ is said to be $\omega$-polystable if $\mathcal{E}$ is a direct sum of $\omega$-stable bundles all of which have the same $\omega$-slope.

\end{definition}

The following theorem generalizing \cite{UY} was proved by Li and Yau \cite{LY1}:
\begin{theorem}
 \label{LY}
		On a complex manifold $X$ endowed with a Gauduchon metric $\omega$, a holomorphic vector bundle $\mathcal{E}$ is $\omega$-polystable if and only if it admits a Hermitian-Yang-Mills metric with respect to $\omega$.\\[0.2cm]
\end{theorem}

\subsection{Controls of constants}

Let $\mathcal{E}$ be a holomorphic vector bundle over a compact Hermitian manifold $(X,g)$, $H$ a Hermitian metric on $\mathcal{E}$, and $\nabla_{H,g}$ the connection on $\mathcal{E}\otimes (\Omega^1)^{\otimes k}$ induced from the Chern connections of $H$ and $g$. Let $\mathbf{r}$ be a smooth positive function on $X$.

We can define the following weighted norms on the usual Sobolev spaces $L^p_k(\mathcal{E})$ over $X$: for each $\sigma \in L^p_k(\mathcal{E})$,
    \begin{equation*}
         \Arrowvert \sigma \Arrowvert_{L^p_{k,\beta}} :=\left( \sum_{j=0}^{k} \int_{X} |\mathbf{r}^{-\frac{2}{3}\beta+\frac{2}{3}j}\nabla_{H,g}^j\sigma|_{H,g}^p\mathbf{r}^{-4}\, dV_g \right)^\frac{1}{p}.
    \end{equation*}
We denote by $L^p_{k,\beta}(\mathcal{E})$ the same space as $L^p_k(\mathcal{E})$ but endowed with the above norm. Here $dV_g$ is the volume form of $g$.
   
There are also the weighted $C^k$-norms:
    \begin{equation*}
         \Arrowvert \sigma \Arrowvert_{C^k_{\beta}} :=\sum_{j=0}^{k} \sup_{X} |\mathbf{r}^{-\frac{2}{3}\beta+\frac{2}{3}j}\nabla_{H,g}^j \sigma |_{H,g}.
    \end{equation*}
We denote by $C^k_{\beta}(\mathcal{E})$ the same space as $C^k(\mathcal{E})$ but endowed with the above norm.\\[0.2cm]

Now let $\{\phi_ {z}:B_z\rightarrow U_z\subset X\}_{z\in X}$ be a system of complex coordinate charts where each $\phi_ {z}$ maps the Euclidean ball of radius $\rho$ in $\mathbb{C}^3$ centered at 0 homeomorphically to $U_z$, an open neighborhood of $z$, such that $\phi_ {z}(0)=z$. Over each $U_z$ define $\bar{g}$ to be $\mathbf{r}(z)^{-\frac{4}{3}}g$. Let $g_e$ denote the standard Euclidean metric on $B_z\subset \mathbb{C}^3$, and $\nabla_e$ the Euclidean derivatives.

	For $m\geq 0$, let $R_m>0$ be constants such that for any $z \in X$ and $y\in U_z$,
          \begin{equation}
           \label{R-1}
              \frac{1}{R_0}\mathbf{r}(z) \leq \mathbf{r}(y) \leq R_0\mathbf{r}(z)
          \end{equation}
          and
          \begin{equation}
           \label{R-2}
					|\nabla_e^m \mathbf{r}|_{g_e}(y) \leq R_m\mathbf{r}(y).
			\end{equation}  
	For $k\geq 0$, let $C_k>0$ be constants such that for any $z\in X$,
 		\begin{equation}
		 \label{R-3}
                 \frac{1}{C_0}g_e \leq \phi_{z}^*\bar{g} \leq C_0g_e
           \end{equation}
        over $B_z$ where $g_e$ is the Euclidean metric, and
           \begin{equation}
	     	 \label{R-4}
                 \Arrowvert \phi_{z}^*\bar{g} \Arrowvert_{C^k(B_z,g_e)} \leq C_k.
           \end{equation}

We may deduce the following version of Sobolev Embedding Theorem. 
\begin{theorem}
  \label{th:so}
    For each $l,p,q,r$ there exists a constant $C>0$ depending only on the constants $R_m$ and $C_k$ above such that
    \begin{equation*}
         \Arrowvert \sigma \Arrowvert_{L^r_{l,\beta}}\leq C \Arrowvert \sigma \Arrowvert_{L^p_{q,\beta}}
    \end{equation*}
    whenever $\frac{1}{r}\leq \frac{1}{p} \leq \frac{1}{r}+\frac{q-l}{6}$ and
    \begin{equation*}
         \Arrowvert \sigma \Arrowvert_{C^l_{\beta}}\leq C \Arrowvert \sigma \Arrowvert_{L^p_{q,\beta}}
    \end{equation*}
    whenever $\frac{1}{p} < \frac{q-l}{6}$.
\end{theorem}

The proof of the above result is standard. Simply put, we integrate over $z \in X$ the Sobolev inequalities on each chart $U_z$, and use the bounds (\ref{R-1})-(\ref{R-4}) to help control the constants of the global inequalities. 

In fact, the method of this proof is useful in controlling not only the Sobolev constants, but the constants in elliptic estimates as well. Consider a linear differential operator $P:C^\infty(\mathcal{E}) \rightarrow C^\infty(\mathcal{E})$ of order $m$ on the space of smooth sections of $\mathcal{E}$. Assume also that $P$ is strongly elliptic, i.e., its principal symbol $\sigma(P)$ satisfies the condition that there is a constant $\lambda>0$ such that $\langle \sigma_{\xi}(P)(v),v \rangle \geq \lambda ||v||^2 $ for any $v \in \mathbb{R}^r$ ($r=\text{rank}\mathcal{E}$) and $\xi \in \mathbb{R}^6$ with norm $||\xi||=1$.

\begin{proposition}
 \label{pr:uniform}
Assume there are constants $\Lambda_k>0$, $k\geq 0$, such that for any $z \in X$ there is a trivialization of $\mathcal{E}|_{U_z}$ under which the operator $P$ above takes the form
\begin{equation*}
		P=\sum_{|\alpha|\leq m} A_{\alpha}\frac{\partial^{|\alpha|}}{\partial w_1^{\alpha_1}...\partial \bar{w}_3^{\alpha_6}}
\end{equation*}
in the coordinates $(w_1,w_2,w_3)\in B_z \subset \mathbb{C}^3$, and the matrix-valued coefficient functions $A_\alpha$ satisfy
\begin{equation*}
		|\nabla_{e}^k A_\alpha|_{g_e} \leq \Lambda_k 
\end{equation*}
for all $\alpha$ and $k$. Here $\alpha=(\alpha_1,...,\alpha_6)$, $\alpha_i \geq 0$, are the multi-indices and $|\alpha|=\alpha_1+...+\alpha_6$.

Assume also that there is a Hermitian metric $H$ on $\mathcal{E}$ and constants $C'_k>0$ for $k\geq 0$, such that when $H$ is viewed as a matrix-valued function on $U_z$ under the above frames, we have ${C'_0}^{-1}I \leq H \leq C'_0I$ and $|\nabla_{e}^k H|_{g_e} \leq C'_k$ on $U_z$ for any $k$ and $z \in X$. 

Then there exists a constant $C>0$ depending only on $p$, $l$, $m$, $\beta$, $\lambda$, $\Lambda_k$, $R_m$, $C_k$, and $C'_k$ such that for any $\sigma \in C^\infty(\mathcal{E})$, we have
    \begin{equation*}
         \Arrowvert \sigma \Arrowvert_{L^p_{l+m,\beta}}\leq C \left( \Arrowvert P(\sigma) \Arrowvert_{L^p_{l,\beta}}+\Arrowvert \sigma \Arrowvert_{L^2_{0,\beta}} \right).\\[0.2cm]
    \end{equation*}
\end{proposition}

\section{Uniform coordinate systems} 

In this section we will construct coordinate systems with special properties over $X_{0,sm}$ and over each $X_t$ for small $t\neq 0$. Later we will mainly be using the wieghted Sobolev spaces and the discussions in Section 2 show that these coordinate systems help providing uniform controls of constants appearing in the weighted versions of Sobolev inequalities and elliptic estimates. The use of weighted Sobolev spaces is now standard in the gluing constructions or desingularization of spaces with conical singularities. See \cite{LM} and \cite{Pa} for more details.

The main goal of this section is to prove the following theorem.

\begin{theorem}
 \label{th:0}   
There is a constant $\rho>0$ such that, for any $t$ ($t$ can be zero), at each point $z\in X_t$ (or $z \in X_{0,sm}$ when $t=0$), there is an open neighborhood $U_z\subset X_t$ (or $U_z\subset X_{0,sm}$ when $t=0$) of $z$ and a diffeomorphic map $\phi_ {t,z}:B_z\rightarrow U_z$ from the Euclidean ball of radius $\rho$ in $\mathbb{C}^3$ centered at 0 to $U_z$ mapping 0 to $z$ so that one has the following properties:
 \begin{enumerate}
   \item[(i)] There are constants $R_m>0$, $m \geq 0$, such that for any $t$, $z \in X_t$ (or $z \in X_{0,sm}$ when $t=0$) and $y\in U_z$,
          \begin{equation}
           \label{eq:18}
              \frac{1}{R_0}\mathbf{r}_t(z) \leq \mathbf{r}_t(y) \leq R_0\mathbf{r}_t(z)
          \end{equation}
          and
          \begin{equation}
           \label{eq:ra-1}
					|\nabla_e^m \mathbf{r}_t|_{g_e}(y) \leq R_m\mathbf{r}_t(y).
			\end{equation}  
    \item[(ii)]  Over each $U_z$ define $\bar{\tilde{g}}_t$ to be $\mathbf{r}_t(z)^{-\frac{4}{3}}\tilde{g}_t$. Then for each $k\geq 0$, there is a constant $C_k$ independent of $t$ and $z\in X_t$ (or $z \in X_{0,sm}$ when $t=0$) such that 
           \begin{equation}
             \label{eq:19}
                 \frac{1}{C_0}g_e \leq \phi_{t,z}^*\bar{\tilde{g}}_t \leq C_0g_e 
           \end{equation}
            over $B_z$, and 
           \begin{equation}
             \label{eq:20}
                 \Arrowvert \phi_{t,z}^*\bar{\tilde{g}}_t \Arrowvert_{C^k(B_z,g_e)} \leq C_k.
           \end{equation}
  \end{enumerate}  
\end{theorem} 

We first consider the following version of this theorem:
\begin{theorem}
 \label{th:0-1}   
Theorem \ref{th:0} holds with $B_z$ understood as a Euclidean ball of radius $\rho$ in $\mathbb{R}^6$ centered at 0 and $g_e$ as the standard Euclidean metric on $B_z\subset \mathbb{R}^6$.
\end{theorem} 
 
The proof of Theorem \ref{th:0-1} begins with a version where $X_t$ are replaced by $Q_t$ and $\tilde{g}_t$ by $g_{co,t}$.
\begin{proposition}
 \label{pr:n-1}
There is a constant $\rho>0$ such that, for any $t$ ($t$ can be zero), at each point $z\in Q_t$ ($z \in Q_{0,sm}$ when $t=0$), there is an open neighborhood $U_z\subset Q_t$ (or $U_z\subset Q_{0,sm}$ when $t=0$) of $z$ and a diffeomorphic map $\phi_ {t,z}:B_z\rightarrow U_z$ from the Euclidean ball of radius $\rho$ in $\mathbb{R}^6$ centered at 0 to $U_z$ mapping 0 to $z$ so that one has the following properties:
 \begin{enumerate}
   \item[(i)] There are constants $R_m>0$, $m \geq 0$, such that for any $t$, $z \in Q_t$ (or $z \in Q_{0,sm}$ when $t=0$) and $y\in U_z$,
          \begin{equation}
           \label{eq:Q-1}
              \frac{1}{R_0}\mathbf{r}_t(z) \leq \mathbf{r}_t(y) \leq R_0\mathbf{r}_t(z)
          \end{equation}
          and
          \begin{equation}
           \label{eq:Q-2}
					|\nabla_e^m \mathbf{r}_t|_{g_e}(y) \leq R_m\mathbf{r}_t(y).
			\end{equation}  
    \item[(ii)]  Over each $U_z$ define $\bar{g}_{co,t}$ to be $\mathbf{r}_t(z)^{-\frac{4}{3}}g_{co,t}$. Then for each $k\geq 1$, there is a constant $C_k$ independent of $t$ and $z\in Q_t$ (or $z \in Q_{0,sm}$ when $t=0$) such that 
           \begin{equation}
            \label{eq:Q-3}
                 \frac{1}{C_0}g_e \leq \phi_{t,z}^*\bar{g}_{co,t} \leq C_0g_e 
           \end{equation}
            over $B_z$, and 
           \begin{equation}
            \label{eq:Q-4}
                 \Arrowvert \phi_{t,z}^*\bar{g}_{co,t} \Arrowvert_{C^k(B_z,g_e)} \leq C_k.
           \end{equation}
  \end{enumerate}  
\end{proposition}

\begin{proof}
While constructing the coordinate charts, we prove (\ref{eq:Q-1}), (\ref{eq:Q-3}), and (\ref{eq:Q-4}) first, leaving (\ref{eq:Q-2}) to be discussed at the end. 

We begin with the $t=0$ case. Choose $\rho<1$ to be significantly smaller then the injectivity radius of the metric $g_\Sigma$ from (\ref{cone}). Then at each point $p\in \Sigma$ one has the coordinates $\Phi_p:\tilde{B}_p\rightarrow \Sigma$ from the Euclidean ball of radius $\rho$ in $\mathbb{R}^5$ centered at 0 to $\Sigma$ mapping 0 to $p$ and satisfying the properties that there are constants $\tilde{C}_k>0$, $k\geq 0$, independent of $p$ such that
\begin{equation}
 \label{eq:c0-1}
     \frac{1}{\tilde{C}_0}\tilde{g}_e \leq  \Phi_p^*g_\Sigma \leq \tilde{C}_0\tilde{g}_e
\end{equation}
over $\tilde{B}_p$, and 
\begin{equation}
 \label{eq:c0-2}
     \Arrowvert \Phi_p^*g_\Sigma \Arrowvert_{C^k(\tilde{B}_p,\tilde{g}_e)} \leq \tilde{C}_k.
\end{equation}
Here $\tilde{g}_e$ is the standard Euclidean metric on $\tilde{B}_p$. More explicitly, we can simply choose a coordinate chart around one point in $\Sigma$ and then define the coordinates around the other points of $\Sigma$ by using the transitive action of $SO(4)$ on $\Sigma$. Since the metric $g_\Sigma$ is $SO(4)$-invariant, the above constants are easily seen to exist.

For $x \in Q_{0,sm}$ with $\phi_0^{-1}(x)=(p,\mathbf{r}_0(x)) \in \Sigma \times (0,\infty)$, define 
\begin{equation*}
    j_x: \tilde{B}_p\times (-\rho, \rho) \hookrightarrow \Sigma \times (0,\infty)
\end{equation*}
which maps $(y,s) \in \tilde{B}_p\times (-\rho, \rho)$ to $j_x(y,s)=(\Phi_p(y), \mathbf{r}_0(x)e^{\frac{3}{2}s})$. Denote the restriction of $j_x$ to $B_x \subset \tilde{B}_p\times (-\rho, \rho)$ by the same notation. Then define $\phi_{0,x}:\phi_0\circ j_x:B_x \hookrightarrow Q_{0,sm}$. Condition (\ref{eq:Q-1}) is manifest.

We have
\begin{equation}
   \label{eq:3}
     \phi_{0,x}^*g_{co,0} =(d(\mathbf{r}_0(x)^\frac{2}{3}e^s))^2+\mathbf{r}_0(x)^\frac{4}{3}e^{2s}\Phi_p^*g_\Sigma=\mathbf{r}_0(x)^\frac{4}{3}e^{2s}((ds)^2+\Phi_p^*g_\Sigma). 
\end{equation}
By choosing $\rho$ small so that $\frac{1}{2}<e^{2s}<2$ for $s \in (-\rho,\rho)$. Using the identity $g_e=(ds)^2+\tilde{g}_e$, one sees that the bound (\ref{eq:Q-3}) for the $t=0$ case follows from (\ref{eq:c0-1}). Moreover, using the fact that the derivatives of $e^{2s}$ and $(ds)^2+\Phi_p^*g_\Sigma$ are bounded in the Euclidean norm on $B_x$, the bound (\ref{eq:Q-4}) for this case follows.\\[0.2cm]

Next we deal with the $t=1$ case. We will use the asymptotically conical behavior of the deformed conifold metrics discussed in Section 2. Recall the explicit diffeomorphism $x_1:Q_1\backslash \{\mathbf{r}_1=1\} \rightarrow Q_{0,sm}$ with inverse $\Upsilon_1$, and also the estimate
\begin{equation}
 \label{eq:2}
     | \nabla_{g_{co,0}}^k(\Upsilon_1^*g_{co,1}-g_{co,0}) |_{g_{co,0}} \leq D_{2,k}\mathbf{r}_0^{-\frac{2}{3}(3+k)}.
\end{equation}
for $\mathbf{r}_0 \in (R,\infty)$ where $R>0$ is a large number. Let $\overline{V_1(R)}$ be the compact subset of $Q_1$ where $\mathbf{r}_1\leq R$. We will specify the choice of $R$ later. It is easy to see that the desired neighborhood $U_w$ exists for $w$ inside $\overline{V_1(R)}$. In fact, for $w\in \overline{V_1(R)}$ we can even choose $B_w$ to be a Euclidean ball of fixed small radius in $\mathbb{C}^3$ with the real coordinates taken from the real and imaginary parts of the complex coordinates. Therefore we focus on $Q_1\backslash \overline{V_1(R)}$. 

For $w \in Q_1\backslash \overline{V_1(R)}$, define 
\begin{equation*}
     \phi_{1,w}:=\Upsilon_1\circ \phi_{0,x_1(w)}: B_w \rightarrow Q_1\backslash \overline{V_1(R)}
\end{equation*}
for each $w \in Q_1\backslash \overline{V_1(R)}$. Here we identify $B_w$ with $B_{x_1(w)}$. What we do is defining the chart around $w \in Q_1\backslash \overline{V_1(R)}$ by pushing forward the chart around $x_1(w)$ via $\Upsilon_1$. Property (\ref{eq:Q-1}) is clear in view of Lemma \ref{lm:2}. \\[0.2cm]

From the $k=0$ case of (\ref{eq:2}) and (\ref{eq:Q-3}) for the $t=0$ case, (\ref{eq:Q-3}) holds for $t=1$ for a constant independent of $w\in Q_1\backslash \overline{V_1(R)}$ when $R$ is large enough. 

We have
\begin{equation}
 \label{eq:15}
   \begin{split}
         \phi_{1,w}^*\left(\mathbf{r}_1(w)^{-\frac{4}{3}}g_{co,1}\right) 
         =\mathbf{r}_1(w)^{-\frac{4}{3}} \phi_{0,x_1(w)}^*\left(  \Upsilon_1^*g_{co,1}-g_{co,0}\right)+\left( \mathbf{r}_1(w)^{-\frac{4}{3}} \phi_{0,x_1(w)}^*g_{co,0}  \right) \\
   \end{split}
\end{equation}

The second term in the RHS of (\ref{eq:15}) is dealt with in a way similar to the $t=0$ case as follows. By (\ref{eq:3}) we can write
\begin{equation*}
   \begin{split}
         \mathbf{r}_1(w)^{-\frac{4}{3}} \phi_{0,x_1(w)}^*g_{co,0}  
         =&\mathbf{r}_1(w)^{-\frac{4}{3}}\mathbf{r}_0(x_1(w))^\frac{4}{3}e^{2s}\left((ds)^2+\Phi_p^*g_\Sigma\right)
   \end{split}
\end{equation*}
Lemma \ref{lm:2} implies that for $R$ large enough we have
\begin{equation*}
			|\mathbf{r}_1(w)^{-\frac{4}{3}}\mathbf{r}_0(x_1(w))^\frac{4}{3}| <A,
\end{equation*}
where $A$ is independent of $w\in Q_1\backslash \overline{V_1(R)}$, and from this we obtain, as in the $t=0$ case,

\begin{equation}
 \label{eq:9}
    \Arrowvert  \mathbf{r}_1(w)^{-\frac{4}{3}} \phi_{0,x_1(w)}^*g_{co,0}  \Arrowvert_{C^k(B_w,g_e)}<C_{0,k}.
\end{equation}

Next we deal with the first term in the RHS of (\ref{eq:15}). Note that by (\ref{eq:2}) and the bound (\ref{eq:Q-3}) for $t=0$, we have, for any $w \in Q_1\backslash \overline{V_1(R)}$ when $R$ is large enough,
\begin{equation}
 \label{eq:7}
     \left\Arrowvert \mathbf{r}_1(w)^{-\frac{4}{3}}\phi_{0,x_1(w)}^*\left( \nabla_{g_{co,0}}^k (\Upsilon_1^*g_{co,1}-g_{co,0}) \right)\right\Arrowvert_{C^0(B_w,g_e)} \leq D'_{2,k}\sup_{y\in x_1(U_w)}( \mathbf{r}_1(w)^{-\frac{4}{3}}\mathbf{r}_0(y)^{-\frac{2}{3}}).
\end{equation}
Here $U_w$ is the image of $B_w$ in $Q_1$.
Note that from (\ref{eq:Q-1}) (for the $t=1$ case) and Lemma \ref{lm:2} one can deduce that 
\begin{equation*}
     \mathbf{r}_1(w)^{-\frac{4}{3}}\mathbf{r}_0(y)^{-\frac{2}{3}}<1
\end{equation*}
for $w\in  Q_1\backslash \overline{V_1(R)}$ and for any $y \in x_1(U_w)$ if $R$ is large enough. 

\begin{lemma}
     For each $k \geq 0$ there is a constant $C_{1,k}>0$ independent of $w \in Q_1 \backslash \overline{V_1(R)}$ such that 
     \begin{equation*}
          \left\Arrowvert  \phi_{0,x_1(w)}^* (\Upsilon_1^*g_{co,1}-g_{co,0}) \right\Arrowvert_{C^k(B_w,g_e)} 
          \leq C_{1,k} \sum_{j=0}^k\Arrowvert \phi_{0,x_1(w)}^*\left( \nabla_{g_{co,0}}^j (\Upsilon_1^*g_{co,1}-g_{co,0}) \right)\Arrowvert_{C^0(B_w,g_e)}. 
     \end{equation*}
\end{lemma}
\begin{proof}
    Recall the expression (\ref{eq:3}) for the pullback of $g_{co,0}$ to $B_w$. Using (\ref{eq:c0-1}) and (\ref{eq:c0-2}), an explicit calculation shows that the Christoffel symbols of the cone metric $g_{co,0}$ and their derivatives are bounded in $B_w$ w.r.t. the Euclidean norm by constants independent of $w \in Q_1 \backslash \overline{V_1(R)}$. The lemma now follows easily. \qed  
\end{proof}

From this lemma we have for $k \geq 1$
\begin{equation}
 \label{eq:14}
    \Arrowvert \mathbf{r}_1(w)^{-\frac{4}{3}}   \phi_{0,x_1(w)}^* (\Upsilon_1^*g_{co,1}-g_{co,0}) \Arrowvert_{C^k(B_w,g_e)} <C_{2,k}
\end{equation}

The required bound (\ref{eq:Q-4}) for the $t=1$ case then follow from (\ref{eq:15}), (\ref{eq:9}) and (\ref{eq:14}). \\[0.3cm]

We proceed to consider the case for general $t \neq 0$. For each point $z=\psi_t(w)$ in $Q_t$, denote $U_z=\psi_t(U_w)$, $B_z=B_w$ and define $\phi_{t,z}=\psi_t\circ \phi_{1,w}$. Then $\{(U_z,\phi_{t,z})|z\in Q_t\}$ is a coordinate system on $Q_t$ and one can check that
    \begin{equation}
      \label{eq:16}
          \frac{1}{C_0}g_e \leq \phi_{t,z}^*\bar{g}_{co,t} \leq C_0g_e
    \end{equation}
    over $B_z$, and
    \begin{equation}
      \label{eq:17}
          \Arrowvert  \phi_ {t,z} ^*\bar{g}_{co,t} \Arrowvert_{C^k(B_z,g_e)} \leq C_k.
    \end{equation}
for the same constants $C_k$ appearing in the $t=1$ case.\\[0.2cm]

Finally, we prove (\ref{eq:Q-2}). In the $t=0$ case, for $y\in U_x$ we have $\mathbf{r}_0=\mathbf{r}_0(x)e^{\frac{2}{3}s}$, $s \in (-\rho,\rho)$, for values of $\mathbf{r}_0$, and (\ref{eq:Q-2}) follows immediately. 

For the $t=1$ case, recall the expression (\ref{eq:6-2}) of $\mathbf{r}_1$ as a function of $\mathbf{r}_0$. If a point $y\in U_w\subset Q_1$ has coordinates $(p,s)\in B_w$, then $\mathbf{r}_1(y)=\mathbf{r}_1(s)=\mathbf{r}_1\left(\mathbf{r}_0(x_1(w))e^{\frac{3}{2}s}\right)$.  From straight forward computation we can see that there exist constants $R'_m$, $m\geq 1$, independent of $w \in Q_1$ such that $\left| \frac{\partial^m}{\partial s^m}\mathbf{r}_1(s) \right| \leq R'_m\mathbf{r}_1(s)$. This implies (\ref{eq:Q-2}) for the $t=1$ case. The general case follows easily from a rescaling argument.

The proof of Proposition \ref{pr:n-1} is now complete.\qed
\end{proof}

It's not hard to deduce the following:
\begin{corollary}
 \label{co:3}
      For any fixed $\beta \in\mathbb{R}\backslash \{0\}$, there are constants $R''_m>0$, $m\geq 1$, such that $|\nabla_{g_{co,t}}^m \mathbf{r}_t^\beta|_{g_{co,t}} \leq R''_m\mathbf{r}_t^{\beta-\frac{2}{3}m}$ on $Q_t$ for any $t$.
\end{corollary}

The above proposition and the uniform geometry of $\bigcup_t X_t[1]$ together imply
\begin{proposition}
 \label{pr:2}   
 Theorem \ref{th:0-1} is true if $\tilde{g}_t$ is replaced by $g_t$.\\[0.2cm]
\end{proposition}

What we have now are charts $B_z$ endowed with some Euclidean coordinates $(y_1,...,y_6)$. In the following we introduce holomorphic coordinates $(w_1,w_2,w_3)$ on $B_z$ (with possibly a smaller common radius) so each $B_z$ can be regarded as a copy of the ball $B$ in Section 2. From the construction above for $z\in X_t\backslash V_t(R|t|^\frac{1}{2},\frac{3}{4})$ we can simply take $w_i=y_i+\sqrt{-1}y_{i+3}$ for $i=1,2,3$. For $z\in V_t(R|t|^\frac{1}{2},\frac{3}{4})$, by our construction it is actually enough to consider $z \in Q_1$ where $\mathbf{r}_1(z) \geq R$. Moreover, by the homogeneity property of $Q_1$ it is enough to consider $z=(\sqrt{-1}\sqrt{\frac{\mathbf{r}_1^2-1}{2}},0,0,\sqrt{\frac{\mathbf{r}_1^2+1}{2}})\in Q_1$.

The coordinates of each point $(z_1,...,z_4)\in Q_1$ near $z$ satisfy
\begin{equation*}
Z=M_1Z_0M_2
\end{equation*}
where
\begin{equation*}
Z=
\begin{pmatrix}                
  z_1+\sqrt{-1}z_2 & -z_3+\sqrt{-1}z_4 \\
  z_3+\sqrt{-1}z_4 & z_1-\sqrt{-1}z_2 
\end{pmatrix},
\end{equation*}
\begin{equation*}
Z_0=\sqrt{-1}
\begin{pmatrix}                
  \sqrt{\frac{\mathbf{r}_1(s)^2-1}{2}}+\sqrt{\frac{\mathbf{r}_1(s)^2+1}{2}} & 0 \\
  0 & \sqrt{\frac{\mathbf{r}_1(s)^2-1}{2}}-\sqrt{\frac{\mathbf{r}_1(s)^2+1}{2}} 
\end{pmatrix},
\end{equation*}
\begin{equation*}
M_1=
\begin{pmatrix}               
  \cos (\theta_1+\frac{\pi}{4})e^{\sqrt{-1}(\psi+\phi_1)} & -\sin (\theta_1+\frac{\pi}{4})e^{-\sqrt{-1}(\psi-\phi_1)} \\
  \sin (\theta_1+\frac{\pi}{4})e^{\sqrt{-1}(\psi-\phi_1)} & \cos (\theta_1+\frac{\pi}{4})e^{-\sqrt{-1}(\psi+\phi_1)} 
\end{pmatrix}
\end{equation*}
 and
 \begin{equation*}
M_2=
\begin{pmatrix}               
  \cos (\theta_2+\frac{\pi}{4})e^{-\sqrt{-1}\phi_2} & \sin (\theta_2+\frac{\pi}{4})e^{\sqrt{-1}\phi_2} \\
  -\sin (\theta_2+\frac{\pi}{4})e^{-\sqrt{-1}\phi_2} & \cos (\theta_2+\frac{\pi}{4})e^{\sqrt{-1}\phi_2} 
\end{pmatrix}
\end{equation*}
for $(\theta_1,\theta_2,\phi_1,\phi_2,\psi,s)\in B_z$, viewed as the ball of radius $0<\rho \ll 1$ in $\mathbb{R}^6$ centered at 0. Here $\mathbf{r}_1(s)=\mathbf{r}_1\left(\mathbf{r}_0(x_1(z))e^{\frac{3}{2}s}\right)$ as before, and $(\theta_1,\theta_2,\phi_1,\phi_2,\psi)$ form a local coordinate system on $\Sigma$. Explicitly, we have $(y_1,...,y_6)=(\theta_1,\theta_2,\phi_1,\phi_2,\psi,s)$.

Near the point $z=(z_1,...,z_4)=(\sqrt{-1}\sqrt{\frac{\mathbf{r}_1^2-1}{2}},0,0,\sqrt{\frac{\mathbf{r}_1^2+1}{2}})\in Q_1$ we can let $(z_1,z_2,z_3)$ be local holomorphic coordinates. Using the above explicit expressions, we can show that, for some $\rho>0$ small enough independent of $z$, on the ball $B_z$ the rescaled holomorphic coordinates $(w_1,w_2,w_3):=\mathbf{r}_1(z)^{-1}(z_1,z_2,z_3)$ satisfy the following property that there exist constants $\Lambda_k>0$ and $\Lambda_{k,l}>0$ for $k \geq 1$ and $l \geq 0$ independent of $z$ such that as functions in coordinates $(x_1,...,x_6)$ on $B_z$ where $w_i=x_i+\sqrt{-1}x_{i+3}$, $i=1,2,3$, the partial derivatives $\frac{\partial^k y_j}{\partial x_{i_1}...\partial x_{i_k}}$ and $\frac{\partial^l }{\partial x_{i_1}...\partial x_{i_l}}\left(\frac{\partial^k x_i}{\partial y_{j_1}...\partial y_{j_k}}\right)$ satisfy
\begin{equation*}
		\left| \frac{\partial^k y_j}{\partial x_{i_1}...\partial x_{i_k}} \right|\leq \Lambda_k\,\,\text{and}\,\,
		\left| \frac{\partial^k }{\partial x_{i_1}...\partial x_{i_k}}\left(\frac{\partial^l x_i}{\partial y_{j_1}...\partial y_{j_l}}\right) \right|\leq \Lambda_{k,l}
\end{equation*}
for $k \geq 1$ and $l \geq 0$. Moreover, there is a constant $\Lambda_0>0$ independent of $z$ such that
		\begin{equation*}
				\frac{1}{\Lambda_0}\leq \frac{\partial (y_1,...,y_6)}{\partial (x_1,...,x_6)} \leq \Lambda_0
		\end{equation*}
on $B_z$.

These properties are not affected if we make a shift in the coordinates $(x_1,...,x_6)$, and we do so to have $B_z$ centered at the origin of $\mathbb{R}^6\cong \mathbb{C}^3$.   We can easily see from the above properties that for some possibly smaller choice of $\rho>0$, the version of Theorem \ref{th:0} with $\tilde{g}_t$ replaced by $g_t$ holds on each $B_z$ endowed with the coordinates $(w_1,w_2,w_3)$ and with $\nabla_e$ now understood as the Euclidean derivative w.r.t. $(w_1,w_2,w_3)$. This is what we'll always have in mind from now on when we work in the charts $B_z$, and in all our later calculations on $B_z$ the coordinates $(w_1,w_2,w_3)$ will always be understood as the choice of holomorphic coordinates introduced here unless stated otherwise.\\[0.3cm]

\noindent \textbf{Remark} 
 For simplicity, in the following we will identify $B_z$ with its image $U_z$ under $\phi_{t,z}$. In particular, $B_z$ can also be regarded as a subset of $X_t$ if $z\in X_t$, and the pullback sign $\phi_{t,z}^*$ will be omitted without causing confusion.\\[0.3cm]

We proceed to prove the original version of Theorem \ref{th:0}. Recall that the Hermitian form $\tilde{\omega}_t$ of the balanced metric $\tilde{g}_t$ on $X_t$ satisfies $\tilde{\omega}_t^2=\omega_t^2+\theta_t+\bar{\theta}_t$ where $\theta_t=\partial\bar{\partial}^*\partial^*\gamma_t$ for some (2,3)-form $\gamma_t$ satisfying the equations $E_t(\gamma_t)=-\partial \Phi^{1,3}_t$ and $\partial \gamma_t=0$, where
\begin{equation*}
    E_t=\partial\bar{\partial}\bar{\partial}^*\partial^*+\partial^*\bar{\partial}\bar{\partial}^*\partial+\partial^*\partial
\end{equation*}
and the $\ast$-operators are with respect to the metric $g_t$. Moreover, $\partial \Phi^{1,3}_t$ is supported on $X_t[1]$ and there is a constant $C>0$ such that 
\begin{equation}
 \label{eq:1}
    |\partial \Phi^{1,3}_t|_{C^k}<C|t|.
\end{equation}

For an arbitrary Hermitian metric $g$ with Hermitian form $\omega$, in a complex coordinate system $(w_1,w_2,w_3)$ we have
\begin{equation*}
		\omega=\frac{\sqrt{-1}}{2}\sum_{1\leq i,j\leq 3}g_{i\bar{j}}dw_i\wedge d\bar{w}_j.
\end{equation*}
Write 
\begin{equation*}
		\omega^2=-\frac{1}{2}\sum_{1\leq i,j\leq 3}G_{i\bar{j}} dw_1\wedge d\bar{w}_1\wedge ...\wedge \widehat{dw_i} \wedge...\wedge\widehat{d\bar{w}_j}\wedge...\wedge dw_3\wedge d\bar{w}_3,
\end{equation*}
then each $g_{i\bar{j}}$ is a polynomial in the $G_{i\bar{j}}$'s and $\det(G_{i\bar{j}})^{-\frac{1}{2}}$. With this elementary fact in mind Theorem \ref{th:0} follows from its version for $g_t$ and
\begin{proposition}
 \label{pr:1}
   For given $k\geq0$, there is a constant $C>0$ which may depend on $k$ such that
   \begin{equation*}
       \Arrowvert \mathbf{r}_t(z)^{-\frac{8}{3}}\theta_t \Arrowvert_{C^k(B'_z,g_e)}<C|t|^\frac{1}{3}
   \end{equation*}
   for any $z \in X_t$ when $t\neq 0$ sufficiently small. Here $B'_z\subset B_z$ is the ball centered at 0 with radius $\frac{\rho}{2}$.
\end{proposition}

\begin{proof}
It is enough to prove for $z\in V_t(\frac{1}{8})$. Let $\Delta_{\bar{\partial}}=\bar{\partial}\bar{\partial}^*+\bar{\partial}^*\bar{\partial}$ be the $\bar{\partial}$-Laplacian w.r.t. $g_t$. Over the region $V_t(1)$ where $g_t$ is just the CO-metric $g_{co,t}$, we have $\Delta_{\bar{\partial}}\theta_t=0$ since 
\begin{equation*}
		\bar{\partial}^*\theta_t=\bar{\partial}^*\partial\bar{\partial}^*\partial^*\gamma_t=\partial\bar{\partial}^*\bar{\partial}^*\partial^*\gamma_t=0
\end{equation*}
and 
\begin{equation}
 \label{th-1}
		\bar{\partial}\theta_t=\bar{\partial}\partial\bar{\partial}^*\partial^*\gamma_t=-E_t(\gamma_t)=\partial \Phi^{1,3}_t=0.
\end{equation}
The second equality of the second line follows because $\partial \gamma_t=0$.

The operator 
\begin{equation*}
		\mathbf{r}_t^\frac{4}{3}\Delta_{\bar{\partial}}:\Gamma(X_t, \Omega^{2,2}) \rightarrow \Gamma(X_t, \Omega^{2,2}) 
\end{equation*}
is elliptic. In general, given a $(p,q)$-form $\psi=\sum \psi_{\alpha_1...\bar{\beta}_q}dw_{\alpha_1}\wedge...\wedge d\bar{w}_{\beta_q}$, Kodaira's Bochner formula says
\begin{equation*}
    \begin{split}
		(\Delta_{\bar{\partial}} \psi)_{\alpha_1...\bar{\beta}_q}
		=&-\sum_{\alpha,\beta}g^{\bar{\beta}\alpha}\nabla_{\alpha}\nabla_{\bar{\beta}}\psi_{\alpha_1...\bar{\beta}_q}\\
		  &+\sum_{i=1}^p\sum_{k=1}^q \sum_{\alpha,\beta} {{R^{\alpha}}_{\alpha_i\bar{\beta}_k}}^{\bar{\beta}}\psi_{\alpha_1...\alpha_{i-1}\alpha\alpha_{i+1}...\bar{\beta}_{k-1}\bar{\beta}\bar{\beta}_{k+1}...\bar{\beta}_q}\\
		  &-\sum_{k=1}^q \sum_{\beta} {{R_{\bar{\beta}_k}}}^{\bar{\beta}}\psi_{\alpha_1...\bar{\beta}_{k-1}\bar{\beta}\bar{\beta}_{k+1}...\bar{\beta}_q}.
    \end{split}
\end{equation*}
Applying this to $\psi=\theta_t=\sum \theta_{\alpha_1\alpha_2\bar{\beta}_1\bar{\beta}_2}dw_{\alpha_1}\wedge dw_{\alpha_2}\wedge d\bar{w}_{\beta_1}\wedge d\bar{w}_{\beta_2}$ and using (\ref{th-1}), we have 
\begin{equation}
 \label{th-2}
    \begin{split}
		&\mathbf{r}_t^\frac{4}{3}\sum_{\alpha,\beta}g^{\bar{\beta}\alpha}\nabla_{\alpha}\nabla_{\bar{\beta}}\theta_{\alpha_1\alpha_2\bar{\beta}_1\bar{\beta}_2}\\
		  &- \sum_{\alpha,\beta} \mathbf{r}_t^\frac{4}{3}\left( {{R^{\alpha}}_{\alpha_1\bar{\beta}_2}}^{\bar{\beta}}\theta_{\alpha\alpha_2\bar{\beta}_{1}\bar{\beta}}
		  + {{R^{\alpha}}_{\alpha_2\bar{\beta}_2}}^{\bar{\beta}}\theta_{\alpha_1\alpha\bar{\beta}_{1}\bar{\beta}}
		  +  {{R^{\alpha}}_{\alpha_1\bar{\beta}_1}}^{\bar{\beta}}\theta_{\alpha\alpha_2\bar{\beta}\bar{\beta}_{2}}
		  +  {{R^{\alpha}}_{\alpha_2\bar{\beta}_1}}^{\bar{\beta}}\theta_{\alpha_1\alpha\bar{\beta}\bar{\beta}_{2}} \right)\\
		  &+ \sum_{\beta} \mathbf{r}_t^\frac{4}{3}\left( {{R_{\bar{\beta}_1}}}^{\bar{\beta}}\theta_{\alpha_1\alpha_2\bar{\beta}\bar{\beta}_2}+ {{R_{\bar{\beta}_2}}}^{\bar{\beta}}\theta_{\alpha_1\alpha_2\bar{\beta}_1\bar{\beta}} \right)=0.
    \end{split}
\end{equation}
The first term above can be written as

\begin{equation}
 \label{th-3}
     \begin{split}
		&\sum_{\alpha,\beta}\mathbf{r}_t^\frac{4}{3}g^{\bar{\beta}\alpha}\nabla_{\alpha}\nabla_{\bar{\beta}}\theta_{\alpha_1\alpha_2\bar{\beta}_1\bar{\beta}_2}\\
	=&\mathbf{r}_t^\frac{4}{3}g^{\bar{\beta}\alpha}\frac{\partial}{\partial w_\alpha}\frac{\partial}{\partial \bar{w}_\beta}\theta_{\alpha_1\alpha_2\bar{\beta}_1\bar{\beta}_2}
		+\,\text{remaining terms},
	\end{split}
\end{equation}
where the remaining terms involve derivatives of $\theta_{\alpha_1\alpha_2\bar{\beta}_1\bar{\beta}_2}$ of order 1 or less, with coefficients bounded as in Proposition \ref{pr:uniform} for constants $\Lambda_k$ independent of $z$ and $t\neq 0$.
Note that the products of $\mathbf{r}_t^\frac{4}{3}$ and the curvature terms in (\ref{th-2}) are bounded similarly. Therefore, $\theta_t$ is the zero of the elliptic operator $\mathbf{r}_t^\frac{4}{3}\Delta_{\bar{\partial}}$ whose coefficients are bounded as in Proposition \ref{pr:uniform} for constants $\Lambda_k$ independent of $z$ and $t\neq 0$. We use the Hermitian metric on $\Omega^{2,2}$ induced by $g_e$. Then there are constants $C_{p,k}>0$ such that
\begin{equation*}
      \Arrowvert \theta_t \Arrowvert_{L^p_{k+2}(B_z,g_e)} \leq C_{p,k}\Arrowvert \theta_t \Arrowvert_{L^2(B_z,g_e)}
\end{equation*}
for $z\in V_t(\frac{1}{8})$ (so $B_z\subset V_t(\frac{1}{4})$ for $\rho$ small enough, which we assume is the case). Each $C_{p,k}$ is independent of $z$ and $t$ since we use the Euclidean metric in each chart. By the usual Sobolev Theorem over the Euclidean ball $(B_z, g_e)$, for $p$ large enough one can get 
\begin{equation*}
   \begin{split}
      \Arrowvert \theta_t \Arrowvert_{C^{k}(B'_z,g_e)} 
      \leq &C'_{p,k}\Arrowvert \theta_t \Arrowvert_{L^2(B_z,g_e)}\\
      =&C'_{p,k}\left (\int_{B_z} |\theta_t|_{g_e}^2dV_e \right)^\frac{1}{2} \leq C'_{p,k}\text{Vol}_e(B_z)^\frac{1}{2} \sup_{B_z}|\theta_t|_{g_e}
      \end{split}
\end{equation*}
for some constants $C'_{p,k}>0$ independent of $z$ and $t$. Therefore,
\begin{equation}
 \label{eq:theta-1}
      \Arrowvert \mathbf{r}_t(z)^{-\frac{8}{3}}\theta_t \Arrowvert_{C^{k}(B_z,g_e)} \leq C'_{p,k}\text{Vol}_e(B_z)^\frac{1}{2} \sup_{B_z}|\mathbf{r}_t(z)^{-\frac{8}{3}}\theta_t|_{g_e}.
\end{equation}

From (\ref{eq:0-1}) one sees easily that
\begin{equation*}
		|\theta_t|_{g_{co,t}}^2 \leq |t|^\frac{2}{3} 
\end{equation*}
for $t\neq 0$ sufficiently small, and by Proposition \ref{pr:n-1} this implies
\begin{equation}
 \label{eq:theta-2}
		|\mathbf{r}_t(z)^{-\frac{8}{3}}\theta_t|_{g_e}^2 \leq C|t|^\frac{2}{3} 
\end{equation}
for $t \neq 0$ sufficiently small. Now (\ref{eq:theta-1}) and (\ref{eq:theta-2}) complete the proof.\qed
\end{proof}

In later section we will need the following result on the sup norm of $\theta_t$:
\begin{proposition}
 \label{pr:6-1}
    There is a constant $C>0$  independent of $t$ such that 
\begin{equation*}
			|\theta_t|_{g_t}  \leq  C\mathbf{r}_t^{-\frac{2}{3}}\cdot  |t|.
\end{equation*}
Consequently, there is a constant $C>0$ such that
\begin{equation*}
		|\tilde{\omega}_t^{-1}-\omega_t^{-1}|_{g_t} \leq C\mathbf{r}_t^{-\frac{2}{3}}\cdot  |t|.
\end{equation*}
 
\end{proposition}
\begin{proof}
Again, it is enough to consider over $B_z$ for $z\in V_t(\frac{1}{8})$. A similar discussion as in Proposition \ref{pr:1} shows that for each $z\in V_t(\frac{1}{8})$ we have
\begin{equation}
 \label{eq:theta-3}
   \begin{split}
      \sup_{B'_z} |\mathbf{r}_t^\frac{2}{3}\theta_t|_{g_t}
      \leq &C\Arrowvert \theta_t \Arrowvert_{C^0_{3}(B'_z,g_e)} 
      \leq C'\Arrowvert \theta_t \Arrowvert_{L^2_{3}(B_z,g_e)}
      =C'\left (\int_{B_z} |\mathbf{r}_t^{-2}\theta_t|_{g_e}^2dV_e \right)^\frac{1}{2}\\
      =&C''\left (\int_{B_z} |\mathbf{r}_t^\frac{2}{3}\theta_t|_{g_t}^2\mathbf{r}_t^{-4}dV_t \right)^\frac{1}{2}
      \leq C''\left (\int_{V_t(\frac{1}{4})} |\theta_t|_{g_t}^2\mathbf{r}_t^{-\frac{8}{3}}dV_t \right)^\frac{1}{2}.
    \end{split}
\end{equation}

It is proved in Lemma 17 of \cite{FLY} that
\begin{equation*}
		\int_{V_t(\frac{1}{4})} |\theta_t|_{g_t}^2\mathbf{r}_t^{-\frac{8}{3}}dV_t \leq C\int_{X_t[\frac{1}{8}]}(|\gamma_t|_{g_t}^2+|\partial \Phi^{1,3}_t |_{g_t}^2)dV_t
\end{equation*}
for some constant $C>0$ independent of $t$. In view of (\ref{eq:1}), to prove the proposition it is enough to show
\begin{equation*}
		\int_{X_t} |\gamma_t|_{g_t}^2dV_t \leq C|t|^2
\end{equation*}
for some constant $C>0$ independent of $t$. Suppose that there is a sequence $\{t_i\}$ converging to 0 such that
\begin{equation*}
    |t_i|^{-2}\int_{X_{t_i}} |\gamma_{t_i}|^2dV_{t_i}=\alpha_i^2\rightarrow \infty\,\,\,\text{when}\,\,i\rightarrow \infty.
\end{equation*}
where $\alpha_1>0$. Define $\tilde{\gamma}_{t_i}=|t_i|^{-1}\alpha_i^{-1}\gamma_{t_i}$ then
\begin{equation*}
    \int_{X_{t_i}} |\tilde{\gamma}_{t_i}|^2dV_{t_i}=1\,\,\,\text{and}\,\,E_{t_i}(\tilde{\gamma}_{t_i})=-|t_i|^{-1}\alpha_i^{-1}\partial \Phi^{1,3}_{t_i}.
\end{equation*}
Thus there exists a smooth (2,3)-form $\tilde{\gamma}_0$ on $X_{0,sm}$ such that $E_0(\tilde{\gamma}_0)=0$ and $\tilde{\gamma}_{t_i} \rightarrow \tilde{\gamma}_0$ pointwise. Then one can prove that
\begin{equation*}
    \int_{X_{0,sm}} |\tilde{\gamma}_0|^2dV_0=1\,\,\,\text{but}\,\,\tilde{\gamma}_0=0.
\end{equation*}
as in \cite{FLY} in exactly the same way, only noticing that in several places we use the fact that $|t_i|^{-2}\alpha_i^{-2}|\partial \Phi^{1,3}_{t_i}|^2 \rightarrow 0$ as $i \rightarrow \infty$. This completes the proof. \qed 
\end{proof}

\section{HYM metrics on the vector bundle over $X_0$}

Let $\mathcal{E}$ be an irreducible holomorphic vector bundle over the K$\ddot{\text{a}}$hler Calabi-Yau threefold $(\hat{X},\omega)$ as before. Our assumption on $\mathcal{E}$ is that it is trivial over a neighborhood of the exceptional curves $C_i$'s. By a rescaling of the metric $\hat{\omega}_0$, we may assume that $\mathcal{E}$ is trivial over $U(1)\subset \hat{X}$. As mentioned in Section 2, over $\hat{X}$ there is a 1-parameter family of balanced metrics $\hat{\omega}_a$, $0< a \ll 1$, constructed as in \cite{FLY}. Since for each $a\neq 0$ the (2,2)-forms $\hat{\omega}_a^2$ and $\omega^2$ differ by smooth $\partial \bar{\partial}$-exact forms, the bundle $\mathcal{E}$ is stable with respect to all $\hat{\omega}_a$ if it is so with respect to $\omega$. Assume that this is the case. Then by the result of \cite{LY1}, there exists a HYM metric $H_a$ on $\mathcal{E}$ with respect to $\hat{\omega}_a$.

In this section, $\hat{H}$ will be a metric such that $\hat{H}=I$ with respect to some a constant frame over $U(1)$ where $\mathcal{E}$ is trivial. By a constant frame we mean the following: under an isomorphism $\mathcal{E}|_{U(1)}\cong \mathcal{O}_{U(1)}^r$, a holomorphic section of $\mathcal{E}$ over $U(1)$ can be viewed as a holomorphic vector-valued function on $U(1)$. Then a constant frame $\{s_1,...,s_r\}$ is a set of such functions which are (pointwise) linearly independent and  each member $s_i$ is a constant (vector-valued) function. A constant frame is in particular a holomorphic frame.

The metric $\hat{H}$ will serve as the reference metric. The constants appearing in this section may depend on $\hat{H}$. We will also often use implicitly the identification $\hat{X}\backslash \bigcup C_{i} \cong X_{0,sm}$.

\subsection{Proof of the first main theorem} The goal of this subsection is to prove the following theorem on the existence of a HYM metric with respect to $\hat{\omega}_0$ over $\mathcal{E}|_{X_{0,sm}}$. The techniques we use are largely based on \cite{Don1} \cite{Don2} \cite{Don3} \cite{Siu} \cite{UY}.
\begin{theorem}
 \label{th:2-1}
    There is a smooth Hermitian metric $H_0$ on $\mathcal{E}|_{X_{0,sm}}$ which is HYM with respect to $\hat{\omega}_0$ such that there is a decreasing sequence $\{a_i\}_{i=1}^\infty$ converging to 0 for which a sequence $\{H_{a_i}\}$ of HYM metrics (w.r.t. $\hat{\omega}_{a_i}$, respectively) converge weakly to $H_0$ in the $L^p_2$-sense for all $p$ on each compactly embedded open subset of $X_{0,sm}$.
\end{theorem}
\begin{proof}
We begin with a boundedness result on the determinants of $h_a:=H_a\hat{H}^{-1}$.
\begin{lemma}
 \label{lm:d-1}
  After a rescaling $H_a$ by positive constants we can assume that $\det h_a$ are bounded from above and below by positive constants independent of $0<a\ll 1$.
\end{lemma}
\begin{proof}
    Let $\varphi_a$ be the unique smooth function on $\hat{X}$ satisfying 
\begin{equation*}
		\hat{\Delta}_a\varphi_a=-\frac{\sqrt{-1}}{r}\text{tr}\,\Lambda_{\hat{\omega}_a}F_{\hat{H}}
\end{equation*}
and $\int_{\hat{X}} \varphi_a   dV_a=0$ 
    where $dV_a$ is the volume form of $\hat{g}_a$. \\
    
\noindent \textbf{Claim} The sup norm of $\varphi_a$ is bounded by a constant independent of $0<a\ll 1$.
        
\begin{proof}
First note that since $\Lambda_{\hat{\omega}_a} F_{\hat{H}}=0$ on $U(1)$, $\varphi_a$ is harmonic over $U(1)$ and so we have by the maximum principle $\sup_{U(1)}|\varphi_a| \leq \sup_{X_0[\frac{3}{4}]}|\varphi_a|$. 
    
Since $\hat{\omega}_a$ is a balanced metric the Laplacian $\hat{\Delta}_a$ coincides (up to a constant multiple) with the negative of the Laplace-Beltrami operator its associated Riemannian metric (see for example \cite{Ga}). We thus have the Greens formula \cite{Au}: for each $x \in X_0[\frac{3}{4}]$, $0<a\ll 1$, $\frac{1}{4}\leq\delta\leq\frac{1}{2}$, and smooth function $f$ on $\hat{X}$,
\begin{equation}
 \label{eq:green}
      f(x)=  \int_{\partial X_0[\delta]} \Gamma_{a,\delta}(x,y) f(y) dS_a(y)+\int_{y\in X_0[\delta]}\,G_{a,\delta}(x,y)\hat{\Delta}_a f(y)\, dV_a(y) 
\end{equation}
where $G_{a,\delta}(x,y)\leq 0$ is the Green's function for $\hat{\Delta}_a$ over the region $X_0[\delta]$, and $\Gamma_{a,\delta}$ is the boundary normal derivative of $G_{a,\delta}(x,y)$ with respect to $y$. Moreover, $dS_a$ is the volume form on $\partial X_0[\delta]$ with respect to the metric induced from $\hat{g}_a$. 

We apply the above formula to $f=\varphi_a$. Since the family of metrics $\{\hat{\omega}_a| 0<a\ll 1\}$ are uniform over $X_0[\frac{1}{4}]$ there is a constant $K_0$ such that for any $0<a\ll 1$, $\frac{1}{4}\leq\delta\leq\frac{1}{2}$, $y \in \partial X_0[\delta]$ and $x \in X_0[\frac{3}{4}]$,
\begin{equation*}
		|\Gamma_{a,\delta}(x,y)| \leq K_0.
\end{equation*}
For the same reason there is a constant $K_1>0$ such that
\begin{equation*}
      -\int_{y\in X_0[\delta]}\,G_{a,\delta}(x,y)\, dV_a(y) \leq K_1
\end{equation*}
for any $x \in X_0[\frac{3}{4}]$, $\frac{1}{4}\leq\delta\leq\frac{1}{2}$, and $0<a\ll 1$.

Because $\Lambda_{\hat{\omega}_a} F_{\hat{H}}=0$ over $U(1)$, $|\frac{1}{r}\text{tr}\,\Lambda_{\hat{\omega}_a} F_{\hat{H}}|$ is bounded by a constant $K_2>0$ independent of $a$. Therefore we have
\begin{equation*}
         | G_{a,\delta}(x,y)\hat{\Delta}_a \varphi_a| \leq -G_{a,\delta}(x,y) |\frac{1}{r}\text{tr}\,\Lambda_{\hat{\omega}_a} F_{\hat{H}}| \leq -K_2\cdot G_{a,\delta}(x,y).
\end{equation*}
We can conclude from the above bounds that
\begin{equation}
 \label{eq:d-1}
   \begin{split}
         | \varphi_a(x) | \leq & K_0\int_{\partial X_0[\delta]} |\varphi_a| dS_a - \int_{y\in X_0[\delta]}\,K_2\cdot G_{a,\delta}(x,y) dV_a(y) \\
                       \leq    & K_0\int_{\partial X_0[\delta]} |\varphi_a|\, dS_a+K_1K_2. \\     
   \end{split}
\end{equation}

Integrate (\ref{eq:d-1}) with respect to $\delta$ from $\frac{1}{4}$ to $\frac{1}{2}$ and use once again the uniformity in the metrics over $X_0[\frac{1}{4}]$ we obtain
\begin{equation}
    \begin{split}
        | \varphi_a(x) |  \leq & K_4\left( \int_{X_0[\frac{1}{2}] \backslash X_0[\frac{1}{4}] } | \varphi_a |  dV_a \right)+4K_1K_2  \\
                 \leq & K_4K_5^\frac{1}{2}\left( \int_{\hat{X}} | \varphi_a |^2 dV_a \right)^\frac{1}{2}+4K_1K_2 
     \end{split}
\end{equation}
for each $x \in X_0[\frac{3}{4}]$. Here $K_5$ is a comment upper bound for the volumes of $\hat{X}$ w.r.t. $\hat{g}_a$. Now, to prove the claim, we have to show that $\int_{\hat{X}} | \varphi_a |^2 dV_a$ is bounded by a constant independent of $0<a\ll 1$. For this we use the estimates on the first eigenvalue of Laplacians due to Yau \cite{Y1} which implies that for a compact Riemannian manifold $(X,g)$ of dimension $n$, if (i) $\text{diag}_g(X) \leq D_1$, (ii) $\text{Vol}_g(X) \geq D_2$ and (iii) $Ric(g)\geq (n-1)K$ hold, then the number
		\begin{equation*}
					\lambda_1:=\inf_{0\neq f\in C^\infty(X),\int_{X} f   dV_g=0} \frac{\int_{X} f\Delta_g^{LB}f \, dV_g}{\int_{X} f^2 \, dV_g}.
		\end{equation*}	
is bounded below by a constant depending only on $D_1$, $D_2$ and $K$. Here $\Delta_g^{LB}$ denotes the Laplace-Beltrami operator of $g$.

For the family of metrics $\{\hat{g}_a\}$ on $\hat{X}$, it is easy to see that the diameters and volumes are bounded as in (i) and (ii) by the same constants $D_1$ and $D_2$. Note that in a neighborhood of the exceptional curves each member $\hat{g}_a$ is Ricci-flat, and so by the uniformity outside that neighborhood, condition (iii) holds for a common value of $K$.

Therefore, there is a constant $K_6>0$ such that 
\begin{equation*}
   \begin{split}
		\int_{\hat{X}} | \varphi_a |^2 dV_a 
		&\leq K_6 \int_{\hat{X}} | \varphi_a||\hat{\Delta}_a\varphi_a|  dV_a = K_6\int_{\hat{X}}  |\varphi_a| |\frac{1}{r}\text{tr}\,\Lambda_{\hat{\omega}_a} F_{\hat{H}}| dV_a \\
		&\leq K_6K_2\int_{\hat{X}} | \varphi_a | dV_a \leq K_6K_2K_5^\frac{1}{2}\left( \int_{\hat{X}} | \varphi_a |^2 dV_a \right)^\frac{1}{2}
	\end{split}
\end{equation*}
and hence
\begin{equation*}
		\left( \int_{\hat{X}} | \varphi_a |^2 dV_a \right)^\frac{1}{2} \leq K_6K_2K_5^\frac{1}{2}.
\end{equation*}
This completes the proof of the claim. \qed
\end{proof}

Define $\hat{H}_a:=e^{\varphi_a}\hat{H}$. Then it follows from the claim that to prove the lemma, it is enough to show that the determinants of $\hat{h}_a:=H_a\hat{H}_a^{-1}$ have common positive upper and lower bounds.

To do so first note that we have $\text{tr}\,\Lambda_{\hat{\omega}_a}F_{\hat{H}_a}=0$. Then the proof of Proposition 2.1 in \cite{UY} shows that this and the fact that $\Lambda_{\hat{\omega}_a}F_{H_a}=0 $ imply $\det \hat{h}_a$ is constant for each $a$. After a rescaling of $H_a$ by a positive constant, we can assume $\det \hat{h}_a=1$, and the proof of Lemma \ref{lm:d-1} is complete.\qed
\end{proof} 

From now on we assume that the rescaling in the above lemma is done. We next show a result on the $C^0$-bound for $\text{tr}\,h_a$.
\begin{proposition}
 \label{pr:2-1}
    Assume that the integrals $\int_{\hat{X}} |\log\text{tr} \,h_a|^2 \,dV_a$ have a common upper bound for $0<a \ll1$. Then there is a constant $C_0>0$ such that for any $0<a \ll1$,
     \begin{equation*}
         -C_0< \log\text{tr}\,h_a<C_0.
     \end{equation*}
\end{proposition}
\begin{proof}
First of all, we have the following inequality whose proof can be found in \cite{Siu}:
\begin{lemma}
 \label{lm:main}
  Let $H_0$ and $H_1$ be two Hermitian metrics on a holomorphic vector bundle $\mathcal{E}$ over a Hermitian manifold $(X,\omega)$, and define $h=H_1H_0^{-1}$. Then
    \begin{equation}
        \Delta_{\omega} \log\text{tr}\,h \geq -(|\Lambda_{\omega} F_{H_0}|_{H_0}+|\Lambda_{\omega} F_{H_1}|_{H_0}).
    \end{equation}
\end{lemma}

By Lemma \ref{lm:main}, we have the inequality
    \begin{equation}
     \label{eq:2-3}
        \hat{\Delta}_a \log\text{tr} (h_a) \geq -(|\Lambda_{\hat{\omega}_a} F_{H_a}|_{\hat{H}}+|\Lambda_{\hat{\omega}_a} F_{\hat{H}}|_{\hat{H}})=-|\Lambda_{\hat{\omega}_a} F_{\tilde{H}}|_{\hat{H}}
    \end{equation}
where the equality follows since $H_a$ is HYM with respect to $\hat{\omega}_a$.

Over $U(\frac{7}{8})$ we have $\hat{\Delta}_a \log\text{tr}\,h_a \geq -|\Lambda_{\hat{\omega}_a} F_{\hat{H}}|_{\hat{H}}=0$ and so by Maximum Principle we have
\begin{equation*} 
  \begin{split}
     \sup_{U(\frac{7}{8})} \log\text{tr}\,h_a \leq \sup_{\partial U(\frac{7}{8})} \log\text{tr}\,h_a \leq \sup_{  X_0[\frac{3}{4}]} \log\text{tr}\,h_a.
  \end{split}
\end{equation*}

Using the Green's formula (\ref{eq:d-1}), we can show as in Lemma \ref{lm:d-1} that 
\begin{equation}
		\sup_{  X_0[\frac{3}{4}]} \log\text{tr}\,h_a \leq K'
\end{equation}
for some $K'>0$ independent of $0<a \ll1$ assuming that the integrals $\int_{\hat{X}} |\log\text{tr} \,h_a|^2 \,dV_a$ have a common upper bound. We thus have a commen upper bound for $\sup_{\hat{X}} \log\text{tr}\,h_a$. 

Together with the fact that the determinants of $h_a$ are bounded from above and below by positive constants independent of $0<a \ll1$, this upper bound also implies a common lower bound for $\log\text{tr}\,h_a$ over $\hat{X}$. The proof is completed. \qed

\end{proof}

Therefore, to get $C^0$-estimate we prove
\begin{proposition}
 \label{pr:L2}
     There is a constant $C'_0>0$ such that
     \begin{equation*}
             \int_{\hat{X}} |\log\text{tr}\,h_a|^2 dV_a<C'_0
     \end{equation*}
     for any $0<a\ll 1$.
\end{proposition}
\begin{proof}
  The idea is basically the same as in the proof of Proposition 4.1 in \cite{UY}. Assume the contrary. Then there is a sequence $\{ a_k \}_{k=1}^\infty$ converging to 0 such that $\lim_{k\rightarrow \infty} \int_{\hat{X}} |\log\text{tr}\,h_{a_k}|^2 \,dV_{a_k}=\infty$. Denote $h(k)=h_{a_k}$, and define $\rho_k=e^{-M_k}$ where $M_k$ is the largest eigenvalue of $\log h(k)$. Then $\rho_k h(k) \leq I$. 
  
The following inequality is proved in Lemma 4.1 of \cite{UY}: 
\begin{lemma}
Suppose
\begin{equation*}
     \Lambda_{\omega} F_H+\Lambda_{\omega} \bar{\partial}((\partial_H h)h^{-1})=0
\end{equation*}
holds for a Hermitian metric $H$ on a vector bundle $\mathcal{E}$ over a Hermitian manifold $(X,\omega)$ and $h\in \Gamma(\text{End}(E))$. Then for $0<\sigma \leq 1$, we have the inequality
\begin{equation*}
     |h^{-\frac{\sigma}{2}}\partial_H h^\sigma|_{H,g}^2-\frac{1}{\sigma}\Delta_{\omega}|h^\sigma|_H \leq -\langle \Lambda_{\omega} F_H, h^\sigma \rangle_H.
\end{equation*}
\end{lemma}

In our case, because $\Lambda_{\hat{\omega}_a} F_{\hat{H}}+\Lambda_{\hat{\omega}_a} \bar{\partial}((\partial_{\hat{H}} h_a)h_a^{-1})=\Lambda_{\hat{\omega}_a} F_{H_a}=0$, apply the above lemma to $\sigma=1$, we see immediately that
\begin{equation}
 \label{eq:s-1}
    -\hat{\Delta}_a |h_a|_{\hat{H}} \leq |\Lambda_{\hat{\omega}_a} F_{\hat{H}}|_{\hat{H}} |h_a|_{\hat{H}} \leq K_7|h_a|_{\hat{H}}
\end{equation}
where $K_7$ is a common upper bound for $|\Lambda_{\hat{\omega}_{a}} F_{\hat{H}}|_{\hat{H}} $. 

Note that $|h_a|_{\hat{H}}$ is subharmonic in $U(\frac{7}{8})$ because of the first inequality in (\ref{eq:s-1}) and the fact that $\Lambda_{\hat{\omega}_a} F_{\hat{H}}=0$ there. Maximum Principle then implies that 
\begin{equation*}
		\sup_{U(\frac{7}{8})} |h_a|_{\hat{H}} \leq \sup_{X_0[\frac{3}{4}]} |h_a|_{\hat{H}}. 
\end{equation*}

From this observation and an iteration argument over $X_0[\frac{3}{4}]$ on (\ref{eq:s-1}), we can deduce that
\begin{equation*}
    \sup_{\hat{X}} |h_a|_{\hat{H}} \leq K_8\left( \int_{X_0[\frac{1}{4}]} |h_a|_{\hat{H}}^2 dV_a \right)^\frac{1}{2}.
\end{equation*}
This implies 
\begin{equation}
 \label{eq:2-2}
    1 \leq K_8\left( \int_{X_0[\frac{1}{4}]} | \rho_k h(k)|_{\hat{H}}^2 dV_{a_k} \right)^\frac{1}{2}.
\end{equation}
for any $k>0$.

As in page S275 of \cite{UY}, one can show that
\begin{equation*}
    \int_{\hat{X}} | \nabla_{\hat{H}}(\rho_k h(k))|_{\hat{H},\hat{g}_{a_k}}^2 dV_{a_k} \leq 4 \max_{\hat{X}} |\Lambda_{\hat{\omega}_{a_k}} F_{\hat{H}}|_{\hat{H}} \cdot \text{Vol}_{a_k}(\hat{X}) \leq 4K_7K_5
\end{equation*}
where $K_5$ is as in the proof of Lemma \ref{lm:d-1}. 

Thus we see that the $L^2_1$-norms of $\rho_k h(k)$ with respect to $\hat{\omega}_{a_k}$ are bounded by a constant independent of $k$. Because the sequence of metrics $\{ \hat{\omega}_{a_k} \}$ are uniformly bounded only on each compactly embedded open subset in $X_{0,sm}$, a subsequence of the sequence $\{\rho_k h(k)\}$ converges strongly on each subset of this kind. After taking a sequence $\{U_l\subset\subset \hat{X} \}_l$ of exhausting increasing subsets and use the diagonal argument, we obtain a subsequence $\{ \rho_{k_i} h(k_i) \}_{i\geq 1}$ of $\{\rho_k  h(k)\}_{k \geq 1}$ and an $\hat{H}$-symmetric endomorphism $h_\infty$ of $\mathcal{E}$ which is the limit of $\{ \rho_{k_i} h_{k_i}|_{U_l} \}_{i\geq 1}$ in $L^2(U_l,\text{End}(\mathcal{E}))$ for all $l$. From (\ref{eq:2-2}) one immediately sees that $h_\infty$ is nontrivial.

Define $h_i=\rho_{k_i} h(k_i)$. The same argument shows that $h_i^\sigma$ converges weakly in the $L^2_1$ sense on each $U_l$ to some $h_\infty^\sigma$. The uniform bound on the $L^2_1$-norm of $h_i^\sigma$ gives the same bound on $h_\infty^\sigma$ for all $\sigma$. It follows that $I-h_\infty^\sigma$ has a weak limit in $L^2_1$ sense on each $U_l$ for some subsequence $\sigma \rightarrow 0$. We call the limit $\pi$. Similar to \cite{UY} except that we consider integrals over each $U_l$, we can show that $\pi$ gives a weakly holomorphic subbundle of $\mathcal{E}$. More precisely \cite{LT}, there is a coherent subsheaf $\mathcal{F}$ of $\mathcal{E}$ and an analytic subset $S\subset \hat{X}$ (containing the exceptional curves) such that $S$ has codimension greater than 1 in $\hat{X}$, the restriction of $\pi$ to $\hat{X}\backslash S$ is smooth and satisfies $\pi^{*_{\hat{H}}}=\pi=\pi^2$ and $(I-\pi)\bar{\partial}\pi=0$, and finally, the restriction $\mathcal{F}':=\mathcal{F}|_{\hat{X}\backslash S}=\pi|_{\hat{X}\backslash S}(\mathcal{E}|_{\hat{X}\backslash S})\hookrightarrow \mathcal{E}$ is a holomorphic subbundle. The rank of $\mathcal{F}$ satisfies $0<\text{rank} \mathcal{F}<\text{rank} \mathcal{E}$.\\[0.2cm]

Following the argument in \cite{K} p. 181-182 (see also Proposition 3.4.9 of \cite{LT}), we have
\begin{equation*}
		\mu_0
		:=\lim_{\delta\rightarrow 0} \frac{1}{\text{rank}\mathcal{F}}\int_{X_0[\delta]} c_1(\det\mathcal{F},u)\wedge \hat{\omega}_0^2
		=\lim_{\delta\rightarrow 0} \frac{1}{\text{rank}\mathcal{F}}\int_{X_0[\delta]} c_1(\mathcal{F}',\hat{H}_1)\wedge \hat{\omega}_0^2.
\end{equation*}
Here $u$ is some smooth Hermitian metric on the holomorphic line bundle $\det\mathcal{F}$ over $\hat{X}$, and $\hat{H}_1$ is the Hermitian metric on the bundle $\mathcal{F}'$ induced by the metric $\hat{H}$ on $\mathcal{E}$. Using the above construction of $\pi$ by convergence on the $U_l$'s one can show by a slight modification of the arguments in \cite{UY} that
\begin{equation*}
		\lim_{\delta\rightarrow 0} \frac{1}{\text{rank}\mathcal{F}}\int_{X_0[\delta]} c_1(\mathcal{F}',\hat{H}_1)\wedge \hat{\omega}_0^2 \geq 0.\\[0.2cm]
\end{equation*}

\noindent \textbf{Claim} For $0<a\ll 1$, $\mu_{\hat{\omega}_a}(\mathcal{F})\geq 0$. 

\begin{proof} It is enough to show $\mu_{\hat{\omega}_a}(\mathcal{F})=\mu_0$. From the construction of $\hat{\omega}_0$ in \cite{FLY}, we have
\begin{equation*}
   		\hat{\omega}_0^2=\Psi+\Phi_0
\end{equation*}
 where $\Psi$ is a (2,2)-form supported outside $U(1)$ and $\Phi_0$ is a $\partial\bar{\partial}$-exact (2,2)-form which is defined only on $\hat{X}\backslash \cup C_i$, is supported in $U(\frac{3}{2})\backslash \cup C_i$, and equals $\omega_{co,0}^2=\frac{9}{4}\sqrt{-1}\partial\bar{\partial}\mathbf{r}^\frac{4}{3}\wedge \sqrt{-1}\partial\bar{\partial}\mathbf{r}^\frac{4}{3}$ on $U(1)\backslash \cup C_i$. The same construction gives $\hat{\omega}_a$ such that
 \begin{equation*}
   		\hat{\omega}_a^2=\Psi+\Phi_a
\end{equation*}
 where $\Phi_a$ is a smooth $\partial\bar{\partial}$-exact (2,2)-form supported in $U(\frac{3}{2})$ which equals $\omega_{co,a}^2$ on $U(1)$. 
 
 Denote the smooth (1,1)-form $c_1(\det\mathcal{F},u)$ on $\hat{X}$ by $c_1$ and $\text{rank}\mathcal{F}$ by $s$. From the above descriptions we have
\begin{equation*}
   \begin{split}
		&\mu_{\hat{\omega}_a}(\mathcal{F})-\mu_0
		=\lim_{\delta\rightarrow 0} \frac{1}{s}\int_{X_0[\delta]} c_1\wedge \left(\hat{\omega}_a^2-\hat{\omega}_0^2\right)
			=\lim_{\delta\rightarrow 0} \frac{1}{s}\int_{X_0[\delta]} c_1\wedge \left(\Phi_a-\Phi_0\right) \\
		=&\frac{1}{s}\int_{\hat{X}} c_1\wedge \Phi_a-\lim_{\delta\rightarrow 0} \frac{1}{s}\int_{X_0[\delta]} c_1\wedge \Phi_0
		=-\lim_{\delta\rightarrow 0} \frac{1}{s}\int_{X_0[\delta]} c_1\wedge \Phi_0
   \end{split}
\end{equation*}
where the last equality follows from the fact that, as smooth forms on $\hat{X}$, $c_1$ is closed and $\Phi_a$ is exact. One can write $c_1\wedge \Phi_0=d(c_1\wedge \varsigma)$ where $\varsigma$ is a 3-form supported on $U(\frac{3}{2})\backslash \cup C_i$ which equals $\frac{9}{8}(\partial \mathbf{r}^\frac{4}{3}-\bar{\partial}\mathbf{r}^\frac{4}{3})\wedge \partial \bar{\partial}\mathbf{r}^\frac{4}{3}$ on $U(1)\backslash \cup C_i$. By Stokes' Theorem, we have
\begin{equation}
 \label{lim}
		-\lim_{\delta\rightarrow 0} \frac{1}{s}\int_{X_0[\delta]} c_1\wedge \Phi_0
		=-\lim_{\delta\rightarrow 0} \frac{9}{8}\frac{1}{s}\int_{\partial X_0[\delta]} c_1\wedge (\partial \mathbf{r}^\frac{4}{3}-\bar{\partial}\mathbf{r}^\frac{4}{3})\wedge \partial\bar{\partial}\mathbf{r}^\frac{4}{3}.
\end{equation}

An explicit calculation on coordinate charts can then show that the last limit is zero. \qed
\end{proof}

Since $\mu_{\hat{\omega}_a}(\mathcal{E})=0$, we get from this claim a contradiction to the assumption that $\mathcal{E}$ is stable with respect to $\hat{\omega}_a$ and complete the proof of Proposition \ref{pr:L2}.\qed
\end{proof}

We continue with the proof of Theorem \ref{th:2-1}. Using $\hat{H}$ and $\hat{g}_a$ one can define $L^2_1$-norms for $h_a$. The next step is to give an $L^2_1$-boundedness.

\begin{proposition}
 \label{pr:L21}
     The $L^2_1$-norm of $h_a$ over $\hat{X}$ is bounded by some constant $C_2$ independent of $0<a \ll 1$.
\end{proposition}
\begin{proof}
The $C^0$-boundedness obtained above and the common upper bound in Vol$_a(\hat{X})$ imply that the $L^2$ norm of $h_a$ is bounded above by a constant independent of $a$.

Choose a finite number of Hermitian metrics $H^{(\nu)}$ for $1 \leq \nu \leq k$ on $\mathcal{E}$ which are constant in some holomorphic frame 
$\mathcal{E}|_{U(1)} \cong \mathcal{O}^r$ over $U(1)$, such that for any smooth Hermitian metric $K$ on $\mathcal{E}$ the entries of the Hermitian matrix representing $K$ are linear functions of $\text{tr}(K(H^{(\nu)} )^{-1})$, $1 \leq \nu \leq k$, whose coefficients are constants depending only on $H^{(\nu)}$. Denote $h_a^{(\nu)}=H_a(H^{(\nu)})^{-1}$. It is therefore enough to bound the integrals
\begin{equation*}
		\int_{\hat{X}} |d\,\text{tr}\,h_a^{(\nu)}|_{\hat{g}_a}^2  dV_a
\end{equation*}
for $1 \leq \nu \leq k$.

From Lemma \ref{lm:main} and the fact that $H_a$ is HYM w.r.t. $\hat{\omega}_a$, we have
    \begin{equation}
       \begin{split}
          \hat{\Delta}_a \log\text{tr}\,h_a^{(\nu)} 
          \geq &-(|\Lambda_{\hat{\omega}_a} F_{H^{(\nu)}}|_{H^{(\nu)}}+|\Lambda_{\hat{\omega}_a} F_{H_a}|_{H^{(\nu)}})
          \geq -|\Lambda_{\hat{\omega}_a} F_{H^{(\nu)}}|_{H^{(\nu)}},
       \end{split}
    \end{equation}
from which we have the inequality
\begin{equation}
 \label{eq:19-5}
			-\hat{\Delta}_a  \text{tr}\,h_a^{(\nu)} 
			\leq |\Lambda_{\hat{\omega}_a} F_{H^{(\nu)}}|_{H^{(\nu)}}\text{tr}\,h_a^{(\nu)}
			\leq K_9\text{tr}\,h_a^{(\nu)}
\end{equation}
for some constant $K_9>0$. Here the last inequality follows from the fact that $F_{H^{(\nu)}}$ is supported on $\hat{X}\backslash U(1)$, where the $\hat{\omega}_a$ are uniform.

Multiplying $\text{tr} \,h_a^{(\nu)}$ on both sides of the inequality (\ref{eq:19-5}) and using integration by parts, we get
\begin{equation*}
  \begin{split}
     \int_{\hat{X}} |d \,\text{tr} h_a^{(\nu)}|_{\hat{g}_a}^2 dV_a \leq  K_9\int_{\hat{X}} |\text{tr} h_a^{(\nu)}|^2 dV_a.
  \end{split}
\end{equation*}
Finally, write $h_a^{(\nu)}=h_a\hat{H}(H^{(\nu)})^{-1}$, and we see that the result follows from the uniform $C^0$ bound of $h_a$. \qed
\end{proof}
Using the diagonal argument, the uniform boundedness of the $L^2_1$-norm of $h_a$ over $\hat{X}$ implies that there is a sequence $\{ a_i \}_{i\geq 1}$ converging to 0 and an $\hat{H}$-symmetric endomorphism $h_0$ of $\mathcal{E}$ which is the limit of $\{ h_{a_i}|_{U_l} \}_{i\geq 1}$ in $L^2(U_l,\text{End}(\mathcal{E}))$ for all $l$. 

As in \cite{Don1} and \cite{Siu}, we can then prove that the sequence $\{ h_{a_i} \}_{i\geq 1}$ converges in the $C^0$-sense to $h_0$ on each $U_l$. Next we argue that there is a uniform $C^1$-bound for $\{ h_{a_i} \}_{i\geq 1}$ over $\hat{X}$. We need the following lemma, whose proof will be given later.

\begin{lemma}   
 \label{lm:X}
   Let $V$ be a K$\ddot{\text{a}}$hler manifold endowed with a Ricci-flat K$\ddot{\text{a}}$hler metric $g$, and let $H$ be a HYM metric on a trivial holomorphic $\mathcal{F}$ bundle over $V$ w.r.t. $g$. Fixed a trivialization of $\mathcal{F}$ and view $H$ as a matrix-valued function on $V$. Then
   \begin{equation*}
			-\Delta_g | \partial HH^{-1}|_{H,g}^2 \leq 0.
   \end{equation*}
\end{lemma}

We apply this lemma to the the restriction of $\mathcal{E}$ to $U(1)$ under a trivialization in which $\hat{H}=I$. Also let $H=H_{a_i}$ and $g$ the restriction of $\hat{g}_{a_i}$ to $U(1)$, where it coincides with the CO-metric on resolved conifold. Then we have
\begin{equation*}
		-\hat{\Delta}_{a_i} | \partial H_{a_i} H_{a_i}^{-1}|_{H_{a_i},\hat{g}_{a_i}}^2 \leq 0
\end{equation*}
and hence, by the Maximum Principle,
\begin{equation*}
		\sup_{U(1)} | \partial H_{a_i}H_{a_i}^{-1}|_{H_{a_i},\hat{g}_{a_i}}^2 \leq \sup_{\partial U(1)} | \partial H_{a_i} H_{a_i}^{-1}|_{H_{a_i},\hat{g}_{a_i}}^2.
\end{equation*}

Using the uniform $C^0$-boundedness of $H_a$ and the fact that $\hat{H}=I$, the above inequality implies 
\begin{equation*}
		\sup_{U(1)}| \partial_{\hat{H}} h_{a_i}|_{\hat{H},\hat{g}_{a_i}}\leq K_{10}\sup_{\partial U(1)} | \partial_{\hat{H}} h_{a_i}|_{\hat{H},\hat{g}_{a_i}}.
\end{equation*}

Therefore, it ia enough to bound the maximum of $| \partial_{\hat{H}} h_{a_i}|_{\hat{H},\hat{g}_{a_i}}$ over $X_0[\frac{1}{2}]$. Let $x_i\in X_0[\frac{1}{2}]$ be a sequence of points such that 
\begin{equation*}
		m_i:=\sup_{X_0[\frac{1}{2}]} | \partial_{\hat{H}} h_{a_i}|_{\hat{H},\hat{g}_{a_i}}=| \partial_{\hat{H}} h_{a_i}|_{\hat{H},\hat{g}_{a_i}}(x_i). 
\end{equation*}

Assume $m_i$ is unbounded. If $\{x_i\}$ has a converging subsequence with limit in the interior of $X_0[\frac{1}{2}]$, then one can argue as in \cite{Don1} and \cite{Siu} and get a contradiction. Thus it is enough to get a uniform bound near $\partial X_0[\frac{1}{2}]$. For this we use Lemma \ref{lm:X} and an iteration argument to conclude that $\sup_{\partial X_0[\frac{1}{2}]}| \partial_{\hat{H}} h_{a_i}|_{\hat{H},\hat{g}_{a_i}}$ is bounded by the $L^2$-integral of $| \partial_{\hat{H}} h_{a_i}|_{\hat{H},\hat{g}_{a_i}}$ in a neighborhood of $\partial X_0[\frac{1}{2}]$, say $V_0(\frac{1}{4},\frac{3}{4})$. This last integral is uniformly bounded by Proposition \ref{pr:L21}. Thus if $\{x_i\}$ has a limit on $\partial X_0[\frac{1}{2}]$, $m_i$ is bounded, which contradicts to the assumption. We therefore prove uniform $C^1$-boundedness for $\{ h_{a_i} \}_{i\geq 1}$.

One can then obtain from this uniform $C^1$-bound a uniform $L^p_2$-bound for $\{ h_{a_i} \}_{i\geq 1}$ over each $U_l$ as in \cite{Don1} and \cite{Siu}. Then after taking a subsequence, we may assume that $h_{a_i}$ converges to $h_0$ weakly in the $L^p_2$ sense for all $p$ over each $U_l$. This implies $\Lambda_{\hat{\omega}_0} F_{H_0}=0$ where $H_0=h_0\hat{H}$. By standard elliptic regularity $H_0$ is smooth.  

The proof of Theorem \ref{th:2-1} is now complete.\qed
\end{proof}

\noindent \textbf{Remark} From Lemma \ref{lm:d-1} and Proposition \ref{pr:2-1} it is easy to see that the largest eigenvalue of $h_0$ is bounded from above over $X_{0,sm}$, and lower eigenvalues of $h_0$ is bounded from below over $X_{0,sm}$. In particular, the $C^0$-norm of $h_0$ is bounded over $X_{0,sm}$.\\[0.2cm]

\noindent{\scshape Proof of Lemma \ref{lm:X}}

The HYM equation takes the form
\begin{equation*}
	\sqrt{-1}\Lambda_ g\bar{\partial}(\partial HH^{-1})=0.
\end{equation*}
In local coordinates this is just
\begin{equation*}
	g^{i\bar{j}} \frac{\partial}{\partial \bar{z}_j} \left(\frac{\partial H}{\partial z_i}H^{-1} \right)=0.
\end{equation*}
In the following we denote $\partial_i=\frac{\partial}{\partial z_i}$ and $\partial_{\bar{j}}=\frac{\partial}{\partial \bar{z}_j}$. Taking partial derivatives on both sides of the above equation, we get 
\begin{equation}
 \label{eq:n-1}
   \begin{split}
		-g^{i\bar{p}} \partial_k g_{\bar{p}q}g^{q\bar{j}} \partial_{\bar{j}} \left(\partial_i H H^{-1}\right) 
		+g^{i\bar{j}} \partial_{\bar{j}} \left((\partial_k\partial_i H)H^{-1} 
		- \partial_i HH^{-1}\partial_k HH^{-1} \right)=0.
   \end{split}
\end{equation}
One can compute that 
\begin{equation*}
    \begin{split}
	    &(\partial_k\partial_i H)H^{-1} - \partial_i HH^{-1}\partial_k HH^{-1}\\
	    =& \partial_i\left( \partial_k HH^{-1} \right)
	    +\partial_k HH^{-1}\partial_i HH^{-1} -\partial_i HH^{-1}\partial_k HH^{-1}
	    =(\partial_H)_i\left( \partial_k HH^{-1}\right).
	\end{split}
\end{equation*}
Note also that $g^{i\bar{p}} \frac{\partial g_{\bar{p}q}}{\partial z_k}$ is the Christoffel symbol $\Gamma^i_{kq}$ of $g$. Therefore (\ref{eq:n-1}) becomes
\begin{equation}
 \label{eq:n-2}
		-\Gamma^i_{kq}g^{q\bar{j}} \partial_{\bar{j}}  \left(\partial_i H H^{-1}\right)  
		+g^{i\bar{j}} \partial_{ \bar{j}} (\partial_H)_i\left( \partial_k HH^{-1}\right)=0.
\end{equation}
Now, in local charts,
\begin{equation}
 \label{eq:n-5}
   \begin{split}
		&-\Delta_g | \partial HH^{-1}|_{H,g}^2=-\sqrt{-1}\Lambda_g \partial \bar{\partial} | \partial HH^{-1}|_{H,g}^2 \\
		\leq &-\langle \sqrt{-1}\Lambda_g \nabla_{H,g}^{1,0}\wedge  \nabla_{H,g}^{0,1} (\partial HH^{-1}) , \partial HH^{-1} \rangle_{H,g}
		-\langle \partial HH^{-1} , \sqrt{-1}\Lambda_g \nabla_{H,g}^{0,1} \wedge \nabla_{H,g}^{1,0} (\partial HH^{-1}) \rangle_{H,g}.		  
   \end{split}
\end{equation}

Here $\Lambda_g: \Gamma(V, \text{End}(\mathcal{F})\otimes \Omega^1 \otimes \Omega^2 ) \rightarrow \Gamma(V, \text{End}(\mathcal{F})\otimes \Omega^1)$ is the contraction of the 2-form part with the K$\ddot{\text{a}}$hler form $\omega_g$ of $g$. The operator $\nabla_{H,g}^{1,0}\wedge  \nabla_{H,g}^{0,1}$ is the composition 
\begin{equation*}
  \begin{split}
		&\Gamma(V, \text{End}(\mathcal{F})\otimes \Omega^1) 
		\xrightarrow{\nabla_{H,g}^{0,1}} \Gamma(V,\text{End}(\mathcal{F})\otimes \Omega^1\otimes \Omega^{0,1})\\
		&\xrightarrow{\nabla_{H,g}^{1,0}} \Gamma(V,\text{End}(\mathcal{F})\otimes \Omega^1\otimes \Omega^{0,1}\otimes \Omega^{1,0})
		\xrightarrow{a}\Gamma(V,\text{End}(\mathcal{F})\otimes \Omega^1 \otimes \Omega^{1,1} )
  \end{split}
\end{equation*}
where the last map is the natural anti-symmetrization. The operator $\nabla_{H,g}^{0,1}\wedge  \nabla_{H,g}^{1,0}$ is analogously defined.

Write $A=\partial HH^{-1}=A_k dz_k$ so $\nabla_{H,g}=\nabla_{A,g}$. Explicitly, we have
\begin{equation}
 \label{eq:n-3}
  \begin{split}
		&\sqrt{-1}\Lambda_g \nabla_{H,g}^{1,0}\wedge  \nabla_{H,g}^{0,1} (\partial HH^{-1})\\
		=&-2(g^{i\bar{j}}(\partial_A)_i\partial_{\bar{j}} A_k) dz_k+2g^{q\bar{j}}(\partial_{\bar{j}} A_i)\Gamma^i_{kq}dz_k  
		=2(-g^{i\bar{j}}\partial_{\bar{j}}(\partial_A)_i A_k+g^{q\bar{j}}(\partial_{\bar{j}} A_i)\Gamma^i_{kq})dz_k
  \end{split}
\end{equation}
where we use the fact that
\begin{equation*}
	g^{i\bar{j}}(\partial_A)_i\partial_{\bar{j}} A_k=g^{i\bar{j}}\partial_{\bar{j}}(\partial_A)_i A_k + [g^{i\bar{j}}(F_A)_{i\bar{j}},A_k]=g^{i\bar{j}}\partial_{\bar{j}}(\partial_A)_i A_k
\end{equation*}
because $d+A$ is a HYM connection. Now (\ref{eq:n-2}) and (\ref{eq:n-3}) together implies that 
\begin{equation}
 \label{eq:n-6}
		\sqrt{-1}\Lambda_g \nabla_{H,g}^{1,0}\wedge  \nabla_{H,g}^{0,1} (\partial HH^{-1})=0.
\end{equation}

Next we compute $\sqrt{-1}\Lambda_g\nabla_{H,g}^{0,1}\wedge  \nabla_{H,g}^{1,0} (\partial HH^{-1})$. We have
\begin{equation}
 \label{eq:n-7}
   \begin{split}
		&\sqrt{-1}\Lambda_g \nabla_{H,g}^{0,1}\wedge \nabla_{H,g}^{1,0} (\partial HH^{-1})
		=2\left(g^{i\bar{j}}(\partial_{\bar{j}}(\partial_A)_i A_k) dz_k -(\partial_{\bar{j}} A_i)g^{q\bar{j}}\Gamma^i_{kq}dz_k 
		 -A_ig^{q\bar{j}}\partial_{\bar{j}} \Gamma^i_{kq}dz_k\right) .
	\end{split}
\end{equation}
Note that 
\begin{equation*}
		-\partial_{\bar{j}} \Gamma^i_{kq}=-\partial_{\bar{j}} (g^{i\bar{p}} \partial_k g_{\bar{p}q})
		=-g^{i\bar{p}}\partial_{\bar{j}}\partial_k g_{\bar{p}q}+g^{i\bar{s}}g^{t\bar{p}}\partial_{\bar{j}} g_{\bar{s}t}\partial_k g_{\bar{p}q}
\end{equation*}
is the full curvature tensor $R^i_{qk\bar{j}}$ of $g$. From the Bianchi identity and the fact that $g$ is Ricci flat, we have
\begin{equation}
 \label{eq:n-4}
		-g^{q\bar{j}}\partial_{\bar{j}} \Gamma^i_{kq}=g^{q\bar{j}}R^i_{qk\bar{j}}
		=g^{q\bar{j}}R_{q\bar{p}k\bar{j}}g^{i\bar{p}}= g^{q\bar{j}}R_{q\bar{j}k\bar{p}}g^{i\bar{p}}=R_{k\bar{p}}g^{i\bar{p}}=0.
\end{equation}

From (\ref{eq:n-2}) and (\ref{eq:n-4}) we then have
\begin{equation}
 \label{eq:n-7}
		\sqrt{-1}\Lambda_g \nabla_{H,g}^{0,1}\wedge \nabla_{H,g}^{1,0} (\partial HH^{-1})
		=0.
\end{equation}

The result now follows from (\ref{eq:n-5}), (\ref{eq:n-6}) and (\ref{eq:n-7}).\qed

\subsection{Boundedness results for $H_0$}
We will now establish some boundedness results for $H_0$. The following $C^1$-boundedness for $h_0$ follows easily from the uniform $C^1$-bound of the sequence $\{ h_{a_i} \}_{i\geq 1}$ which converges to $h_0$.
\begin{proposition}
 \label{pr:2-4}
 	There is a constant $C'_1>0$ such that $| \nabla_{\hat{H}}h_0 |_{\hat{H},\hat{g}_0} \leq C'_1$ on $X_{0,sm}$.   
\end{proposition}

Higher order bounds for $H_0$ will be described in the uniform coordinate system $\{(B_z,\phi_ {0,z})| z\in X_{0,sm} \} $ from Section 3.
\begin{proposition}
 \label{pr:2-3}
    There are constants $C'_k>0$ for $k \geq 0$ such that in the above coordinate system, 
    \begin{equation*}
        \Arrowvert  h_0 \Arrowvert_{C^k(B_z,\hat{H},g_e)}<C'_k
    \end{equation*}
    for each $z \in X_{0,sm}$.
\end{proposition}
\begin{proof}
It is enough to focus on $V_{0,sm}(1)$, where $\mathcal{E}$ is the trivial bundle. Moreover, by gauge invariance of the norm, it is enough to work under a holomorphic frame in which $\hat{H}=I$. With this understood, $h_0$ is just $H_0$.

The result for the $k=0$ cases is Proposition \ref{pr:2-1}. For the $k=1$ case, note that by Proposition \ref{pr:2-4} we have locally
\begin{equation*}
    \text{tr}\left( {\hat{g}_0}^{i\bar{j}}\frac{\partial H_0}{\partial w_i }\frac{\partial H^*_0}{\partial \bar{w}_j} \right)<(C'_1)^2.
\end{equation*}  
Here $\ast$ is w.r.t. $I$, and in this case $H_0^*=H_0$. Therefore, because the norm $\mathbf{r}_0(z)^{-\frac{4}{3}}\hat{g}_0\leq C_0g_e$ where $g_e$ is the Euclidean metric in $(w_1,w_2,w_3)$, we have
\begin{equation}
 \label{eq:22}
      \text{tr}\left({g}_e^{i\bar{j}}\frac{\partial H_0}{\partial w_i }\frac{\partial H_0}{\partial \bar{w}_j} \right) 
      \leq C_0 \text{tr}\left(\mathbf{r}_0(z)^\frac{4}{3}(\hat{g}_0)^{i\bar{j}}\frac{\partial H_0}{\partial w_i }\frac{\partial H_0}{\partial \bar{w}_j} \right)<(C'_1)^2C_0\mathbf{r}_0(z)^\frac{4}{3}<K_{11},
\end{equation} 
for some constant $K_{11}$ independent of $z$. This is the desired result for $k=1$. 
 
For the $k \geq 2$ case, note that the metric $H_0$ is HYM, so in each coordinate chart $B_z$ it satisfies the equation 
\begin{equation}
 \label{eq:21}
          \mathbf{r}_0(z)^\frac{4}{3}{\hat{g}_0}^{i\bar{j}}\frac{\partial^2H_0}{\partial w_i \partial \bar{w}_j}
          =\mathbf{r}_0(z)^\frac{4}{3}{\hat{g}_0}^{i\bar{j}}\frac{\partial H_0}{\partial w_i }{H_0}^{-1}
           			\frac{\partial H_0}{\partial \bar{w}_j }.
\end{equation}
By the $k=0,1$ cases and Proposition \ref{pr:2-4} the right hand side of (\ref{eq:21}) is bounded by some constant independent of $z \in V_{0,sm}(1)$. Moreover, there is a constant $\lambda>0$ independent of $z\in V_{0,sm}(1)$ such that
\begin{equation}
 \label{eq:23}
     \mathbf{r}_0(z)^\frac{4}{3}(\hat{g}_0)^{i\bar{j}}\xi_i \bar{\xi}_j \geq \lambda |\xi|^2
\end{equation}
over any $B_z$. Therefore, by p.15 of \cite{J}, the bounds in (\ref{eq:22}) and (\ref{eq:23}) together with the estimates on the higher derivatives of $\mathbf{r}_0(z)^{-\frac{4}{3}}\hat{g}_0$ from (\ref{eq:20}) imply that
\begin{equation*}
    \Arrowvert H_0 \Arrowvert_{C^{1,\frac{1}{2}}(B'_z,g_e)}<K_{12}
\end{equation*}
where $B'_z\subset B_z$ is the ball of radius $\frac{\rho}{2}$ and $K_{12}$ is a constant independent of $z$. It is not hard to improve this to
\begin{equation*}
    \Arrowvert H_0 \Arrowvert_{C^{1,\frac{1}{2}}(B_z,g_e)}<K_{13}
\end{equation*}
by considering the estimates over $B_y$ for $y \in B_z \backslash B'_z$. What is important is that this $C^{1,\frac{1}{2}}(B_z,g_e)$ bound of $H_0$ implies that the right hand side of (\ref{eq:21}) is bounded in the $C^{0,\frac{1}{2}}$ sense, and so by elliptic regularity we get
\begin{equation*}
    \Arrowvert H_0 \Arrowvert_{C^{2,\frac{1}{2}}(B'_z,g_e)}<K_{14},
\end{equation*}
which can be improved to $B_z$ as before. Using bootstrap arguments, we can obtain, for any $k\geq 1$, a constant $C'_k$ independent of $z \in V_{0,sm}(1)$ such that
\begin{equation*}
    \Arrowvert H_0 \Arrowvert_{C^{k}(B_z,g_e)}<C'_k.
\end{equation*}
Here the derivatives is w.r.t. the Euclidean derivatives. However, these are also the derivatives w.r.t. $g_e$ and $\hat{H}$ since $\hat{H}=I$ here. 
\qed
\end{proof}

Let $\alpha$ be a number such that $0<\alpha <\frac{1}{2}$. We will specify the choice of $\alpha$ later. If we restrict ourselves to the region $V_0(\frac{1}{2}R|t|^\alpha,3R|t|^\alpha)$, where the bundle $\mathcal{E}$ is trivial, we have the following result which we will need in the next section.
\begin{proposition}
 \label{pr:1-4}
     For every small $t$ and $w_t \in V_0(\frac{1}{2}R|t|^\alpha,3R|t|^\alpha)$, we have
     \begin{equation*}
          \Arrowvert H_0H_0(w_t)^{-1}-I \Arrowvert_{C^{2,\alpha}(B_z,\hat{H},g_e)}<D|t|^{\frac{2}{3}\alpha}
     \end{equation*}
     where $D>0$ is a constant independent of $t$ and $z\in V_0(\frac{1}{2}R|t|^\alpha,3R|t|^\alpha)$. Here $H_0(w_t)$ is viewed as a constant metric on $\mathcal{E}|_{V_0(\frac{1}{2}R|t|^\alpha,3R|t|^\alpha)} \cong \mathcal{O}^r$.
\end{proposition}
\begin{proof}
	We work in a holomorphic frame over $V_0(\frac{1}{2}R|t|^\alpha,3R|t|^\alpha)$ under which $\hat{H}=I$, so $H_0(w_t)$ is constant a matrix. Because of the bound in the remark before the proof of Lemma \ref{lm:X}, it is enough to show that 
	     \begin{equation*}
          \Arrowvert H_0-H_0(w_t) \Arrowvert_{C^{2}(B_z,g_e)}<D|t|^{\frac{2}{3}\alpha}
     \end{equation*}
     for some constant $D$.
     
     Since $|\nabla_{\hat{H}} H_0|_{\hat{H},\hat{g}_0}<C'_1$ for some constant $C'_1$ and there is a constant $K_{15}>0$ such that $\text{dist}_{\hat{g}_0}(z,w_t)<K_{15}|t|^{\frac{2}{3}\alpha}$ for any small $t$ and $z \in V_0(\frac{1}{2}R|t|^\alpha,3R|t|^\alpha)$, by the mean value theorem we have $|H_0-H_0(w_t)|<K_{16}|t|^{\frac{2}{3}\alpha}$ on $V_0(\frac{1}{2}R|t|^\alpha,3R|t|^\alpha)$.
     For each $z \in V_0(\frac{1}{2}R|t|^\alpha,3R|t|^\alpha)$ in each coordinate chart $B_z$ we have 
\begin{equation}
 \label{eq:24}
           \mathbf{r}_0(z)^\frac{4}{3}(\hat{g}_0)^{i\bar{j}}\frac{\partial^2(H_0-H_0(w_t))}{\partial w_i \partial \bar{w}_j}
          =\mathbf{r}_0(z)^\frac{4}{3}(\hat{g}_0)^{i\bar{j}}\frac{\partial H_0}{\partial w_i }{H_0}^{-1}
           			\frac{\partial H_0}{\partial \bar{w}_j }.
\end{equation}
Notice that equation (\ref{eq:22}) actually implies that the right hand side of equation (\ref{eq:24}) is bounded by $(C'_1)^2C_0\mathbf{r}_0(z)^\frac{4}{3}$, which is less than $K_{17}|t|^{\frac{4}{3}\alpha}$ for some constant $K_{17}>0$. Therefore, in view of (\ref{eq:20}), by elliptic regularity there is a constant $K_{18}$ independent of $t$ and $z \in V_0(\frac{1}{2}R|t|^\alpha,3R|t|^\alpha)$ such that
     \begin{equation*}
        \begin{split}
           &\Arrowvert H_0-H_0(w_t) \Arrowvert_{C^{1,\frac{1}{2}}(B'_z,g_e)} \\
          \leq &K_{18}\left( \left\Arrowvert RHS\,of\,(\ref{eq:24})\right \Arrowvert_{C^0(B_z)}
          + \Arrowvert H_0-H_0(w_t) \Arrowvert_{C^0(B_z)} \right) 
          \leq K_{18}(K_{17}+K_{16})|t|^{\frac{2}{3}\alpha}.
        \end{split}
     \end{equation*}
As in the proof of Proposition \ref{pr:2-3}, this final estimate can be improved to $\Arrowvert H_0-H_0(w_t) \Arrowvert_{C^{1,\frac{1}{2}}(B_z,g_e)} \leq K_{19}|t|^{\frac{2}{3}\alpha}$, and we use elliptic regularity once again to get the desired bound.\qed
\end{proof}

\section{The approximate Hermitian metrics on $\mathcal{E}_t$ over $X_t$}

\subsection{Construction of approximate metrics} In this subsection we construct approximate Hermitian metrics on $\mathcal{E}_t$. We will compare the estimates on the bundles $\mathcal{E}_t$, each over a different manifold $X_t$. For this we first recall the smooth family of diffeomorphisms $x_t:Q_t\backslash \{\mathbf{r}_t=|t|^\frac{1}{2}\}\rightarrow Q_{0,sm}$ from Section 2. Recall also the fixed large number $R \gg 1$ from Section 3 (after (\ref{eq:2})). For $t$ small, restricting to $V_t(\frac{1}{2}R|t|^\frac{1}{2},1)$ we get a smooth family of injective maps
\begin{equation*}
		x_t:V_t(\frac{1}{2}R|t|^\frac{1}{2},1)\rightarrow V_0(\frac{1}{4}R|t|^\frac{1}{2},\frac{3}{2}).
\end{equation*}
We can extend these to a smooth family of injective maps, still denoted by $x_t$:
\begin{equation*}
		x_t:X_t[\frac{1}{2}R|t|^\frac{1}{2}]\rightarrow X_0[\frac{1}{4}R|t|^\frac{1}{2}].
\end{equation*}

Next, choose a smooth family 
\begin{equation*}
		f_t:\mathcal{E}_t|_{X_t[\frac{1}{2}R|t|^\frac{1}{2}]} \rightarrow \mathcal{E}|_{X_0[\frac{1}{4}R|t|^\frac{1}{2}]}
\end{equation*}
of maps between smooth complex vector bundles which commute with $x_t$ and are diffeomorphic onto the images. 

In addition, we require the following condition on $f_t$. Denote by $(\mathcal{X}, \tilde{\mathcal{E}})$ the smoothing of the pair $(X_0,\pi_*\mathcal{E})$ mentioned in the introduction. By our assumption on $\mathcal{E}$ the restriction of $\tilde{\mathcal{E}}$ to $\mathcal{V}:=\bigcup_{t \in \Delta_\epsilon}V_t(1)$ is a trivial holomorphic bundle. Fix a holomorphic trivialization $\tilde{\mathcal{E}}|_{\mathcal{V}}\cong \mathcal{O}_{\mathcal{V}}^r$ inducing the trivialization $\mathcal{E}|_{V_{0,sm}(1)}\cong \mathcal{O}_{V_{0,sm}(1)}^r$ under which $\hat{H}=I$, the $r\times r$ identity matrix. With the induced holomorphic trivialization of $\mathcal{E}_t|_{V_t(1)}$ for all small $t$, we require the family $f_t$ to be such that when restricting to $V_t(\frac{1}{2}R|t|^\frac{1}{2},\frac{3}{4})$, we get a map from the trivial rank $r$ bundle to another trivial rank $r$ bundle which is the product of the map on the base and the identity map on the $\mathbb{C}^r$ fibers.

Over $X_t[\frac{1}{2}R|t|^\frac{1}{2}]$ we let $H''_t=f_t^*H_0$, the pullback of the HYM metric $H_0$ from $X_0[\frac{1}{4}R|t|^\frac{1}{2}]$. Note that our choice of $f_t$ over $V_t(\frac{1}{2}R|t|^\frac{1}{2},\frac{3}{4})$ is one such that $f_t^*$ becomes the pullback of vector-valued functions by $x_t$. In particular, the pullback of a constant frame of $\mathcal{E}|_{x_t(V_t(2R|t|^\alpha,\frac{3}{4}))}$ by $f_t$ is again a constant frame of $\mathcal{E}_t|_{V_t(2R|t|^\alpha,\frac{3}{4})}$. Therefore, under some constant frame of $\mathcal{E}_t|_{V_t(\frac{1}{2}R|t|^\frac{1}{2},\frac{3}{4})}$, the pullback $\hat{H}_t$ of $\hat{H}$ can be seen as an identity matrix. We can extend this constant frame of $\mathcal{E}_t|_{V_t(2R|t|^\alpha,\frac{3}{4})}$ naturally to one over $V_t(\frac{3}{4})$, and we then can extend $\hat{H}_t$ over $V_t(\frac{3}{4})$ by taking the identity matrix under this constant frame. We then further extend $\hat{H}_t$ over the whole $X_t$ to form a smooth family. We still denote these extensions by $\hat{H}_t$, and they will serve as reference metrics on $\mathcal{E}_t$. \\[0.2cm]

From Proposition \ref{pr:2-3} one can deduce
\begin{lemma}
 \label{lm:5}
      There exists a constant $C_k$ such that for any $t \neq 0$ and $z \in X_t[R|t|^\frac{1}{2}]$ we have $\Arrowvert f_t^*h_0 \Arrowvert_{C^k(B_z,\hat{H}_t,\tilde{g}_t)} \leq C_k $.
\end{lemma}

In view of Theorem \ref{th:0} and Proposition \ref{pr:n-1}, we can deduce
\begin{corollary}
 \label{co:4}
      There exists a constant $C''_k$ such that for any $t$ over $V_t(R|t|^\frac{1}{2},\frac{3}{4})$, we have 
      \begin{equation*}
           \sum_{j=0}^k| \mathbf{r}_t^{\frac{2}{3}j} \nabla_{g_{co,t}}^j (f_t^*h_0) |_{\hat{H}_t,g_{co,t}} \leq C''_k.\\[0.2cm]
      \end{equation*}
\end{corollary}

For $\alpha$ such that $0<\alpha <\frac{1}{2}$ and $t$ small, the image of the restriction of $x_t$ to $V_t(R|t|^\alpha,2R|t|^\alpha)$ lies in $V_0(\frac{1}{2}R|t|^\alpha,3R|t|^\alpha)$. For $w_t$ as in Proposition \ref{pr:1-4}, define $H'_t:=f_t^*(H_0(w_t))$ to be the constant metric on $\mathcal{E}_t|_{V_t(2R|t|^\alpha)}$ (w.r.t. a constant frame).

Then by Proposition \ref{pr:1-4} we immediately get
\begin{lemma}
 \label{lm:8}
     There is a constant $D>0$ such that for any $t$ and $z \in V_t(R|t|^\alpha,2R|t|^\alpha)$, we have
     \begin{equation*}
          \Arrowvert H''_t(H'_t)^{-1} -I \Arrowvert_{C^2(B_z,\hat{H}_t,g_e)}<D|t|^{\frac{2}{3}\alpha}.
     \end{equation*}
\end{lemma}

Now let $\tilde{\tau}_{t}(s)$ be a smooth increasing cutoff function on $\mathbb{R}^1$ such that 
\begin{equation*}
    \tilde{\tau}_{t}(s)=\{ 
        \begin{array}{ll}
 				1, & s \geq 2R|t|^{\alpha-\frac{1}{2}} \\
				0, & s \leq R|t|^{\alpha-\frac{1}{2}},
		  \end{array} 
\end{equation*}
and such that its $l$-th derivative $\tilde{\tau}_{t}^{(l)}$ satisfies $|\tilde{\tau}_{t}^{(l)}| \leq \tilde{K}_l|t|^{(\frac{1}{2}-\alpha)l}$ for $l \geq 1$ for a constant $\tilde{K}_l>0$ independent of $t$.
Define $\tau_t=\tilde{\tau}_{t}(|t|^{-\frac{1}{2}}\mathbf{r}_t)$, which is a cutoff function on $X_t$, and define the approximate Hermitian metric to be
\begin{equation*}
     H_t=(1-\tau_t)H'_t+\tau_t H''_t=(I+\tau_t(H''_t(H'_t)^{-1}-I))H'_t.
\end{equation*}

\noindent \textbf{Remark} The metric $H_t$ is just an interpolation between $f_t^*H_0$ and $H'_t=f_t^*(H_0(w_t))$. Because the determinant $f_t^*h_0$ is bounded uniformly both from above and below, the common $C^0$-bound of $f_t^*h_0$ w.r.t. $\hat{H}_t$ in Lemma \ref{lm:5} implies that the norms $|\cdot|_{H_t}$ and $|\cdot|_{\hat{H}_t}$ are in fact equivalent (uniformly in $t$). \\[0.2cm]

The following estimates for $H_t$ are analogous to those for $H_0$ in Proposition \ref{pr:2-3}. They follow from that proposition with the help of Corollary \ref{co:3}.
\begin{proposition}
 \label{pr:3-3}
    There are constants $C_k>0$ for $k \geq 0$ such that  
    \begin{equation*}
        \Arrowvert H_t (\hat{H}_t)^{-1}\Arrowvert_{C^k(B_z,\hat{H}_t,g_e)}<C_k
    \end{equation*}
    for each $z \in X_t$. \\[0.2cm]
\end{proposition}

\subsection{Bounds for the mean curvatures} The following proposition gives the bounds for the mean curvatures $\sqrt{-1}\Lambda_{\tilde{\omega}_t}F_{H_t}$ of the approximate metrics $H_t$.
\begin{proposition}
 \label{pr:3-4}
        There are constants $\Lambda_k>0$ and $\tilde{Z}_k>0$ such that for $t$ small enough, we have the following: 
     \begin{enumerate} 
      \item For any $z \in X_t$ and $k\geq 1$,
      			\begin{equation}
				 \label{eq:F-3}
				 		\Arrowvert \mathbf{r}_t^\frac{4}{3} \Lambda_{\tilde{\omega}_t}F_{H_t} \Arrowvert_{C^k(B_z,\hat{H}_t,g_e)} \leq \Lambda_k,
				\end{equation}
      \item  \begin{equation}
        	 		\label{eq:20-5}
						|\mathbf{r}_t^\frac{4}{3} \Lambda_{\tilde{\omega}_t}F_{H_t} |_{H_t}\leq \tilde{Z}_0 \max \{|t|^{\frac{2}{3}\alpha},\,|t|^{1-2\alpha}\},
			   \end{equation}
and
        \item  \begin{equation}
         		  \label{eq:20-6}
				         \Arrowvert \Lambda_{\tilde{\omega}_t}F_{H_t} \Arrowvert_{L^k_{0,-4}(X_t,H_t,\tilde{g}_t)}\leq \tilde{Z}_k \max \{|t|^{2\alpha},\,|t|^{1-\frac{2}{3}\alpha}\}.
				\end{equation}
	\end{enumerate}
\end{proposition} 
\begin{proof}
The first estimates (\ref{eq:F-3}) follow from Theorem \ref{th:0} and Proposition \ref{pr:3-3}.

For (\ref{eq:20-5}), first of all we have
\begin{equation}
 \label{eq:3-18}
      \Lambda_{\tilde{\omega}_t} F_{H_t}=0\,\,\,\text{on}\,\,V_t(R|t|^\alpha)
\end{equation}
because $H_t=H'_t$ there and $H'_t$ is a flat metric. 

Next consider the annulus $V_t(R|t|^\alpha,2R|t|^\alpha)$. Let 
\begin{equation*}
      h'_t=I+\tau_t(H''_t(H'_t)^{-1} -I),
\end{equation*} 
then we have 
\begin{equation*}
      \Lambda_{\tilde{\omega}_t} F_{H_t}
      =\Lambda_{\tilde{\omega}_t} F_{H'_t}+\Lambda_{\tilde{\omega}_t}\bar{\partial}(\partial_{H'_t}h'_t(h'_t)^{-1})
      =\Lambda_{\tilde{\omega}_t}\bar{\partial}(\partial_{H'_t}h'_t(h'_t)^{-1}).
\end{equation*}  

Now, on each local coordinate chart $B_z\cap V_t(R|t|^\alpha,2R|t|^\alpha)$, compute in a frame under which $H_t'$ is constant, we have
\begin{equation}
 \label{eq:3-3}
   \begin{split}
       \mathbf{r}_t^\frac{4}{3}\Lambda_{\tilde{\omega}_t}\bar{\partial}(\partial_{H'_t}h'_t(h'_t)^{-1}) 
     =&\mathbf{r}_t^\frac{4}{3}\Lambda_{\tilde{\omega}_t}\bar{\partial}\left(\left(\partial h'_t-\partial H'_t(H'_t)^{-1} h'_t+h'_t \partial H'_t(H'_t)^{-1}\right)(h'_t)^{-1}\right) \\
     =&\mathbf{r}_t^\frac{4}{3}\Lambda_{\tilde{\omega}_t}\bar{\partial}(\partial h'_t(h'_t)^{-1}) \\
     =&\mathbf{r}_t^\frac{4}{3}\Lambda_{\tilde{\omega}_t}\partial h'_t(h'_t)^{-1}\bar{\partial}h'_t (h'_t)^{-1} 
       +\mathbf{r}_t^\frac{4}{3}\Lambda_{\tilde{\omega}_t}  \bar{\partial}\partial h'_t(h'_t)^{-1}. 
   \end{split}
\end{equation}
To bound the derivatives of $h'_t$ we need to bound the derivatives of $\tau_t$. The first order derivative of $\tau_t$ can be bounded as 
\begin{equation}
 \label{eq:3-1}
     |\nabla_e \tau_t|_{g_e} \leq |\tilde{\tau}'(|t|^{-\frac{1}{2}}\mathbf{r}_t)|t|^{-\frac{1}{2}} \nabla_e \mathbf{r}_t|_{g_e} 
                                  \leq \tilde{K}_1|t|^{\frac{1}{2}-\alpha} |t|^{-\frac{1}{2}}\cdot R_1\mathbf{r}_t \leq \tilde{K}_1R_1 |t|^{-\alpha}\mathbf{r}_t \leq 2\tilde{K}_1R_1R
\end{equation} 
where (\ref{eq:ra-1}) is used. The last inequality follows from the fact that the support of $\nabla\tau_t$ is contained in $V_t(R|t|^\alpha,2R|t|^\alpha)$.

Similarly, the second order derivative of $\tau_t$ can be bounded as 
\begin{equation}
 \label{eq:3-2}
   \begin{split}
     |\nabla_e^2 \tau_t|_{g_e}
\leq & |\tilde{\tau}''(|t|^{-\frac{1}{2}}\mathbf{r}_t)|t|^{-1}||\nabla_e \mathbf{r}_t|_{g_e}^2+ |\tilde{\tau}'(|t|^{-\frac{1}{2}}\mathbf{r}_t)|t|^{-\frac{1}{2}} \nabla_e^2 \mathbf{r}_t|_{g_e} \\
\leq & \tilde{K}_2 |t|^{1-2\alpha} |t|^{-1}\mathbf{r}_t^2+\tilde{K}_1|t|^{\frac{1}{2}-\alpha} |t|^{-\frac{1}{2}}\cdot R_2\mathbf{r}_t \\
\leq & \tilde{K}_2 |t|^{-2\alpha} \mathbf{r}_t^2+\tilde{K}_1R_2|t|^{-\alpha} \mathbf{r}_t \leq  4\tilde{K}_2R^2+2\tilde{K}_1R_2R
    \end{split} 
\end{equation}
where (\ref{eq:ra-1}) is used again and the last inequality follows as in (\ref{eq:3-1}).

From (\ref{eq:3-1}), (\ref{eq:3-2}) and Lemma \ref{lm:8} we can obtain the estimates
\begin{equation*}
   \begin{split}
     |\partial h'_t|_{\hat{H}_t,g_e}=&|\partial(\tau_t(H''_t(H'_t)^{-1}-I))|_{\hat{H}_t,g_e} \\
     									\leq   & |\nabla_e\tau_t|_{g_e} |H''_t(H'_t)^{-1}-I |_{\hat{H}_t}+\tau_t|\nabla_e (H''_t(H'_t)^{-1})|_{\hat{H}_t,g_e}  \\
									    \leq & 2\tilde{K}_1R_1R\cdot D|t|^{\frac{2}{3}\alpha}+D|t|^{\frac{2}{3}\alpha}
									    \leq  (2\tilde{K}_1R_1R+1) D\cdot|t|^{\frac{2}{3}\alpha}
   \end{split}
\end{equation*} 
and
\begin{equation*}
   \begin{split}
     |\bar{\partial}\partial h'_t|_{\hat{H}_t,g_e} 
     \leq & |\nabla_e^2\tau_t|_{g_e} |H''_t(H'_t)^{-1} -I |_{\hat{H}_t}+\tau_t|\nabla_e^2 (H''_t(H'_t)^{-1})|_{\hat{H}_t,g_e}\\
           &+2|\nabla_e\tau_t|_{g_e}|\nabla_e (H''_t(H'_t)^{-1})|_{\hat{H}_t,g_e} \\
      \leq &(4\tilde{K}_2R^2+2\tilde{K}_1R_2R)\cdot D|t|^{\frac{2}{3}\alpha}+D|t|^{\frac{2}{3}\alpha}+2\cdot 2\tilde{K}_1R_1R\cdot D|t|^{\frac{2}{3}\alpha}\\
      \leq & (4\tilde{K}_2R^2+2\tilde{K}_1R_2R+4\tilde{K}_1R_1R+1)\cdot D|t|^{\frac{2}{3}\alpha}.
   \end{split}
\end{equation*}
In local charts, the term $\mathbf{r}_t^\frac{4}{3}\Lambda_{\tilde{\omega}_t}$ contributes to $\mathbf{r}_t^\frac{4}{3}\tilde{g}_t^{-1}$, which is bounded by Theorem \ref{th:0}. Therefore, from expression (\ref{eq:3-3}) we can now conclude that 
\begin{equation*}
      |\mathbf{r}_t^\frac{4}{3}\Lambda_{\tilde{\omega}_t}F_{H_t}|_{\hat{H}_t}
      = |\mathbf{r}_t^\frac{4}{3}\Lambda_{\tilde{\omega}_t}\bar{\partial}(\partial_{H'_t}h'_t(h'_t)^{-1})|_{\hat{H}_t} \leq Z_1|t|^{\frac{2}{3}\alpha}
\end{equation*}
on $V_t(R|t|^\alpha,2R|t|^\alpha)$ for some constant $Z_1>0$ independent of $t$.

From the remark before Proposition \ref{pr:3-3} we get
\begin{equation}
 \label{eq:8-2}
      |\mathbf{r}_t^\frac{4}{3}\Lambda_{\tilde{\omega}_t}F_{H_t}|_{H_t} \leq Z_2|t|^{\frac{2}{3}\alpha}
\end{equation}
on $V_t(R|t|^\alpha,2R|t|^\alpha)$ for some constant $Z_2>0$ independent of $t$..

We now estimate the $L^k_{0,-4}$-norm of $\sqrt{-1}\Lambda_{\tilde{\omega}_t}F_{H_t}$ on $V_t(R|t|^\alpha,2R|t|^\alpha)$ with respect to $\tilde{g}_t$ and $H_t$.
\begin{equation*}
   \begin{split}
       &\int_{V_t(R|t|^\alpha,2R|t|^\alpha)} |\mathbf{r}_t^{\frac{8}{3}} \Lambda_{\tilde{\omega}_t}F_{H_t}|_{H_t}^k \mathbf{r}_t^{-4}dV_t
         = \int_{V_t(R|t|^\alpha,2R|t|^\alpha)} \mathbf{r}_t^{\frac{4}{3} k} |\mathbf{r}_t^\frac{4}{3} \Lambda_{\tilde{\omega}_t} F_{H_t}|_{H_t}^k \mathbf{r}_t^{-4}dV_t \\
    \leq & (2R)^{\frac{4}{3}k} |t|^{\frac{4}{3}\alpha k}\cdot Z_2^k |t|^{\frac{2}{3}\alpha k} \int_{V_t(R|t|^\alpha,2R|t|^\alpha)} \mathbf{r}_t^{-4}dV_t 
    \leq  (2R)^{\frac{4}{3}k}Z_2^k |t|^{2\alpha k} Z_3.
   \end{split}
\end{equation*}
where $Z_3>1$ is an upper bound for $\int_{V_t(R|t|^\alpha,2R|t|^\alpha)} \mathbf{r}_t^{-4}dV_t$ for any $t \neq 0$ small. Thus 
\begin{equation}
 \label{eq:3-17}
     \Arrowvert \Lambda_{\tilde{\omega}_t}F_{H_t} \Arrowvert_{L^k_{0,-4}(V_t(R|t|^\alpha,2R|t|^\alpha),\tilde{g}_t,H_t)} \leq (2R)^{\frac{4}{3}}Z_2  Z_3 |t|^{2\alpha}.
\end{equation}

We proceed to consider the region $V_t(2R|t|^\alpha,\frac{3}{4})$. We will first give a pointwise estimate on the mean curvature of the Hermitian metric $H_t=f_t^*H_0$. We will use $\partial_t$ and $\bar{\partial}_t$ to emphasize that they are the $\partial$- and $\bar{\partial}$-operators on $X_t$, respectively. The calculation will be done under the specific choices of frames as mentioned before Lemma \ref{lm:5}. With these choices, we have $\hat{H}_t=I$ and $f_t^*H_0$ can be regarded as the pullback by $x_t$ of a matrix-valued function representing $H_0$. Since constant frames are holomorphic, the curvature of $f_t^*H_0$ can be computed using this pullback matrix-valued function which we still denote by $f_t^*H_0$.

\begin{lemma}
 \label{lm:7}
    There is a constant $Z_4>0$ independent of $t$ such that 
    \begin{equation*}
         |\mathbf{r}_t^\frac{4}{3} \Lambda_{\tilde{\omega}_t}\left(\bar{\partial}_t(\partial_t (f_t^*H_0)(f_t^*H_0)^{-1})\right)|_{\hat{H}_t} \leq Z_4\cdot |t|\mathbf{r}_t^{-2}
    \end{equation*}
    on $V_t(2R|t|^\alpha,\frac{3}{4})$.
\end{lemma}
\begin{proof}
 We expand and get 
\begin{equation}
 \label{eq:3-19}
    \bar{\partial}_t(\partial_t (f_t^*H_0)(f_t^*H_0)^{-1})
    = (\bar{\partial}_t\partial_t (f_t^*H_0))(f_t^*H_0)^{-1}+\partial_t (f_t^*H_0)\wedge(f_t^*H_0)^{-1}\bar{\partial}_t (f_t^*H_0)  (f_t^*H_0)^{-1}.
\end{equation}
We compute 
\begin{equation}
 \label{eq:3-7}
   \begin{split}
       \bar{\partial}_t\partial_t (f_t^*H_0) 
    =&-\frac{\sqrt{-1}}{2} dJ_t d(f_t^*H_0)=-\frac{\sqrt{-1}}{2} dJ_tf_t^*(dH_0)\\ 
    = &-\frac{\sqrt{-1}}{2} d(x_t^*J_0d(f_t^*H_0))  -\frac{\sqrt{-1}}{2} d\left [(J_t-x_t^*J_0)d(f_t^*H_0) \right ]  \\
    =& -\frac{\sqrt{-1}}{2} f_t^*(dJ_0d(f_t^*H_0))) -\frac{\sqrt{-1}}{2}  d\left [(J_t-x_t^*J_0)d(f_t^*H_0) \right ]  \\
    =& f_t^*(\bar{\partial}_0\partial_0 H_0) -\frac{\sqrt{-1}}{2} d\left [(J_t-x_t^*J_0)f_t^*(dH_0) \right ].
   \end{split}
\end{equation}
Moreover, 
\begin{equation}
 \label{eq:3-8}
   \begin{split}
       \partial_t (f_t^*H_0) 
    =& \frac{d-\sqrt{-1}J_td}{2}(f_t^*H_0)=\frac{1}{2}(1-\sqrt{-1}J_t)f_t^*(dH_0)\\
    =&\frac{1}{2}(1-\sqrt{-1}x_t^*J_0)f_t^*(dH_0)-\frac{\sqrt{-1}}{2}(J_t-x_t^*J_0)f_t^*(dH_0) \\
    =& f_t^*(\partial_0 H_0)-\frac{\sqrt{-1}}{2}(J_t-x_t^*J_0)f_t^*(dH_0),
   \end{split}
\end{equation}
and similarly
\begin{equation}
 \label{eq:3-9}
       \bar{\partial}_t (f_t^*H_0) 
    = f_t^*(\bar{\partial}_0 H_0)+\frac{\sqrt{-1}}{2}(J_t-x_t^*J_0)f_t^*(dH_0).
\end{equation}

Plug in (\ref{eq:3-7}), (\ref{eq:3-8}) and (\ref{eq:3-9}) to (\ref{eq:3-19}), we get
\begin{equation*}
   \begin{split}
       & \bar{\partial}_t(\partial_t (f_t^*H_0)(f_t^*H_0)^{-1}) \\
     =& f_t^*(\bar{\partial}_0\partial_0 H_0)(f_t^*H_0)^{-1}
           -\frac{\sqrt{-1}}{2}d\left [(J_t-x_t^*J_0)f_t^*(dH_0) \right ]\cdot (f_t^*H_0)^{-1}\\
       & +(f_t^*\partial_0 H_0)(f_t^*H_0)^{-1}\wedge (f_t^*\bar{\partial}_0 H_0)(f_t^*H_0)^{-1} \\
       & +\frac{\sqrt{-1}}{2}(f_t^*\partial_0 H_0)(f_t^*H_0)^{-1} \wedge [(J_t-x_t^*J_0)f_t^*(dH_0)](f_t^*H_0)^{-1} \\
       & -\frac{\sqrt{-1}}{2}[(J_t-x_t^*J_0)f_t^*(dH_0)](f_t^*H_0)^{-1} \wedge (f_t^*\bar{\partial}_0 H_0)(f_t^*H_0)^{-1} \\
       & +\frac{1}{4}[(J_t-x_t^*J_0)f_t^*(dH_0)](f_t^*H_0)^{-1} \wedge [(J_t-x_t^*J_0)f_t^*(dH_0)](f_t^*H_0)^{-1} \\
     =& f_t^*(\bar{\partial}_0(\partial_0 H_0(H_0)^{-1})))
           -\frac{\sqrt{-1}}{2}d\left [(J_t-x_t^*J_0)f_t^*(dH_0) \right ]\cdot (f_t^*H_0)^{-1} \\
       & +\frac{\sqrt{-1}}{2}(f_t^*\partial_0 H_0)(f_t^*H_0)^{-1} \wedge [(J_t-x_t^*J_0)f_t^*(dH_0)](f_t^*H_0)^{-1} \\
       & -\frac{\sqrt{-1}}{2}[(J_t-x_t^*J_0)f_t^*(dH_0)](f_t^*H_0)^{-1} \wedge (f_t^*\bar{\partial}_0 H_0)(f_t^*H_0)^{-1} \\
       & +\frac{1}{4}[(J_t-x_t^*J_0)f_t^*(dH_0)](f_t^*H_0)^{-1} \wedge [(J_t-x_t^*J_0)f_t^*(dH_0)](f_t^*H_0)^{-1} \\
   \end{split}
\end{equation*}

Therefore we have
\begin{equation}
 \label{eq:3-10}
   \begin{split}
       &|\mathbf{r}_t^\frac{4}{3} \Lambda_{\tilde{\omega}_t} \bar{\partial}_t(\partial_t (f_t^*H_0)(f_t^*H_0)^{-1})|_{\hat{H}_t} \\
 \leq &|\mathbf{r}_t^\frac{4}{3} \Lambda_{\tilde{\omega}_t} f_t^*(\bar{\partial}_0(\partial_0 H_0(H_0)^{-1}))|_{\hat{H}_t}
          +\frac{1}{2}|\mathbf{r}_t^\frac{4}{3} \Lambda_{\tilde{\omega}_t}d\left [(J_t-x_t^*J_0)f_t^*(dH_0) \right ]\cdot (f_t^*H_0)^{-1}|_{\hat{H}_t} \\
       & +\frac{1}{2} | \mathbf{r}_t^\frac{4}{3} \Lambda_{\tilde{\omega}_t}\left [(f_t^*\partial_0 H_0)(f_t^*H_0)^{-1} \wedge [(J_t-x_t^*J_0)f_t^*(dH_0)](f_t^*H_0)^{-1}\right] |_{\hat{H}_t}\\
       & +\frac{1}{2} | \mathbf{r}_t^\frac{4}{3} \Lambda_{\tilde{\omega}_t}\left [[(J_t-x_t^*J_0)f_t^*(dH_0)](f_t^*H_0)^{-1} \wedge (f_t^*\bar{\partial}_0 H_0)(f_t^*H_0)^{-1} \right ] |_{\hat{H}_t} \\
       & +\frac{1}{4} | \mathbf{r}_t^\frac{4}{3} \Lambda_{\tilde{\omega}_t}\left [[(J_t-x_t^*J_0)f_t^*(dH_0)](f_t^*H_0)^{-1} \wedge [(J_t-x_t^*J_0)f_t^*(dH_0)](f_t^*H_0)^{-1} \right] |_{\hat{H}_t}
   \end{split}
\end{equation}

Note that because $\hat{H}_t=I$ under the chosen frame, we have $f_t^*H_0=f_t^*h_0$. Using the bounds in Proposition \ref{pr:2-3}, we can estimate the first term on the RHS of (\ref{eq:3-10}) in each coordinate chart $B_z$ as
\begin{equation}
 \label{eq:3-11}
   \begin{split}
             & |\mathbf{r}_t^\frac{4}{3} \Lambda_{\tilde{\omega}_t} f_t^*(\bar{\partial}_0(\partial_0 H_0(H_0)^{-1}))|_{\hat{H}_t} \\
       \leq & |\mathbf{r}_t^\frac{4}{3} \Lambda_{x_t^*\hat{\omega}_0} f_t^*(\bar{\partial}_0(\partial_0 H_0(H_0)^{-1}))|_{\hat{H}_t}
              +|\mathbf{r}_t^\frac{4}{3} (\Lambda_{\tilde{\omega}_t} - \Lambda_{x_t^*\hat{\omega}_0})f_t^*(\bar{\partial}_0(\partial_0 H_0(H_0)^{-1}))|_{\hat{H}_t}  \\
       \leq & Z_5|\mathbf{r}_t^\frac{4}{3} f_t^*(\bar{\partial}_0(\partial_0 H_0(H_0)^{-1}))|_{\hat{H}_t,g_{co,t}} \cdot |\tilde{\omega}_t^{-1} - x_t^*\omega_{co,0}^{-1}|_{g_{co,t}} \\
       \leq & Z_6|f_t^*(\bar{\partial}_0(\partial_0 H_0(H_0)^{-1})) |_{C^0(B_z,\hat{H}_t,\bar{g}_{co,t})} \cdot |\tilde{\omega}_t^{-1}-\omega_{co,t}^{-1}|_{g_{co,t}} \\
       \leq & Z_7|f_t^*(\bar{\partial}_0(\partial_0 H_0(H_0)^{-1})) |_{C^0(B_z,\hat{H}_t,g_e)} \cdot C''|t|\mathbf{r}_t^{-\frac{2}{3}} \\
       \leq & Z_8|\bar{\partial}_0(\partial_0 H_0(H_0)^{-1})) |_{C^0(B_{x_t(z)},\hat{H},g_e)} \cdot |t|\mathbf{r}_t^{-\frac{2}{3}}\\
       \leq & Z_9\left(\Arrowvert H_0\Arrowvert_{C^2(B_{x_t(z)},\hat{H},g_e)}+\Arrowvert H_0\Arrowvert_{C^1(B_{x_t(z)},\hat{H},g_e)}^2\right)\cdot |t|\mathbf{r}_t^{-\frac{2}{3}} \\
       \leq &Z_9\cdot (C'_2+(C'_1)^2) \cdot |t|\mathbf{r}_t^{-\frac{2}{3}}
       \leq Z_{10}|t|\mathbf{r}_t^{-2}
   \end{split}
\end{equation}
where Proposition \ref{pr:6-1} and the equation in Lemma \ref{lm:4} are applied. We have also used the fact that
\begin{equation*}
    \Lambda_{x_t^*\hat{\omega}_{0}}f_t^*(\bar{\partial}_0(\partial_0 H_0(H_0)^{-1}))=f_t^*(\Lambda_{\hat{\omega}_{0}}(\bar{\partial}_0(\partial_0 H_0(H_0)^{-1})))=0
\end{equation*}
since $H_0$ is HYM with respect to the balanced metric $\hat{\omega}_{0}$ on $X_{0,sm}$.

The second term on the RHS of (\ref{eq:3-10}) is bounded as
\begin{equation}
 \label{eq:3-12}
    \begin{split}
             & \frac{1}{2}|\mathbf{r}_t^\frac{4}{3} \Lambda_{\tilde{\omega}_t}d\left [(J_t-x_t^*J_0)f_t^*(dH_0) \right ]\cdot (f_t^*H_0)^{-1} |_{\hat{H}_t} \\
       \leq & \frac{1}{2}|\mathbf{r}_t^\frac{4}{3} \Lambda_{\omega_{co,t}}d\left [(J_t-x_t^*J_0)f_t^*(dH_0) \right ]\cdot (f_t^*H_0)^{-1} |_{\hat{H}_t} \\
              &+\frac{1}{2}|\mathbf{r}_t^\frac{4}{3} (\Lambda_{\omega_{co,t}}-\Lambda_{\tilde{\omega}_t}) d\left [(J_t-x_t^*J_0)f_t^*(dH_0) \right ]\cdot (f_t^*H_0)^{-1}|_{\hat{H}_t} \\
       \leq & Z_{11}|\mathbf{r}_t^\frac{4}{3} \Lambda_{\omega_{co,t}}d\left [(J_t-x_t^*J_0)f_t^*(dH_0) \right ]|_{\hat{H}_t} \\
               &+Z_{11}|\tilde{\omega}_t^{-1}-\omega_{co,t}^{-1}|_{g_{co,t}} | \mathbf{r}_t^\frac{4}{3}d\left [(J_t-x_t^*J_0)f_t^*(dH_0) \right ] |_{\hat{H}_t,g_{co,t}}\\
       \leq & Z_{12}(1+|\tilde{\omega}_t^{-1}-\omega_{co,t}^{-1}|_{g_{co,t}})\cdot \\
               & \left( \mathbf{r}_t^\frac{2}{3}|\nabla_{g_{co,t}}(J_t-x_t^*J_0)|_{g_{co,t}} |\mathbf{r}_t^\frac{2}{3} d(f_t^*H_0)|_{\hat{H}_t,g_{co,t}} 
               + |J_t-x_t^*J_0|_{g_{co,t}} \sum_{j=0}^2|\mathbf{r}_t^{\frac{2}{3}j}\nabla_{g_{co,t}}^j (f_t^*H_0)|_{\hat{H}_t,g_{co,t}} \right). \\
    \end{split}
\end{equation}

To proceed, let $\nabla_{\Upsilon_t^*g_{co,t}}-\nabla_{g_{co,0}}$ be the difference between the two connections. It is in fact the difference between the Christoffel symbols of $\Upsilon_t^*g_{co,t}$ and $g_{co,0}$. From the explicit formulas of Christoffel symbols in terms of the metrics and Proposition \ref{pr:n-1}, for some universal constans $D_1>0$ and $D_2>0$ we have
\begin{equation}
 \label{dif}
		|\nabla_{\Upsilon_t^*g_{co,t}}-\nabla_{g_{co,0}}|_{\Upsilon_t^*g_{co,t}}\leq D_1(|g_{co,0}^{-1}dg_{co,0}|_{\Upsilon_t^*g_{co,t}}+|\Upsilon_t^*g_{co,t}^{-1}d\Upsilon_t^*g_{co,t}|_{\Upsilon_t^*g_{co,t}})\leq D_2\mathbf{r}_t^{-\frac{2}{3}}.
\end{equation}

Now, by Corollary \ref{asym} we have
\begin{equation}
 \label{est-1}
    \begin{split}
		|J_t-x_t^*J_0|_{g_{co,t}}
		\leq  |\Upsilon_t^*J_t-J_0|_{\Upsilon_t^*g_{co,t}} 
		\leq D_0|t|\mathbf{r}_t^{-2}
    \end{split}
\end{equation}
and by (\ref{dif}) we have
\begin{equation}
 \label{est-2}
    \begin{split}
		|\nabla_{g_{co,t}}(J_t-x_t^*J_0)|_{g_{co,t}} 
		\leq & |\nabla_{\Upsilon_t^*g_{co,t}}(\Upsilon_t^*J_t-J_0)|_{\Upsilon_t^*g_{co,t}} \\
		\leq & |\nabla_{g_{co,0}}(\Upsilon_t^*J_t-J_0)|_{\Upsilon_t^*g_{co,t}}
					+|(\nabla_{\Upsilon_t^*g_{co,t}}-\nabla_{g_{co,0}})(\Upsilon_t^*J_t-J_0)|_{\Upsilon_t^*g_{co,t}}\\
		\leq & D_0|t|\mathbf{r}_t^{-\frac{8}{3}}+|\nabla_{\Upsilon_t^*g_{co,t}}-\nabla_{g_{co,0}}|_{\Upsilon_t^*g_{co,t}}\cdot |\Upsilon_t^*J_t-J_0|_{\Upsilon_t^*g_{co,t}} \\
		\leq & D_0|t|\mathbf{r}_t^{-\frac{8}{3}}+D_2D_0|t|\mathbf{r}_t^{-\frac{8}{3}}.
    \end{split}
\end{equation}

We also have $|\mathbf{r}_t^\frac{2}{3} d(f_t^*H_0)|_{\hat{H}_t,g_{co,t}} \leq C''_1$ and $\sum_{j=0}^2|\mathbf{r}_t^{\frac{2}{3}j}\nabla_{g_{co,t}}^j (f_t^*H_0)|_{\hat{H}_t,g_{co,t}}  \leq C''_2$ from Corollary \ref{co:4}, and $|\tilde{\omega}_t^{-1}-\omega_{co,t}^{-1}|_{g_{co,t}} \leq C''|t|^\frac{2}{3}$ from Proposition \ref{pr:6-1}. Plug these, (\ref{est-1}) and (\ref{est-2}) into (\ref{eq:3-12}) we get
\begin{equation}
 \label{eq:3-12-1}
    \begin{split}
             & \frac{1}{2}|(f_t^*H_0)^{-1}\mathbf{r}_t^\frac{4}{3} \Lambda_{\tilde{\omega}_t}d\left [(J_t-x_t^*J_0)f_t^*(dH_0) \right ]|_{\hat{H}_t} \\
       \leq & Z_{12}(1+C''|t|^{\frac{2}{3}})\left( \mathbf{r}_t^\frac{2}{3}(D_0+D_2D_0)|t|\mathbf{r}_t^{-\frac{8}{3}}\cdot C''_1 +D_0|t|\mathbf{r}_t^{-2} \cdot C''_2\right)
       \leq  Z_{13}|t|\mathbf{r}_t^{-2}.
    \end{split}
\end{equation}

The third term on the RHS of (\ref{eq:3-10}) is bounded as
\begin{equation}
 \label{eq:3-13}
    \begin{split}
             & \frac{1}{2} | \mathbf{r}_t^\frac{4}{3} \Lambda_{\tilde{\omega}_t}\left [(f_t^*\bar{\partial}_0 H_0) (f_t^*H_0)^{-1}\wedge [(J_t-x_t^*J_0)f_t^*(dH_0)](f_t^*H_0)^{-1} \right] |_{\hat{H}_t}  \\
       \leq  &\frac{1}{2} | \mathbf{r}_t^\frac{4}{3} \Lambda_{\omega_{co,t}}\left [(f_t^*\bar{\partial}_0 H_0) (f_t^*H_0)^{-1}\wedge [(J_t-x_t^*J_0)f_t^*(dH_0)](f_t^*H_0)^{-1} \right] |_{\hat{H}_t}  \\
             &+\frac{1}{2} | \mathbf{r}_t^\frac{4}{3} (\Lambda_{\omega_{co,t}}-\Lambda_{\tilde{\omega}_t}) \left [(f_t^*\bar{\partial}_0 H_0) (f_t^*H_0)^{-1}\wedge [(J_t-x_t^*J_0)f_t^*(dH_0)](f_t^*H_0)^{-1} \right] |_{\hat{H}_t}  \\
       \leq &Z_{14}(1+|\tilde{\omega}_t^{-1}-\omega_{co,t}^{-1}|_{g_{co,t}})\cdot  |J_t-x_t^*J_0|_{g_{co,t}} \cdot |\mathbf{r}_t^\frac{2}{3}d (f_t^*H_0)|_{\hat{H}_t,g_{co,t}} ^2  \\
       \leq &Z_{14} (1+C''|t|^\frac{2}{3}) \cdot D_0|t|\mathbf{r}_t^{-2} \cdot (C''_1)^2 
       \leq  Z_{15} |t|\mathbf{r}_t^{-2}.
    \end{split}
\end{equation}
where (\ref{est-1}) and Corollary \ref{co:4} have been used again.

The last two terms on the RHS of (\ref{eq:3-10}) are also bounded by $Z_{16} |t|\mathbf{r}_t^{-2}$ by similar discussion. This together with (\ref{eq:3-10}), (\ref{eq:3-11}),(\ref{eq:3-12-1}) and (\ref{eq:3-13}) complete the proof of Lemma \ref{lm:7}. \qed
\end{proof}

We continue with the proof of Proposition \ref{pr:3-4}. From the remark before Proposition \ref{pr:3-3} we get, for $V_t(2R|t|^\alpha,\frac{3}{4})$, that
\begin{equation}
 \label{eq:8-1}
    |\mathbf{r}_t^\frac{4}{3} \Lambda_{\tilde{\omega}_t}F_{H_t} |_{H_t}
    =|\mathbf{r}_t^\frac{4}{3} \Lambda_{\tilde{\omega}_t}\left(\bar{\partial}_t(\partial_t (f_t^*H_0)\cdot(f_t^*H_0)^{-1})\right)|_{H_t} 
    \leq Z_{17}\cdot  |t|\mathbf{r}_t^{-2}
\end{equation}
for some constant $Z_{17}>0$. Consequently, in this region we have
\begin{equation}
 \label{eq:20-4}
         |\mathbf{r}_t^\frac{4}{3} \Lambda_{\tilde{\omega}_t}F_{H_t} |_{H_t} \leq Z_{17}\cdot \frac{1}{4R^2} |t|^{1-2\alpha} .
\end{equation}
and one can estimate
\begin{equation}
 \label{eq:3-14}
   \begin{split}
            &\int_{V_t(2R|t|^\alpha,\frac{3}{4})}    |\mathbf{r}_t^\frac{8}{3} \Lambda_{\tilde{\omega}_t} F_{H_t} |_{H_t}^k \mathbf{r}_t^{-4}dV_t 
          =\int_{V_t(2R|t|^\alpha,\frac{3}{4})}    \mathbf{r}_t^{\frac{4}{3} k} |\mathbf{r}_t^\frac{4}{3} \Lambda_{\tilde{\omega}_t} F_{H_t} |_{H_t}^k \mathbf{r}_t^{-4}dV_t \\
      \leq & Z_{17}^k \int_{V_t(2R|t|^\alpha,\frac{3}{4})}    \mathbf{r}_t^{\frac{4}{3} k} (|t|\mathbf{r}_t^{-2})^k \mathbf{r}_t^{-4}dV_t 
      \leq  Z_{17}^kZ_{18} |t|^k \int_{\mathbf{r}_t=2R|t|^\alpha}^{\frac{3}{4}}  \mathbf{r}_t^{-\frac{2}{3}k-1} d\mathbf{r}_t\\
      \leq  & Z_{17}^kZ_{18} |t|^k \cdot\frac{3}{2k}(2R)^{-\frac{2}{3}k}  |t|^{-\frac{2}{3}\alpha k}.
   \end{split}
\end{equation}
We thus obtain
\begin{equation}
 \label{eq:3-15}
     \Arrowvert \Lambda_{\tilde{\omega}_t}F_{H_t} \Arrowvert_{L^k_{0,-4}(V_t(2R|t|^\alpha,\frac{3}{4}),\tilde{g}_t,H_t)}\leq Z_{19}|t|^{1-\frac{2}{3}\alpha}.
\end{equation}
This ends the discussion on the region $V_t(2R|t|^\alpha,\frac{3}{4})$. As for the region $X_t[\frac{3}{4}]$, because the geometry is uniform there it is easy to see that 
\begin{equation}
 \label{eq:20-7}
         |\mathbf{r}_t^\frac{4}{3} \Lambda_{\tilde{\omega}_t} F_{H_t} |_{H_t} \leq Z_{20}\cdot |t|.
\end{equation}
and
\begin{equation}
 \label{eq:3-16}
     \Arrowvert \Lambda_{\tilde{\omega}_t}F_{H_t} \Arrowvert_{L^k_{0,-4}(X_t[\frac{3}{4}],\tilde{g}_t,H_t)} \leq Z_{20,k} |t|
\end{equation}
when $t$ is small.

Finally, from (\ref{eq:3-18}), (\ref{eq:8-2}), (\ref{eq:20-4}) and (\ref{eq:20-7}) we get (\ref{eq:20-5}), and from (\ref{eq:3-18}), (\ref{eq:3-17}), (\ref{eq:3-15}), and (\ref{eq:3-16}) we get (\ref{eq:20-6}). The proof of Proposition \ref{pr:3-4} is complete now. \qed
\end{proof}

\noindent \textbf{Remark} From now on we fix $\alpha=\frac{3}{8}$. Then we have 
\begin{equation}
 \label{eq:F-1}
		|\mathbf{r}_t^\frac{4}{3} \Lambda_{\tilde{\omega}_t}F_{H_t} |_{H_t} \leq \tilde{Z}_0 |t|^\frac{1}{4}
\end{equation}
and 
\begin{equation}
 \label{eq:F-2}
		\Arrowvert \Lambda_{\tilde{\omega}_t}F_{H_t} \Arrowvert_{L^k_{0,-4}(X_t,\tilde{g}_t,H_t)}\leq \tilde{Z}_k |t|^\frac{3}{4}.
\end{equation}

\section{Contraction mapping argument}

Our background Hermitian metric on $\mathcal{E}_t$ as constructed in Section 5 is denoted by $H_t$. Let $\tilde{H}_t$ be another Hermitian metric on $\mathcal{E}_t$ and write $\tilde{h}=\tilde{H}_tH_t^{-1}=I+h$ where $h$ is $H_t$-symmetric.

It is known that the mean curvature 
\begin{equation*}
     \sqrt{-1}\Lambda_{\tilde{\omega}_t}F_{\tilde{H}_t}=\sqrt{-1}\Lambda_{\tilde{\omega}_t} \bar{\partial}((\partial_{H_t}(I+h))(I+h)^{-1})+\sqrt{-1}\Lambda_{\tilde{\omega}_t}F_{H_t}
\end{equation*}
of $\tilde{H}_t$ is $\tilde{H}_t$-symmetric.

To make it $H_t$-symmetric, consider a positive square root of $\tilde{H}_tH_t^{-1}$, denoted by $(\tilde{H}_tH_t^{-1})^{\frac{1}{2}}$. More explicitly, write $h=P^{-1}DP$ where $D$ is diagonal with positive eigenvalues, then $(\tilde{H}_tH_t^{-1})^{\frac{1}{2}}=P^{-1}(I+D)^\frac{1}{2}P$.\\[0.2cm]

\noindent \textbf{Remark} Write $(\tilde{H}_tH_t^{-1})^{\frac{1}{2}}=I+u(h)$. Then it is easy to see that the linear part of $u(h)$ in $h$ is $\frac{1}{2}h$.\\[0.2cm]

After twisting the mean curvature above by $I+u(h)$, we obtain
\begin{equation}
 \label{eq:4-1}
    \sqrt{-1}(I+u(h))^{-1}[\Lambda_{\tilde{\omega}_t} \bar{\partial}((\partial_{H_t}(I+h))(I+h)^{-1})+\Lambda_{\tilde{\omega}_t}F_{H_t}](I+u(h)),
\end{equation}
which is $H_t$-symmetric. The equation 
\begin{equation*}
    \sqrt{-1}\Lambda_{\tilde{\omega}_t}F_{\tilde{H}_t}=0
\end{equation*}
is equivalent to the equation
\begin{equation*}
    \sqrt{-1}(I+u(h))^{-1}[\Lambda_{\tilde{\omega}_t} \bar{\partial}((\partial_{H_t}(I+h))(I+h)^{-1})+\Lambda_{\tilde{\omega}_t}F_{H_t}](I+u(h))=0,
\end{equation*}
which can be written in the form
\begin{equation*}
    L_t(h)=Q_t(h)
\end{equation*}
where
\begin{equation*}
    L_t(h)=\sqrt{-1}\left(\Lambda_{\tilde{\omega}_t} \bar{\partial}\partial_{H_t}h+\frac{1}{2}[\Lambda_{\tilde{\omega}_t}F_{H_t},h] \right)
\end{equation*}
is a linear map from 
\begin{equation*}
       \text{Herm}_{H_t}(\text{End} (\mathcal{E}_t)):=\{H_t\text{-symmetric endomorphisms of}\, \mathcal{E}_t\}
\end{equation*}
to itself, and
\begin{equation}
 \label{eq:4-2}
   \begin{split}
      Q_t(h) &=-\sqrt{-1}(I+u(h))^{-1}(\Lambda_{\tilde{\omega}_t} \bar{\partial}\partial_{H_t}h)(I+h)^{-1} (I+u(h))+\sqrt{-1}\Lambda_{\tilde{\omega}_t} \bar{\partial}\partial_{H_t}h \\
               & -\sqrt{-1}(I+u(h))^{-1}\Lambda_{\tilde{\omega}_t} (\partial_{H_t} h\cdot(I+h)^{-1}\wedge  \bar{\partial}h\cdot(I+h)^{-1})(I+u(h)) \\
               &-\sqrt{-1}\left((I+u(h))^{-1}\Lambda_{\tilde{\omega}_t}F_{H_t}(I+u(h))-\frac{1}{2}[\Lambda_{\tilde{\omega}_t}F_{H_t},h] \right).
   \end{split}
\end{equation}
In the above formulas, we use the fact that $\frac{1}{2}h$ is the linear part of $u(h)$.

Notice that since
 \begin{equation*}
     \int_X \langle \sqrt{-1}\left(\Lambda_{\tilde{\omega}_t} \bar{\partial}\partial_{H_t}h+\frac{1}{2}[\Lambda_{\tilde{\omega}_t}F_{H_t},h]  \right),I \rangle_{H_t} dV_t=0
 \end{equation*}
we have an induced map from 
\begin{equation*}
  \begin{split}
       &\text{Herm}^0_{H_t}(\text{End} (\mathcal{E}_t)):= \\
       &\{H_t\text{-symmetric endomorphisms of}\, \mathcal{E}_t\,\text{which are orthogonal to}\, I\}
  \end{split}
\end{equation*}
to itself. Because (\ref{eq:4-1}) is a $H_t$-symmetric endomorphisms of $\mathcal{E}_t$ which are orthogonal to $I$, we see the same is true for $Q_t(h)$. In this section $h$ will always be a section for the bundle $\text{Herm}^0_{H_t}(\text{End} (\mathcal{E}_t))$.

We consider the contraction mapping problem via weighted norms introduced in Section 2. The metrics that define these norms and all the pointwise norms will be w.r.t. the balanced metrics $\tilde{g}_t$ on $X_t$ and the Hermitian metrics $H_t$ on $\mathcal{E}_t$, and the connections we use are always the Chern connections of $H_t$. Therefore we remove $\tilde{g}_t$ and $H_t$ from the subscripts of the norms for simplicity unless needed. 

As in Section 2, we now consider the following norms defined on the usual Sobolev space $L^k_l(\text{Herm}^0_{H_t}(\text{End} (\mathcal{E}_t)))$:
    \begin{equation*}
         \Arrowvert h \Arrowvert _{L^k_{l,\beta}} =\left( \sum_{j=0}^l \int_{X_t} |\mathbf{r}_t^{-\frac{2}{3}\beta+\frac{2}{3}j}\nabla^jh|_t^k\mathbf{r}_t^{-4}\, dV_t \right)^\frac{1}{k}.
    \end{equation*}
    As before, we use $L^k_{l,\beta}$ to denote $L^k_{l,\beta}(\text{Herm}^0_{H_t}(\text{End} (\mathcal{E}_t)))$ for simplicity.
The following Sobolev inequalities will be used in our discussion:
\begin{proposition}
  \label{pr:1-3}
    For each $l,p,q,r$ there exists a constant $C>0$ independent of $t$ such that for any section $h$ of $\text{Herm}^0_{H_t}(\text{End} (\mathcal{E}_t))$,
    \begin{equation*}
         \Arrowvert h \Arrowvert_{L^r_{l,\beta}}\leq C \Arrowvert h \Arrowvert_{L^p_{q,\beta}}
    \end{equation*}
    whenever $\frac{1}{r}\leq \frac{1}{p} \leq \frac{1}{r}+\frac{q-l}{6}$ and
    \begin{equation*}
         \Arrowvert h \Arrowvert_{C^l_{\beta}}\leq C \Arrowvert h \Arrowvert_{L^p_{q,\beta}}
    \end{equation*}
    whenever $\frac{1}{p} < \frac{q-l}{6}$. Here the norms are with respect to $H_t$ and $\tilde{g}_t$.\\[0.2cm]
\end{proposition}

We now begin the discussion on the properties of the operator $L_t$. 
\begin{lemma}
  \label{lm:1}
     For any given $0<\nu \ll 1$ and $t \neq0$ small enough, we have
     \begin{equation*}
         \Arrowvert h \Arrowvert_{L^2_{1,-2}} \leq 8|t|^{-2\nu}\Arrowvert L_t(h) \Arrowvert_{L^2_{0,-4}}.
     \end{equation*}
     In particular, the operator $L_t$ is injective on $L^2_{2,-2}(\text{Herm}^0_{H_t}(\text{End} (\mathcal{E}_t)))$.
\end{lemma}
    \begin{proof}
    Later in Proposition \ref{pr:5-1} we will show that for arbitrarily given $\nu >0 $, we have
    \begin{equation*}
        \Arrowvert h \Arrowvert_{L^2_{0,-2}} \leq |t|^{-\nu}\Arrowvert \mathbf{r}_t^\frac{2}{3} \partial_{H_t}h \Arrowvert_{L^2_{0,-2}}
    \end{equation*}
    for $t\neq0$ small enough.
    Using this one easily deduce that 
    \begin{equation*}
        \Arrowvert h \Arrowvert_{L^2_{1,-2}} \leq 2|t|^{-\nu} \Arrowvert \mathbf{r}_t^\frac{2}{3} \partial_{H_t}h \Arrowvert_{L^2_{0,-2}}
    \end{equation*}
     for $t\neq0$ small enough. Now
    \begin{equation*}
    \begin{split}
        \Arrowvert \mathbf{r}_t^\frac{2}{3} \partial_{H_t}h \Arrowvert_{L^2_{0,-2}}^2 
                                =& \int_{X_t} \langle\partial_{H_t}h,\partial_{H_t}h\rangle \, dV_t  
                                 = \int_{X_t} \langle \sqrt{-1}\Lambda_{\tilde{\omega}_t} \bar{\partial}\partial_{H_t}h,h\rangle \, dV_t \\ 
                            \leq &\int_{X_t} | L_t(h)-\frac{\sqrt{-1}}{2}[\Lambda_{\tilde{\omega}_t}F_{H_t},h] ||h| \, dV_t \\
                             \leq &\int_{X_t} | L_t(h)||h| \, dV_t
                                      + \int_{X_t} |h|^2|\Lambda_{\tilde{\omega}_t}F_{H_t}| \, dV_t.
    \end{split}
    \end{equation*}



From (\ref{eq:F-1}) we have $|\mathbf{r}_t^\frac{4}{3}\Lambda_{\tilde{\omega}_t}F_{H_t}| \leq \tilde{Z}_0|t|^\frac{1}{4}$. Therefore we can bound
\begin{equation*} 
   \begin{split}
     \int_{X_t} |h|^2|\Lambda_{\tilde{\omega}_t}F_{H_t}| \, dV_t
         =&\int_{X_t} |\mathbf{r}_t^{\frac{4}{3}}h|^2|\mathbf{r}_t^\frac{4}{3}\Lambda_{\tilde{\omega}_t}F_{H_t}| \mathbf{r}_t^{-4}\, dV_t \\
     \leq & \tilde{Z}_0|t|^\frac{1}{4} \int_{X_t} |\mathbf{r}_t^{\frac{4}{3}}h|^2 \mathbf{r}_t^{-4}\, dV_t \leq \tilde{Z}_0|t|^\frac{1}{4} \Arrowvert h \Arrowvert^2_{L^2_{1,-2}}
   \end{split}
\end{equation*}
Using this bound, we now have
\begin{equation*}
      \begin{split}
          \Arrowvert h \Arrowvert_{L^2_{1,-2}}^2 
                           \leq & 4|t|^{-2\nu}\Arrowvert \mathbf{r}_t^\frac{2}{3} \partial_{H_t}h \Arrowvert_{L^2_{0,-4}}^2 \\
                           \leq & 4|t|^{-2\nu} \left( \int_{X_t} | L_t(h)||h| \, dV_t+\tilde{Z}_0|t|^\frac{1}{4}\Arrowvert h \Arrowvert^2_{L^2_{1,-2}}\right) \\
                           \leq &  4|t|^{-2\nu} \left(\int_{X_t} | \mathbf{r}_t^{\frac{8}{3}}L_t(h)|^2 \mathbf{r}_t^{-4}\, dV_t \right)^\frac{1}{2}
                                               \left(\int_{X_t} | \mathbf{r}_t^{\frac{4}{3}}h|^2 \mathbf{r}_t^{-4}\, dV_t \right)^\frac{1}{2} 
                            			+4\tilde{Z}_0|t|^{\frac{1}{4}-2\nu}\Arrowvert h \Arrowvert^2_{L^2_{1,-2}} \\
                           \leq & 4|t|^{-2\nu} \Arrowvert L_t(h) \Arrowvert_{L^2_{0,-4}}\Arrowvert h \Arrowvert_{L^2_{1,-2}}
                            +4\tilde{Z}_0|t|^{\frac{1}{4}-2\nu}\Arrowvert h \Arrowvert^2_{L^2_{1,-2}}. \\
      \end{split} 
\end{equation*}
Therefore for $\nu \ll1$ and $t\neq 0$ small enough such that $4\tilde{Z}_0|t|^{\frac{1}{4}-2\nu} \leq \frac{1}{2}$, we have the desired result. \qed 
\end{proof}

    We conclude from this that for $k \geq 6$, the operator $L_t: L^k_{2,-2} \rightarrow L^k_{0,-4}$ is injective.
    
    The operator $L_t$ is also surjective.  First of all, $\Lambda_{\tilde{\omega}_t} \bar{\partial}\partial_{H_t}$ is a self-adjoint Fredholm operator, so it has index zero. Secondly, since for each $t \neq0$, $\Lambda_{\tilde{\omega}_t}F_{H_t}$ is a smooth function on $X_t$, the operator 
    \begin{equation*}
         h \rightarrow \frac{1}{2}[\Lambda_{\tilde{\omega}_t}F_{H_t},h] 
    \end{equation*}
    from $L^k_{2,-2}$ to $L^k_{0,-4}$ is a compact operator. Therefore $L_t$ has index zero, and the injectivity of $L_t$ implies its surjectivity. Let the inverse be denoted by $P_t$. 
\begin{proposition}
 \label{pr:4-1}
    There exist constants $\hat{Z}_k>0$ such that for any $0<\nu \ll1$ and $t\neq0$ small enough,
    \begin{equation}
        \Arrowvert h \Arrowvert_{L^k_{2,-2}} \leq \hat{Z}_k(-\log |t|)^\frac{1}{2}|t|^{-2\nu} \Arrowvert L_t(h) \Arrowvert_{L^k_{0,-4}}.
    \end{equation}
    Consequently, the norm of the operator $P_t:L^k_{0,-4} \rightarrow L^k_{2,-2}$ is bounded as 
    \begin{equation*}
			\Arrowvert P_t \Arrowvert \leq \hat{Z}_k(-\log |t|)^\frac{1}{2}|t|^{-2\nu}.
    \end{equation*}
\end{proposition}
 \begin{proof} 
 From the estimates of $\tilde{g}_t$ in Theorem \ref{th:0}, the estimates of $H_t$ in Proposition \ref{pr:3-3}, and the estimates (\ref{eq:F-3}) of $\mathbf{r}_t^\frac{4}{3}\Lambda_{\tilde{\omega}_t}F_{H_t}$ in Proposition \ref{pr:3-4}, we can apply Proposition \ref{pr:uniform} to the operator $\mathbf{r}_t^\frac{4}{3}L_t$, and obtain
    \begin{equation*}
     \begin{split}
         \Arrowvert h \Arrowvert_{L^k_{2,-2}}
         \leq &\hat{C}_k \left( \Arrowvert \mathbf{r}_t^\frac{4}{3}L_t(h) \Arrowvert_{L^k_{0,-2}}+\Arrowvert h \Arrowvert_{L^2_{0,-2}} \right)\\
         \leq &\hat{C}'_k \left( \Arrowvert L_t(h) \Arrowvert_{L^k_{0,-4}}+\Arrowvert h \Arrowvert_{L^2_{0,-2}} \right)\\
      \end{split}
    \end{equation*}
 		for constans $\hat{C}'_k>0$ independent of $t$.
By Lemma \ref{lm:1} and H$\ddot{\text{o}}$lder inequality,
     \begin{equation}
         \begin{split}
         			\Arrowvert h \Arrowvert_{L^2_{0,-2}}^2
        			\leq &64|t|^{-4\nu}\Arrowvert L_t(h) \Arrowvert_{L^2_{0,-4}}^2 
			    = 64|t|^{-4\nu}\int_{X_t} |\mathbf{r}_t^{\frac{8}{3}}L_t(h)|^2 \mathbf{r}_t^{-4} dV_t \\
			    \leq &64|t|^{-4\nu}\left(\int_{X_t} 1\cdot \mathbf{r}_t^{-4} dV_t\right)^{1-\frac{2}{k}} \left(\int_{X_t} |\mathbf{r}_t^{\frac{8}{3}}L_t(h)|^k \mathbf{r}_t^{-4} dV_t\right)^\frac{2}{k} \\
        		\leq &Z'_0|t|^{-4\nu} (-\log |t|)^{1-\frac{2}{k}}\Arrowvert L_t(h) \Arrowvert_{L^k_{0,-4}}^2 
		\leq Z'_0|t|^{-4\nu}(-\log |t|)\Arrowvert L_t(h) \Arrowvert_{L^k_{0,-4}}^2.
         \end{split}
    \end{equation}
for $t\neq 0$ small. The claim follows now.\qed
 \end{proof}
    Now we consider the contraction mapping problem for the map 
    \begin{equation*}
    U_t: L^k_{2,-2}\rightarrow L^k_{2,-2},\, \, \, \, U_t(h)=P_t(Q_t(h)).
    \end{equation*}
   Here $Q_t(h)$ is given in (\ref{eq:4-2}).
 
 Take $\beta'$ to be a number such that $0<\beta'-2 \ll 1$. We restrict ourselves to a ball $B(\beta')$ of radius $|t|^{\frac{\beta'}{3}}$ centered at 0 inside $L^k_{2,-2}$, and show that $U_t$ is a contraction mapping from the ball into itself when $t \neq 0$ is small enough.
 
\begin{proposition}
 \label{pr:4-2}
     For each $k$ large enough, there is a constant $\hat{Z}'_k>0$ such that when $t \neq 0$ is small enough the operator $h \mapsto Q_t(h)$ maps the ball of radius $|t|^{\frac{\beta'}{3}}$ in $L^k_{2,-2}$ into the ball of radius $\hat{Z}'_k|t|^{\frac{\beta'-2}{3}}\cdot |t|^{\frac{\beta'}{3}}$ in $L^k_{0,-4}$.
\end{proposition}
 
\begin{proof}
Note that when $k$ is large enough one has the Sobolev embedding $L^k_{2,-2}\hookrightarrow C^1_{-2}$. Proposition \ref{pr:1-3} implies the existence of a constant $C^{sb}_k$ independent of $t$ such that 
\begin{equation}
\label{eq:4-0}
       \Arrowvert h \Arrowvert_{C^1_{-2}} \leq C^{sb}_k \Arrowvert h \Arrowvert_{L^k_{2,-2}}.
\end{equation}
In this case, $\Arrowvert h \Arrowvert_{L^k_{2,-2}}<|t|^{\frac{\beta'}{3}}$ implies in particular that $|h|\mathbf{r}_t^{\frac{4}{3}}<C^{sb}_k|t|^{\frac{\beta'}{3}}$, and hence 
\begin{equation}
 \label{eq:4-01}
      |h|<C^{sb}_k|t|^{\frac{\beta'-2}{3}}.
\end{equation}
Therefore, because $0<\beta'-2$, when $t \neq 0$ is small it makes sense to take the inverse of $I+h$ and $I+u(h)$, and there are constants $Z'_1$ and $Z'_2$ such that, for $t\neq 0$ small,
\begin{equation}
 \label{eq:4-3}
   \begin{split}
       |(I+u(h))^{-1}&(\Lambda_{\tilde{\omega}_t}  \bar{\partial}\partial_{H_t}h)(I+h)^{-1} (I+u(h))-\sqrt{-1}\Lambda_{\tilde{\omega}_t} \bar{\partial}\partial_{H_t}h|  \\
       &\leq Z'_1|\bar{\partial}\partial_{H_t}h||h|  \leq Z'_1C^{sb}_k|t|^{\frac{\beta'-2}{3}}|\bar{\partial}\partial_{H_t}h|
         \leq Z'_1C^{sb}_k|t|^{\frac{\beta'-2}{3}} |\nabla_{H_t}^2 h|
   \end{split}
\end{equation}
and
\begin{equation}
 \label{eq:4-4}
       \max \{ |h|,\,\,|I+u(h)|,\,\, |(I+u(h))^{-1}|,\,\, |I+h|,\,\, |(I+h)^{-1}|\}<Z'_2.
\end{equation}

From the expression (\ref{eq:4-2}) for $Q_t$ we can bound it as
\begin{equation}
 \label{eq:4-5}
    \begin{split}
        |Q_t(h)|
                   \leq & |(I+u(h))^{-1}(\Lambda_{\tilde{\omega}_t} \bar{\partial}\partial_{H_t}h)(I+h)^{-1} (I+u(h))-\sqrt{-1}\Lambda_{\tilde{\omega}_t} \bar{\partial}\partial_{H_t}h| \\
                     +  & |(I+u(h))^{-1}\Lambda_{\tilde{\omega}_t} (\partial_{H_t} h\cdot(I+h)^{-1}\wedge  \bar{\partial}h\cdot (I+h)^{-1})(I+u(h)) |\\
                     +  & |((I+u(h))^{-1}| \cdot |\Lambda_{\tilde{\omega}_t}F_{H_t}| \cdot |(I+u(h))| + |h| \cdot |\Lambda_{\tilde{\omega}_t}F_{H_t}| \\
                   \leq & Z'_1C^{sb}_k |t|^{\frac{\beta'-2}{3}} |\nabla_{H_t}^2 h| + (Z'_2)^4|\nabla_{H_t} h|^2 
                            + ((Z'_2)^2+Z'_2)|\Lambda_{\tilde{\omega}_t}F_{H_t}|                   
    \end{split}
\end{equation}
where (\ref{eq:4-3}) and (\ref{eq:4-4}) are used.
 
 Now we estimate the $L^k_{0,-4}$-norm of $|\nabla_{H_t}^2 h|$ and $|\nabla_{H_t} h|^2$. First of all we have
  \begin{equation*}
    \begin{split}
        \Arrowvert |\nabla_{H_t}^2 h| \Arrowvert_{L^k_{0,-4}}^k &
     			=\int_{X_t} \mathbf{r}_t^{\frac{8}{3}k}|\nabla_{H_t}^2 h|^k \mathbf{r}_t^{-4}dV_t 
        =  \int_{X_t} |\mathbf{r}_t^{\frac{2}{3}(2+2)}\nabla_{H_t}^2 h|^k \mathbf{r}_t^{-4}dV_t 
        \leq   \Arrowvert h \Arrowvert_{L^k_{2,-2}}^k
    \end{split}
 \end{equation*}
 and hence 
 \begin{equation}
  \label{eq:4-6}
       \Arrowvert |\nabla_{H_t}^2 h| \Arrowvert_{L^k_{0,-4}}  \leq  \Arrowvert h \Arrowvert_{L^k_{2,-2}} \leq |t|^{\frac{\beta'}{3}}.
 \end{equation}
 
 Next we estimate 
 \begin{equation*}
    \begin{split}
       \Arrowvert |\nabla_{H_t} h|^2 \Arrowvert_{L^k_{0,-4}}^k &=\int_{X_t} \mathbf{r}_t^{\frac{8}{3}k}|\nabla_{H_t} h|^{2k}\mathbf{r}_t^{-4}dV_t \\
                                                          								  		 &=\int_{X_t} \mathbf{r}_t^{-\frac{4}{3} k}\mathbf{r}_t^{\frac{2}{3}(2+1)2k}|\nabla_{H_t} h|^{2k}\mathbf{r}_t^{-4}dV_t 
								               												  	 	  \leq |t|^{-\frac{2}{3}k}\Arrowvert h \Arrowvert_{L^{2k}_{1,-2}}^{2k}.
    \end{split}
 \end{equation*} 
By Proposition \ref{pr:1-3} we have, for large $k$, $\Arrowvert h \Arrowvert_{L^{2k}_{1,-2}} \leq \hat{C}^{sb}_k \Arrowvert h \Arrowvert_{L^{k}_{2,-2}} \leq \hat{C}^{sb}_k|t|^{\frac{\beta'}{3}}$ for some constant $\hat{C}^{sb}_k$ independent of $t$. Thus we get 
\begin{equation}
  \label{eq:4-7}
        \Arrowvert |\nabla_{H_t} h|^2 \Arrowvert_{L^k_{0,-4}} 
        \leq |t|^{-\frac{2}{3}} \Arrowvert h \Arrowvert_{L^{2k}_{1,-2}}^2 
        \leq  (\hat{C}^{sb}_k)^2|t|^{\frac{1}{3}(2\beta'-2) }<(\hat{C}^{sb}_k)^2|t|^{\frac{\beta'-2}{3}} |t|^{\frac{\beta'}{3}}.
\end{equation} 

From the remark after Proposition \ref{pr:3-4}, we have for some constants $\tilde{Z}_k>0$
\begin{equation}
 \label{eq:4-8}
         \Arrowvert \Lambda_{\tilde{\omega}_t}F_{H_t}  \Arrowvert_{L^k_{0,-4}} \leq \tilde{Z}_k|t|^{\frac{3}{4}} \leq \tilde{Z}_k|t|^{\frac{3}{4}-\frac{\beta'}{3}}|t|^{\frac{\beta'}{3}}.
\end{equation}
Note that for $0<\beta'-2\ll 1$, $\frac{3}{4}-\frac{\beta'}{3}>\frac{\beta'-2}{3}>0$. We fix such a $\beta'$.

Now, from (\ref{eq:4-5}), (\ref{eq:4-6}), (\ref{eq:4-7}), and (\ref{eq:4-8}) we have
\begin{equation}
 \label{eq:4-9}
    \begin{split}
         \Arrowvert Q_t(h)\Arrowvert_{L^k_{0,-4}}  
                   \leq 
                          &\left(Z'_1C^{sb}_k + (Z'_2)^4\cdot (\hat{C}^{sb}_k)^2  
                            + ((Z'_2)^2+Z'_2)\tilde{Z}_k \right)|t|^{\frac{\beta'-2}{3}}\cdot |t|^{\frac{\beta'}{3}}  
    \end{split}
\end{equation}
for $t\neq 0$ small enough. \qed
\end{proof}

Fix $\beta'$ as in Proposition \ref{pr:4-2} and choose $\nu<\frac{1}{6}\beta'-\frac{1}{3}$ in Proposition \ref{pr:4-1}, then for $t\neq 0$ sufficiently small, $U_t$ maps $B(\beta')$ to itself. Next we show

\begin{proposition}
     $U_t$ is a contraction mapping on $B(\beta')$ for $t\neq 0$ small enough.
\end{proposition} 
\begin{proof}
   We first show that when $t \neq 0$ is small enough and $k$ large enough, there are constants $\hat{Z}''_k>0$ such that for any $h_1$ and $h_2$ contained in $B(\beta')$, we have 
\begin{equation}
 \label{eq:18-4}
      \Arrowvert Q_t(h_1)-Q_t(h_2) \Arrowvert_{L^k_{0,-4}} \leq \hat{Z}''_k|t|^{\frac{\beta'-2}{3}}\Arrowvert h_1-h_2\Arrowvert_{L^k_{2,-2}}.
\end{equation}

   As discussed in Proposition \ref{pr:4-2}, for $i=1,2$ when $|h_i| \in B(\beta')$ we have $|h_i|<C^{sb}_k|t|^{\frac{\beta'-2}{3}}$ for some constants $C^{sb}_k$. In this case there is a constant $Z'_3$ independent of $t$ such that
\begin{equation}
 \label{eq:4-10}
     \begin{split}
         | (I+h_1)^{-1}(I+u(h_1)) - (I+h_2)^{-1}(I+u(h_2)) | & \leq Z'_3 |h_1-h_2|, \\
              			 | (I+u(h_1))^{-1} - (I+u(h_2))^{-1} | & \leq Z'_3 |h_1-h_2|, \\
			 						 | (I+u(h_1)) - (I+u(h_2)) | & \leq Z'_3 |h_1-h_2|,\\
								  | (I+h_1)^{-1} - (I+h_2)^{-1} | & \leq Z'_3 |h_1-h_2|, \\
									 		 | u(h_1) - u(h_2) | & \leq Z'_3 |h_1-h_2|.
     \end{split}
\end{equation}
Using these bounds, the bounds in (\ref{eq:4-4}), and the expression in (\ref{eq:4-2}) for $Q_t(h)$, it is not hard to see that for some constant $Z'_4$ we have
\begin{equation}
    \begin{split}
     & |Q_t(h_1)-Q_t(h_2)| \\
      \leq &Z'_4 ((|\Lambda_{\tilde{\omega}_t} \bar{\partial}\partial_{H_t}h_1 | + |\Lambda_{\tilde{\omega}_t} \bar{\partial}\partial_{H_t}h_2 |) |h_1-h_2| 
      									+ (|h_1|+|h_2|)|\Lambda_{\tilde{\omega}_t} \bar{\partial}\partial_{H_t}(h_1-h_2)| \\
                            &+(|\nabla_{H_t} h_1|^2+|\nabla_{H_t} h_2|^2) |h_1-h_2| \\
                            &+ (|\nabla_{H_t} h_1|+|\nabla_{H_t} h_2|) |\nabla_{H_t} (h_1-h_2)|
      						+|\Lambda_{\tilde{\omega}_t}F_{H_t} | |h_1-h_2| ) \\
      \leq &Z'_4 (|t|^{-\frac{2}{3} } ( |\Lambda_{\tilde{\omega}_t} \bar{\partial}\partial_{H_t}h_1 |+ |\Lambda_{\tilde{\omega}_t} \bar{\partial}\partial_{H_t}h_2 | 
      															+|\Lambda_{\tilde{\omega}_t}F_{H_t} |) |\mathbf{r}_t^{\frac{4}{3}}(h_1-h_2)|\\
								&+|t|^{-\frac{2}{3} } (|\nabla_{H_t} h_1|^2+|\nabla_{H_t} h_2|^2) |\mathbf{r}_t^{\frac{4}{3}}(h_1-h_2)| \\
								& +|t|^{-\frac{2}{3} } (|\mathbf{r}_t^{-\frac{2}{3}}\nabla_{H_t} h_1|+|\mathbf{r}_t^{-\frac{2}{3}}\nabla_{H_t} h_2|) |\mathbf{r}_t^{2}\nabla_{H_t} (h_1-h_2)| \\
      							& + (|h_1|+|h_2|)|\Lambda_{\tilde{\omega}_t} \bar{\partial}\partial_{H_t}(h_1-h_2)| )
	\end{split}
\end{equation}
where in the last line we use the fact that $|t|^{-\frac{2}{3}}\mathbf{r}_t^{\frac{4}{3}} \geq 1$ on $X_t$.
Therefore
\begin{equation}
 \label{eq:4-11}
    \begin{split}
        &\Arrowvert Q_t(h_1)-Q_t(h_2) \Arrowvert_{L^k_{0,-4}} \\
      \leq & Z'_4 (|t|^{-\frac{2}{3}} (\Arrowvert \Lambda_{\tilde{\omega}_t} \bar{\partial}\partial_{H_t}h_1 \Arrowvert_{L^k_{0,-4}} 
      														+\Arrowvert \Lambda_{\tilde{\omega}_t} \bar{\partial}\partial_{H_t}h_2 \Arrowvert_{L^k_{0,-4}} 
        													+\Arrowvert \Lambda_{\tilde{\omega}_t}F_{H_t} \Arrowvert_{L^k_{0,-4}}) \sup_{X_t}|\mathbf{r}_t^{\frac{4}{3}}(h_1-h_2)| \\
						&+|t|^{-\frac{2}{3}} (\Arrowvert |\nabla_{H_t} h_1|^2 \Arrowvert_{L^{k}_{0,-4}}
							 +\Arrowvert |\nabla_{H_t} h_2|^2 \Arrowvert_{L^{k}_{0,-4}} )\sup_{X_t}|\mathbf{r}_t^{\frac{4}{3}}(h_1-h_2)| \\
						&+|t|^{-\frac{2}{3}} ( \Arrowvert \mathbf{r}_t^{-\frac{2}{3}}\nabla_{H_t} h_1 \Arrowvert_{L^k_{0,-4}}
																+\Arrowvert \mathbf{r}_t^{-\frac{2}{3}}\nabla_{H_t} h_2 \Arrowvert_{L^k_{0,-4}} ) \sup_{X_t}|\mathbf{r}_t^{2}\nabla_{H_t}(h_1-h_2)| \\
            &+2C^{sb}_k|t|^{\frac{\beta'-2}{3}} \Arrowvert \Lambda_{\tilde{\omega}_t} \bar{\partial}\partial_{H_t}(h_1-h_2) \Arrowvert_{L^k_{0,-4}} ) 
       	\end{split}
\end{equation}
where (\ref{eq:4-01}) is used to bound $|h_1|+|h_2|$.

The first term in the RHS of (\ref{eq:4-11}) is bounded as
\begin{equation}
 \label{eq:4-12}
    \begin{split}
             &|t|^{-\frac{2}{3}} (\Arrowvert \Lambda_{\tilde{\omega}_t} \bar{\partial}\partial_{H_t}h_1 \Arrowvert_{L^k_{0,-4}} 
      														+\Arrowvert \Lambda_{\tilde{\omega}_t} \bar{\partial}\partial_{H_t}h_2 \Arrowvert_{L^k_{0,-4}} 
        													+\Arrowvert \Lambda_{\tilde{\omega}_t}F_{H_t} \Arrowvert_{L^k_{0,-4}}) \sup_{X_t}|\mathbf{r}_t^{\frac{4}{3}}(h_1-h_2)| \\
       \leq & (2+\tilde{Z}_k)|t|^{-\frac{2}{3}} |t|^{\frac{\beta'}{3}} \Arrowvert h_1-h_2\Arrowvert_{C^0_{-2}}
               =(2+\tilde{Z}_k)C^{sb}_k|t|^{\frac{\beta'-2}{3}} \Arrowvert h_1-h_2\Arrowvert_{L^k_{2,-2}}
     \end{split}
\end{equation}
for $t$ small enough. Here we have used (\ref{eq:4-6}), (\ref{eq:4-8}) and (\ref{eq:4-0}).

The second term in the RHS of (\ref{eq:4-11}) is bounded as
\begin{equation}
 \label{eq:4-13}
    \begin{split}
    		|t|^{-\frac{2}{3}}& (\Arrowvert |\nabla_{H_t} h_1|^2 \Arrowvert_{L^{k}_{0,-4}}
							 +\Arrowvert |\nabla_{H_t} h_2|^2 \Arrowvert_{L^{k}_{0,-4}} )\sup_{X_t}|\mathbf{r}_t^{\frac{4}{3}}(h_1-h_2)| \\
      \leq &|t|^{-\frac{2}{3}}\cdot 2(\hat{C}^{sb}_k)^2|t|^{\frac{\beta'-2}{3}} |t|^{\frac{\beta'}{3}} \cdot \Arrowvert h_1-h_2\Arrowvert_{C^0_{-2}}    
         \leq 2(\hat{C}^{sb}_k)^2C^{sb}_k|t|^{\frac{2\beta'-4}{3}}  \Arrowvert h_1-h_2\Arrowvert_{L^k_{2,-2}}
     \end{split}
\end{equation} 
for $t$ small enough. Here (\ref{eq:4-7}) and (\ref{eq:4-0}) are used.

To bound the third term in the RHS of (\ref{eq:4-11}), we first estimate that, for $\Arrowvert h \Arrowvert_{L^k_{2,-2}} \leq |t|^{\frac{\beta'}{3}}$,
\begin{equation*}
    \begin{split}
        \Arrowvert \mathbf{r}_t^{-\frac{2}{3}}\nabla_{H_t} h \Arrowvert_{L^k_{0,-4}}^k &
        =\int_{X_t} \mathbf{r}_t^{\frac{8}{3}k}|\mathbf{r}_t^{-\frac{2}{3}}\nabla_{H_t} h|^k \mathbf{r}_t^{-4}dV_t 
     \leq \int_{X_t} | \mathbf{r}_t^{2} \nabla_{H_t} h|^k \mathbf{r}_t^{-4}dV_t 
     \leq  \Arrowvert h \Arrowvert_{L^k_{2,-2}}^k 
    \end{split}
 \end{equation*}
 and hence 
 \begin{equation*}
       \Arrowvert \mathbf{r}_t^{-\frac{2}{3}}\nabla_{H_t} h \Arrowvert_{L^k_{0,-4}}  \leq \Arrowvert h \Arrowvert_{L^k_{2,-2}} 
       \leq |t|^{\frac{\beta'}{3}}.
 \end{equation*}
 Therefore we have
 \begin{equation}
  \label{eq:4-14}
       \begin{split}
          & |t|^{-\frac{2}{3}} \left( \Arrowvert \mathbf{r}_t^{-\frac{2}{3}}\nabla_{H_t} h_1 \Arrowvert_{L^k_{0,-4}}
				+\Arrowvert \mathbf{r}_t^{-\frac{2}{3}}\nabla_{H_t} h_2 \Arrowvert_{L^k_{0,-4}} \right ) \sup_{X_t}|\mathbf{r}_t^{2}\nabla_{H_t}(h_1-h_2)|\\
		 	\leq & 2|t|^{-\frac{2}{3}} |t|^{\frac{\beta'}{3}} \Arrowvert h_1-h_2\Arrowvert_{C^1_{-2}} 
			\leq 2C^{sb}_k|t|^{\frac{\beta'-2}{3}} \Arrowvert h_1-h_2\Arrowvert_{L^k_{2,-2}}
		\end{split}
 \end{equation}
where the above estimate and (\ref{eq:4-0}) are used.

Finally, it is easy to see that the last term in (\ref{eq:4-11}) is also bounded as
\begin{equation}
 \label{last}
		2C^{sb}_k|t|^{\frac{\beta'-2}{3}} \Arrowvert \Lambda_{\tilde{\omega}_t} \bar{\partial}\partial_{H_t}(h_1-h_2) \Arrowvert_{L^k_{0,-4}}
		\leq 2C^{sb}_k |t|^{\frac{\beta'-2}{3}} \Arrowvert h_1-h_2\Arrowvert_{L^k_{2,-2}}.
\end{equation}
Plugging  (\ref{eq:4-12}), (\ref{eq:4-13}), (\ref{eq:4-14}) and (\ref{last}) into (\ref{eq:4-11}) proves (\ref{eq:18-4}).  

Recall that we have chosen $\nu<\frac{1}{6}\beta'-\frac{1}{3}$. Therefore (\ref{eq:18-4}) and the bound for the norm of $P_t$ given in Proposition \ref{pr:4-1} show that for $t \neq 0$ small enough $U_t$ is a contraction mapping, as desired. \qed 
\end{proof}

Using the contraction mapping theorem on $U_t:B(\beta') \rightarrow B(\beta')$, we have now proved 
\begin{theorem}
       For $t\neq 0$ sufficiently small, the bundle $\mathcal{E}_t$ admits a smooth Hermitian-Yang-Mills metric with respect to the balanced metric $\tilde{\omega}_t$.\\[0.2cm]
\end{theorem}

\section{Proposition 7.1}

What remains to be proved is the following proposition.
\begin{proposition}
 \label{pr:5-1}
For each $\nu >0$, we have
  \begin{equation*}
      \int_{X_t} |\mathbf{r}_t^{\frac{4}{3}}h|^2\mathbf{r}_t^{-4}dV_t\leq |t|^{-\nu}\int_{X_t} |\partial_{H_t}h|^2dV_t
  \end{equation*}
  for $t\neq 0$ small.
\end{proposition}
 
 We can regard this proposition as a problem of smallest eigenvalue of a self-adjoint operator. Consider the pairing 
\begin{equation*}
     \langle h_1,h_2\rangle_{L^2_{0,-2}}:=\int_{X_t} \mathbf{r}_t^{\frac{8}{3}}\langle h_1,h_2 \rangle_{H_t} \mathbf{r}_t^{-4}dV_t.
\end{equation*}
One can compute
\begin{equation*}
   \begin{split}
        &\int_{X_t} \langle \partial_{H_t}h_1,\partial_{H_t}h_2 \rangle_{H_t,\tilde{g}_t} dV_t 
       = \int_{X_t}  \langle \sqrt{-1}\Lambda_{\tilde{\omega}_t}\bar{\partial} \partial_{H_t} h_1, h_2 \rangle_{H_t} dV_t \\
  		=&\int_{X_t} \mathbf{r}_t^{\frac{8}{3}} \langle \sqrt{-1}\mathbf{r}_t^\frac{4}{3}\Lambda_{\tilde{\omega}_t}\bar{\partial} \partial_{H_t} h_1, h_2 \rangle_{H_t} \mathbf{r}_t^{-4}dV_t
      =\langle \sqrt{-1}\mathbf{r}_t^\frac{4}{3}\Lambda_{\tilde{\omega}_t}\bar{\partial} \partial_{H_t} h_1,h_2\rangle_{L^2_{0,-2}}.
   \end{split}
\end{equation*}
 From this we see that the operator $\sqrt{-1}\mathbf{r}_t^\frac{4}{3}\Lambda_{\tilde{\omega}_t}\bar{\partial} \partial_{H_t}$ is self-adjoint on $L^2_{0,-2}(\text{Herm}^0_{H_t}(\text{End} (\mathcal{E}_t)))$.

Define the number
\begin{equation*}
     \lambda_t:=\inf_{0\neq h\in L^2_{0,-2}(\text{Herm}^0_{H_t}(\text{End} (\mathcal{E}_t)))}\frac{\int_{X_t} |\partial_{H_t}h|^2 dV_t}{\int_{X_t} |h|^2\mathbf{r}_t^{-\frac{4}{3}}dV_t}.
\end{equation*}
It is not hard to show that the above infimum is achieved at those $h$ satisfying 
\begin{equation}
 \label{eq:11-2}
				\sqrt{-1}\mathbf{r}_t^\frac{4}{3}\Lambda_{\tilde{\omega}_t}\bar{\partial} \partial_{H_t}h=\lambda_t h,
\end{equation}
 i.e., $h$ is an eigenvector of the operator $\sqrt{-1}\mathbf{r}_t^\frac{4}{3}\Lambda_{\tilde{\omega}_t}\bar{\partial} \partial_{H_t}$ corresponding to the smallest nonzero eigenvalue $\lambda_t$ on $L^2_{0,-2}(\text{Herm}^0_{H_t}(\text{End} (\mathcal{E}_t)))$. For each $t \neq 0$ let $h_t$ be such an element which satisfies $\Arrowvert h_t \Arrowvert_{L^2_{0,-2}}=1$.

\begin{proof}
Our goal is to show that for each $\nu >0$ one has $\lambda_t>|t|^\nu$ when $t \neq 0$ is small.
Suppose such a bound does not exist. Then for some $\nu>0$ there is a sequence $\{t_n\}$ converging to $0$ such that $\lambda_{t_n} \leq |t_n|^\nu$. The endomorphisms $h_{t_n}$ introduced above satisfy 
\begin{equation}
 \label{eq:11-1}
     \sqrt{-1}\mathbf{r}^\frac{4}{3}\Lambda_{\tilde{\omega}_n}\bar{\partial} \partial_{H_n}h_n=\lambda_n h_n,
\end{equation}
\begin{equation}
 \label{eq:11-3}
		\int_{X_n} |h_n|^2\mathbf{r}^{-\frac{4}{3}}dV_n=1
\end{equation}
and
\begin{equation}
 \label{eq:11-4}
			\int_{X_n} |\partial_{H_n}h_n|^2 dV_n \leq |t_n|^\nu.
\end{equation}
Here we use the notations $\mathbf{r}$, $\tilde{\omega}_n$, $H_n$ and $\lambda_n$ to denote $\mathbf{r}_{t_n}$, $\tilde{\omega}_{t_n}$, $H_{t_n}$ and $\lambda_{t_n}$, respectively. In the following we will replace the subscripts $t_n$ with $n$.

For each fixed $\delta>0$ and $n$ sufficiently large, because the Riemannian manifold $(X_n[\delta], \tilde{\omega}_n)$ has uniform geometry, and because the coefficients  in the equations (\ref{eq:11-1}) are uniformly bounded, there is a constant $C$ independent of large $n$ such that
\begin{equation*}
 \label{Lp}
		\Arrowvert h_n \Arrowvert_{L^p_3(X_n[2\delta])} 
		\leq C\Arrowvert h_n \Arrowvert_{L^2(X_n[\delta])}
		\leq C'\left(\int_{X_n} |h_n|^2\mathbf{r}^{-\frac{4}{3}}dV_n\right)^\frac{1}{2} \leq C'
\end{equation*}
where $C'$ depends only on $\delta$ and $p$. For $p$ large enough we see that $\Arrowvert h_n \Arrowvert_{C^2(X_n[2\delta])} $ is bounded independent of $n$. Therefore by using the diagonal argument, there is a subsequence of $\{h_n\}$ converging to an $H_{0}$-symmetric endomorphism $h$ in the $C^1$ sense over each compactly embedded open subset of $X_{0,sm}$. From (\ref{eq:11-4}) one sees that $\bar{\partial}h=0$ over $X_{0,sm}$. But then $h$ is a holomorphic endomorphism of $\mathcal{E}|_{\hat{X}\backslash \bigcup C_i}$, and by Hartog's Theorem it extends to a holomorphic endomorphism of $\mathcal{E}$ over $\hat{X}$. Since $\mathcal{E}$ is irreducible, the existence of a HYM metric on $\mathcal{E}$ implies that it is stable and hence simple. Therefore $h=\mu I$ for some constant $\mu$.


\begin{lemma}
  \label{lm:3}
There exists an $0<\iota<\frac{1}{6}$ and a constant $C_{10}>0$ such that for any $0<\delta<\frac{1}{4}$ and large $n$,
\begin{equation*}
   \int_{V_n(\delta)} |h_n|^2 \mathbf{r}^{-\frac{4}{3}}dV_n \leq C_{10}\delta^{2\iota}.
\end{equation*}
\end{lemma}

Let's assume the lemma first. Then we have
\begin{equation*}
   \begin{split}
       \int_{X_{0,sm}} |h|^2 \mathbf{r}^{-\frac{4}{3}}dV_0 =\lim_{\delta\rightarrow 0}\int_{X_{0}[\delta]} |h|^2 \mathbf{r}^{-\frac{4}{3}} dV_0 
      															       &=\lim_{\delta\rightarrow 0}\lim_{n\rightarrow \infty}\int_{X_{n}[\delta]} |h_n|^2\mathbf{r}^{-\frac{4}{3}}dV_n \\
															           &\geq \lim_{\delta\rightarrow 0}\lim_{n\rightarrow \infty} (1-C_{10}\delta^{2\iota})=1.
   \end{split}
\end{equation*}
On the other hand
\begin{equation*}
   \begin{split}
       \int_{X_{0,sm}} |h|^2 \mathbf{r}^{-\frac{4}{3}}dV_0 &=\lim_{\delta\rightarrow 0}\lim_{n\rightarrow \infty}\int_{X_{n}[\delta]} |h_n|^2\mathbf{r}^{-\frac{4}{3}}dV_n 
        																   \leq \lim_{\delta\rightarrow 0}\lim_{n\rightarrow \infty} 1=1,
   \end{split}
\end{equation*}
so we have
\begin{equation*}
      \int_{X_{0,sm}} |h|^2\mathbf{r}^{-\frac{4}{3}}dV_0=1.
\end{equation*}
Since $h=\mu I$, this implies that
\begin{equation}
 \label{eq:f-1}
    |\mu|^2=\left(\text{rank}(\mathcal{E})\int_{X_{0,sm}} \mathbf{r}^{-\frac{4}{3}}dV_0 \right)^{-1}.
\end{equation}
On the other hand, note that for each $\delta>0$,
\begin{equation*}
    |\mu|\text{rank}(\mathcal{E})\text{Vol}_0(X_0[\delta])=\left| \int_{X_0[\delta]}\text{tr}\,h\, dV_0   \right|=\lim_{n \rightarrow \infty}\left| \int_{X_n[\delta]}\text{tr}\,h_n\, dV_n   \right|.
\end{equation*}
Because
\begin{equation*}
    \int_{X_n}\text{tr}\,h_n\, dV_n=0,
\end{equation*}
we have
\begin{equation}
 \label{eq:f-2}
   \begin{split}
      &|\mu|\text{rank}(\mathcal{E})\text{Vol}_0(X_0[\delta])
      		=\lim_{n\rightarrow \infty}\left| \int_{X_n[{\delta}]}\text{tr}\,h_n\, dV_n   \right|
      			=\lim_{n\rightarrow \infty}\left| \int_{V_n(\delta)}\text{tr}\,h_n\, dV_n   \right| \\
     		&\leq C_1\lim_{n\rightarrow \infty}\left( \int_{V_n(\delta)} |h_n|^2dV_n\right)^\frac{1}{2}
				\leq C_2\lim_{n\rightarrow \infty} \left( \int_{V_n(\delta)} |h_n|^2\mathbf{r}^{-\frac{4}{3}}dV_n \right)^\frac{1}{2} \leq C_3\delta^\iota.
   \end{split}
\end{equation}

Now choose $\delta$ small enough such that
\begin{equation}
 \label{eq:f-3}
    \text{Vol}_0(X_0[\delta])\geq \frac{1}{2}\text{Vol}_0(X_0)
\end{equation}
and
\begin{equation}
 \label{eq:f-4}
    \left(\text{rank}(\mathcal{E})\int_{X_{0,sm}} \mathbf{r}^{-\frac{4}{3}}dV_0 \right)^{-\frac{1}{2}}>\frac{2C_3\delta^\iota}{\text{rank}(\mathcal{E})\text{Vol}_0(X_0)}.
\end{equation}
We see that a contradiction arises from (\ref{eq:f-1})-(\ref{eq:f-4}). We have thus shown Proposition \ref{pr:5-1}.\qed

\noindent{\scshape Proof of Lemma \ref{lm:3}}
First of all, by H$\ddot{\text{o}}$lder inequality,
\begin{equation*}
	\int_{V_n(\delta)} |h_n|^2 \mathbf{r}^{-\frac{4}{3}} dV_n \leq \left( \int_{V_n(\frac{1}{4})} |h_n|^3 \mathbf{r}^{-3\iota} dV_n \right)^\frac{2}{3} \left(\int_{V_n(\delta)} \mathbf{r}^{-4+6\iota} dV_n \right)^\frac{1}{3}.
\end{equation*}
Because
\begin{equation*}
		\left(\int_{V_n(\delta)} \mathbf{r}^{-4+6\iota} dV_n \right)^\frac{1}{3} \leq C_3\delta^{2\iota},
\end{equation*}
it is enough to prove that 
\begin{equation*}
		\left( \int_{V_n(\frac{1}{4})} |h_n|^3 \mathbf{r}^{-3\iota} dV_n \right)^\frac{2}{3} \leq C_4
\end{equation*}
for some constant $C_4>0$.

The proof makes use of Michael-Simon's Sobolev inequality \cite{MS} which we now describe. Let $M$ be an $m$-dimensional submanifold in $\mathbb{R}^N$. Denote the mean curvature vector of $M$ by $H$. Then for any nonnegative function $f$ on $M$ with compact support, one has
\begin{equation}
 \label{eq:9-1}
       \left( \int_M f^\frac{m}{m-1} dV_{g_E} \right) \leq C(m)\int_M (|\nabla f|_{g_E}+|H|\cdot f)\,dV_{g_E}
\end{equation}
where $C(m)$ is a constant depending only on $m$. Here all metrics and norms are the induced ones from the Euclidean metric on $\mathbb{C}^4$. We denote this induced metric by $g_E$. Do not confuse this metric with the metric $g_e$ appearing in earlier sections. In our case $M$ is the space $V_t(\frac{1}{2})$ identified as part of the submanifold $Q_{t} \subset \mathbb{C}^4$. As pointed out in \cite{FLY}, the relations between the volumes and norms for the CO-metric $g_{co,t}$ and those for the induced metric $g_E$ are
\begin{equation}
 \label{eq:r-1}
     dV_{g_{co,t}}=\frac{2}{3}\mathbf{r}_t^{-2}dV_{g_E}
\end{equation}
and 
\begin{equation}
 \label{eq:r-2}
     |\nabla f|_{g_E}^2 \leq C\mathbf{r}_t^{-\frac{2}{3}}|\nabla f|_{g_{co,t}}^2
\end{equation}
for any smooth function $f$ on $V_t(\frac{\delta}{2})$.

Let $\tau(\mathbf{r})$ be a cutoff function defined on $V_n(1)$ such that $\tau(\mathbf{r})=1$ when $\mathbf{r}\leq \frac{1}{4}$ and $\tau(\mathbf{r})=0$ when $\mathbf{r} \geq \frac{1}{2}$. Extend it to $X_n$ by zero. From (\ref{eq:r-1}) we have
\begin{equation}
 \label{eq:9-2}
     \int_{V_n(\frac{1}{4})} | h_n |^3\mathbf{r}^{-3\iota}dV_{co,n}
         \leq  \frac{2}{3}\int_{V_n(\frac{1}{2})} |h_n |^{3} \mathbf{r}^{-3\iota-2} \tau^3  dV_{g_E}.
\end{equation}
where $dV_{co,n}$ is the volume form with respect to the CO-metric $\omega_{co,t_n}$.

Moreover, using H$\ddot{\text{o}}$lder inequality, one can deduce from (\ref{eq:9-1}) that
\begin{equation*}
    \left( \int_{V_n(\frac{1}{2})} f^3dV_{g_E} \right)^\frac{2}{3} \leq C\int_{V_n(\frac{1}{2})} |\nabla f|_{g_E}^2dV_{g_E} ,
\end{equation*}
and using (\ref{eq:r-1}) and (\ref{eq:r-2}) we get
\begin{equation}
 \label{eq:r-3}
    \left( \int_{V_n(\frac{1}{2})} f^3dV_{g_E} \right)^\frac{2}{3} \leq C_5\int_{V_n(\frac{1}{2})} |\nabla f|_{co,n}^2\mathbf{r}^\frac{4}{3}dV_{g_{co,n}}
\end{equation}
where $|\cdot|_{co,n}$ is the used to denote $|\cdot|_{g_{co,t_n}}$.

Apply (\ref{eq:r-3}) to $f=|h_n| \mathbf{r}^{-\iota-\frac{2}{3}}\tau$, and then together with (\ref{eq:9-2}) (and Lemma \ref{lm:0-1}) we have

\begin{equation}
 \label{eq:19-1}
   \begin{split}
		&\left( \int_{V_n(\frac{1}{4})} |h_n|^3 \mathbf{r}^{-3\iota} dV_n \right)^\frac{2}{3} \\
		\leq & \tilde{C}_1^\frac{2}{3}\left( \int_{V_n(\frac{1}{4})} |h_n|^3 \mathbf{r}^{-3\iota} dV_{co,n} \right)^\frac{2}{3} 
					\leq \tilde{C}_1^\frac{2}{3}\left( \frac{2}{3}\int_{V_n(\frac{1}{2})} (|h_n| \mathbf{r}^{-\iota-\frac{2}{3}}\tau)^3  dV_{g_E} \right)^\frac{2}{3} \\
		\leq & C_6\int_{V_n(\frac{1}{2})} |\nabla (|h_n| \mathbf{r}^{-\iota-\frac{2}{3}}\tau)|_{co,n}^2 \mathbf{r}^\frac{4}{3}  dV_{co,n} \\
		\leq & 3C_6\int_{V_n(\frac{1}{2})} |\nabla |h_n| |_{co,n}^2 \mathbf{r}^{-2\iota}\tau^2  dV_{co,n} 
				+3C_6\int_{V_n(\frac{1}{2})} |h_n|^2 |\nabla \mathbf{r}^{-\iota-\frac{2}{3}}|_{co,n}^2 \tau^2 \mathbf{r}^\frac{4}{3}   dV_{co,n} \\
				&+ 3C_6\int_{V_n(\frac{1}{2})} |h_n|^2 \mathbf{r}^{-2\iota} |\nabla\tau|_{co,n}^2  dV_{co,n}. \\
   \end{split}
\end{equation}

The third term on the RHS of (\ref{eq:19-1}) is an integral over $V_n(\frac{1}{4},\frac{1}{2})$ in which the support of $\nabla \tau$ lies. From (\ref{eq:11-3}) one sees that it is bounded by some constant $C_7>0$ independent of $n$. Later whenever we encounter an integral with a derivative of $\tau$ in the integrant, we will bound it by a constant for the same reason.

Because $h_n$ is $H_n$-hermitian symmetric, $\bar{\partial} h_n=(\partial_{H_n} h_n)^{*_{H_n}}$, and so the first term on the RHS of (\ref{eq:19-1}) can be bounded as
\begin{equation}
 \label{eq:r-4}
   \begin{split}
		 \int_{V_n(\frac{1}{2})} |\nabla |h_n| |_{co,n}^2 \mathbf{r}^{-2\iota}\tau^2  dV_{co,n} 
		\leq &\int_{V_n(\frac{1}{2})} (\langle \partial_{H_n} h_n, \partial_{H_n} h_n \rangle_{co,n}+\langle \bar{\partial} h_n, \bar{\partial} h_n \rangle_{co,n}) \mathbf{r}^{-2\iota}\tau^2  dV_{co,n} \\
		   = &2\int_{V_n(\frac{1}{2})} \langle \partial_{H_n} h_n, \partial_{H_n} h_n \rangle_{co,n} \mathbf{r}^{-2\iota}\tau^2  dV_{co,n}
		   \leq \tilde{C}_3|t_n|^{\nu-\iota}
   \end{split}
\end{equation}
for some constant $\tilde{C}_3>0$ independent of $n$. The last inequality follows from (\ref{eq:11-4}) and Lemma \ref{lm:0-1}.
We now fix an $\iota$ such that $0<\iota<\min\{\frac{1}{6},\nu\}$. Then we see that as $n$ goes to infinity, this term goes to zero.

Finally we deal with the second term on the RHS of (\ref{eq:19-1}). It can be bounded as
\begin{equation}
 \label{eq:19-2}
    \begin{split}
		\int_{V_n(\frac{1}{2})} |h_n|^2 |\nabla \mathbf{r}^{-\iota-\frac{2}{3}}|_{co,n}^2 \tau^2 \mathbf{r}^\frac{4}{3}   dV_{co,n}
		\leq &C_8\int_{V_n(\frac{1}{2})} |h_n|^2 \mathbf{r}^{-2\iota-\frac{4}{3}} \tau^2  dV_{co,n} 
	\end{split}
\end{equation}
for some constant $C_8>0$. hence it is enough to bound the term on the right.

To do so, we introduce the notation $\phi_2=\mathbf{r}^{-2\iota}$, and denote $\partial_{H_n}^{\phi_2}=\partial_{H_n}+\partial \log \phi_2 \wedge$. We can estimate
\begin{equation*}
   \begin{split}
	0 \leq& \int_{V_n(\frac{1}{2})} \langle \bar{\partial} h_n, \bar{\partial} h_n \rangle_{\tilde{g}_n} \phi_2\tau^2  dV_{n} 
		 \leq  -\int_{V_n(\frac{1}{2})} \langle \sqrt{-1}\Lambda_{\tilde{\omega}_n} \partial_{H_n}^{\phi_2} \bar{\partial} h_n, h_n \rangle_{\tilde{g}_n} \phi_2\tau^2  dV_{n}+C_7 \\
		    =& \int_{V_n(\frac{1}{2})} -\langle \sqrt{-1}\Lambda_{\tilde{\omega}_n} \partial_{H_n}^{\phi_2} \bar{\partial} h_n+\sqrt{-1}\Lambda_{\tilde{\omega}_n} \bar{\partial} \partial_{H_n}^{\phi_2} h_n, h_n \rangle_{\tilde{g}_n}\phi_2\tau^2  dV_{n} \\ 
		    &+\int_{V_n(\frac{1}{2})} \langle \sqrt{-1}\Lambda_{\tilde{\omega}_n}  \bar{\partial} \partial_{H_n}^{\phi_2} h_n, h_n \rangle_{\tilde{g}_n} \phi_2\tau^2  dV_{n} 
		    +C_7.
   \end{split}
\end{equation*}
One can compute that 
\begin{equation*}
		\sqrt{-1}\Lambda_{\tilde{\omega}_n} \partial_{H_n}^{\phi_2} \bar{\partial} h_n+\sqrt{-1}\Lambda_{\tilde{\omega}_n} \bar{\partial} \partial_{H_n}^{\phi_2} h_n
		= -[\sqrt{-1}\Lambda_{\tilde{\omega}_n}F_{H_n}, h_n]+(\sqrt{-1}\Lambda_{\tilde{\omega}_n} \bar{\partial} \partial \log\phi_2) h_n,
\end{equation*}
and so we have
\begin{equation*}
   \begin{split}
		0\leq  & \int_{V_n(\frac{1}{2})} \langle  [\sqrt{-1}\Lambda_{\tilde{\omega}_n}F_{H_n}, h_n],h_n\rangle_{\tilde{g}_n}-\langle (\sqrt{-1}\Lambda_{\tilde{\omega}_n} \bar{\partial} \partial \log\phi_2) h_n, h_n \rangle_{\tilde{g}_n} \phi_2\tau^2  dV_{n} \\ 
		    &+\int_{V_n(\frac{1}{2})} \langle  \partial_{H_n}^{\phi_2} h_n, \partial_{H_n}^{\phi_2} h_n \rangle_{\tilde{g}_n} \phi_2\tau^2  dV_{n} 
		    +C_7 \\
		\leq &2 \int_{V_n(\frac{1}{2})}  |\Lambda_{\tilde{\omega}_n}F_{H_n}| | h_n|^2 \phi_2\tau^2  dV_{n}
			+ 2\int_{V_n(\frac{1}{2})} |\partial_{H_n} h_n|_{\tilde{g}_n}^2 \phi_2\tau^2  dV_{n} \\
			&+  \int_{V_n(\frac{1}{2})} (2|\partial\log \phi_2|_{\tilde{g}_n}^2-\sqrt{-1}\Lambda_{\tilde{\omega}_n} \bar{\partial} \partial \log\phi_2) | h_n|^2 \phi_2\tau^2  dV_{n} +C_7
   \end{split}
\end{equation*}
To proceed, we use the bound $|\Lambda_{\tilde{\omega}_n}F_{H_n}| \leq \tilde{Z}_0\mathbf{r}^{-\frac{4}{3}}|t_n|^{\frac{1}{4}}$ from the remark at the end of Section 5 to deal with the first term. We use (\ref{eq:11-4}) to take care of the second term. Finally, we have
\begin{equation*}
			|\partial\log \phi_2|_{\tilde{g}_n}^2<3\iota^2 \mathbf{r}^{-\frac{4}{3}}\,\,\,\text{and}\,\,\sqrt{-1}\Lambda_{\tilde{\omega}_n} \bar{\partial} \partial \log\phi_2>\iota \mathbf{r}^{-\frac{4}{3}},
\end{equation*}
which follow from (bottom of) p.31 of \cite{FLY} together with the observation $\sqrt{-1}\bar{\partial} \partial \log\phi_2\geq 0$ and the crude estimate $\frac{1}{2}g_{co,t} \leq \tilde{g}_t \leq 2g_{co,t}$ on $V_t(\frac{1}{2})$ for $t$ sufficiently small.

Thus
\begin{equation*}
   \begin{split}
		0\leq & 2\tilde{Z}_0 |t_n|^{\frac{1}{4}} \int_{V_n(\frac{1}{2})}   | h_n|^2 \mathbf{r}^{-\frac{4}{3}}\phi_2\tau^2  dV_{n}
			+ 2|t_n|^{\nu-\iota} 
			+ \int_{V_n(\frac{1}{2})} (6\iota^2-\iota) | h_n|^2 \mathbf{r}^{-\frac{4}{3}}\phi_2\tau^2  dV_{n} +C_7 \\
		\leq &(2\tilde{Z}_0 |t_n|^\frac{1}{4}+6\iota^2-\iota) \int_{V_n(\frac{1}{2})}   | h_n|^2 \mathbf{r}^{-2\iota-\frac{4}{3}}\tau^2  dV_{n}+ 2|t_n|^{\nu-\iota} +C_4.
   \end{split}
\end{equation*}
Recall that $0<\iota<\min\{\frac{1}{6},\nu\}$ is fixed. Let $n$ be large so that $2\tilde{Z}_0 |t_n|^{\frac{1}{4}}+6\iota^2-\iota<0$, we see from above that
\begin{equation}
 \label{eq:19-4}
			\int_{V_n(\frac{1}{2})}   | h_n|^2 \mathbf{r}^{-2\iota-\frac{4}{3}}\tau^2  dV_{n} \leq C(\iota)
\end{equation}
for some constant $C(\iota)>0$ depending on $\iota$.

From (\ref{eq:19-1}), (\ref{eq:r-4}), (\ref{eq:19-2}), and (\ref{eq:19-4}) the proof is complete. \qed
\end{proof}

\end{document}